\documentclass[12 pt,a4paper,oneside]{amsart}
\usepackage{a4wide}

\usepackage{amsfonts,amsmath,amssymb,amsthm}
\usepackage{mathrsfs}
\usepackage{mathtools}
\usepackage{esint}
\usepackage{graphicx}
\usepackage{subcaption}
\usepackage{tikz}
\usetikzlibrary{quotes,angles,positioning}
\usepackage{enumitem}

\usepackage[T1]{fontenc}
\usepackage{dsfont}
\usepackage{xcolor}
\usepackage{hyperref}
\hypersetup{hidelinks}

\usepackage{algorithm}
\usepackage{algorithmic}

\usepackage{calc}
\usepackage{ragged2e}
\usepackage{setspace}
\newcommand{\FundingLogos}{%
  \raisebox{0pt}{\includegraphics[height=1.5cm]{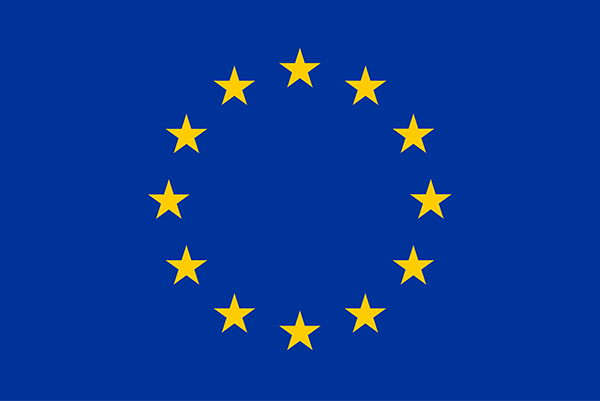}}%
  \hspace{1em}%
  \raisebox{0pt}{\includegraphics[height=1.5cm]{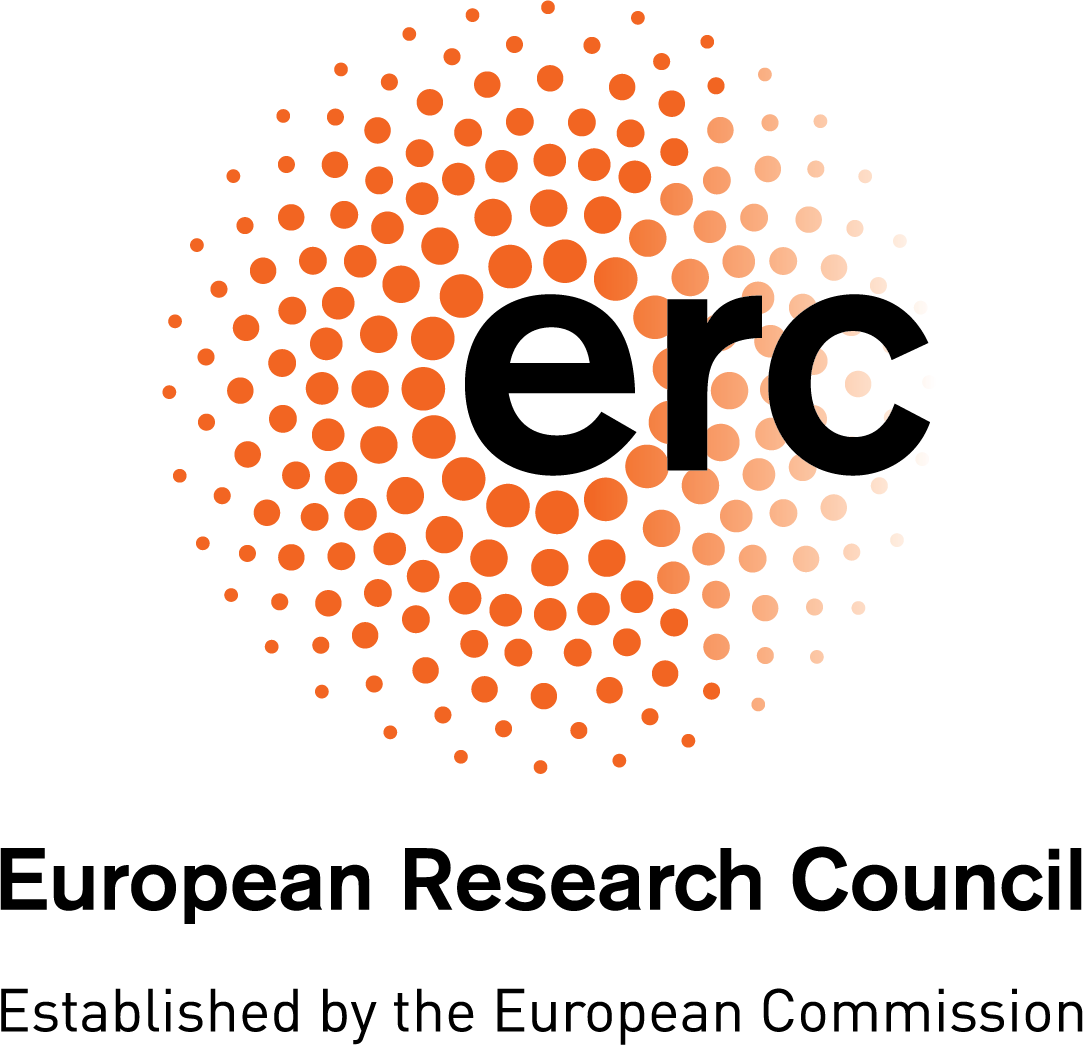}}%
  \hspace{1em}
  \raisebox{0pt}{\includegraphics[height=1.5cm]{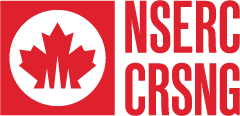}}%
}

\usepackage{tcolorbox}
\tcbset{colback=white,boxrule=0.5pt,arc=2pt,left=4pt,right=4pt,top=4pt,bottom=4pt}

\usepackage{accents}

\newcommand{\widebar}[1]{%
\accentset{\overline{\hphantom{#1}}}{#1}
}

\numberwithin{equation}{section}

\newtheorem{theorem}{Theorem}[section]
\newtheorem{lemma}[theorem]{Lemma}
\newtheorem{proposition}[theorem]{Proposition}
\newtheorem{corollary}[theorem]{Corollary}
\newtheorem{definition}[theorem]{Definition}
\newtheorem{remark}[theorem]{Remark}
\newtheorem{example}[theorem]{Example}
\newtheorem{assumption}[theorem]{Assumption}
\newtheorem{condition}[theorem]{Condition}

\newcommand{\R}{\mathbb{R}}

\newcommand{\N}{\mathbb{N}}

\newcommand{\eps}{\epsilon}

\DeclareMathOperator{\diag}{diag}
\DeclareMathOperator{\rank}{rank}

\DeclareMathOperator*{\argmax}{arg\,max}

\newcommand{\bZ}{\mathbb{Z}}
\newcommand{\bG}{\mathbb{G}}

\newcommand{\FGMart}{\mathrm{FG}_\mathrm{mgle}}

\newcommand{\bL}{\mathbf{L}}
\newcommand{\cL}{\mathcal{L}}
\newcommand{\cW}{\mathcal{W}}

\newcommand{\bE}{\mathbb{E}}

\newcommand{\Image}{\textup{Im}}

\DeclareMathOperator*{\tr}{\mathrm{tr}}
\DeclareMathOperator*{\argmin}{arg\,min}
\newcommand{\Cpl}{\mathrm{Cpl}}
\newcommand{\AW}{\mathcal{AW}}
\newcommand{\BW}{\mathrm{BW}}
\newcommand{\FP}{\mathrm{FP}}
\newcommand{\cFG}{\mathcal{FG}}
\newcommand{\FG}{\mathrm{FG}}
\newcommand{\cFP}{\mathcal{FP}}
\newcommand{\cB}{\mathcal{B}}
\newcommand{\cF}{\mathcal{F}}
\newcommand{\cN}{\mathcal{N}}

\newcommand{\bP}{\mathbb{P}}
\newcommand{\bF}{\mathbb{F}}
\newcommand{\bR}{\mathbb{R}}
\newcommand{\bX}{\mathbb{X}}
\newcommand{\bY}{\mathbb{Y}}
\newcommand{\E}{\mathbb{E}}
\newcommand{\beps}{\boldsymbol{\epsilon}}
\newcommand{\Cov}{\textup{Cov}}

\begin{document}

\title[Adapted Wasserstein Barycenters of Gaussian Processes]{Adapted Wasserstein Barycenters of Gaussian Processes}

\author{
Madhu Gunasingam
\address[Madhu Gunasingam]{Department of Statistical Sciences, University of Toronto}
\email{madhu.gunasingam@mail.utoronto.ca}
\hspace*{0.5cm}
Francesco Mattesini
\address[Francesco Mattesini]{Department of Mathematics, CIT School, Technical University of Munich \& Munich Center for Machine Learning (MCML)}
\email{francesco.mattesini@tum.de}
\\[1ex]
Johannes Wiesel
\address[Johannes Wiesel]{Department of Mathematics, University of Copenhagen}
\email{wiesel@math.ku.dk}
\hspace*{0.5cm}
Ting-Kam Leonard Wong
\address[Ting-Kam Leonard Wong]{Department of Statistical Sciences, University of Toronto}
\email{tkl.wong@utoronto.ca}
}

\begin{abstract}
We study barycenters of filtered Gaussian processes in adapted Wasserstein space. The adapted Wasserstein distance refines classical optimal transport by requiring transport plans to respect the temporal flow of information, making it the natural metric for stochastic systems with filtration constraints, as in stochastic control, mathematical finance, and sequential decision problems. We prove that the \emph{unrestricted} barycenter problem for weighted Fr\'echet means of filtered Gaussian inputs admits a solution with Gaussian underlying law, representable as an enlarged filtered Gaussian process but not necessarily as an ordinary one. The problem decomposes into finitely many classical Bures--Wasserstein barycenter problems for the covariance contributions of the successive innovations.

\noindent
We then treat the \emph{restricted} problem, in which the barycenter is required to be an ordinary filtered Gaussian process, giving a rank and common-noise criterion for when the two problems agree, sufficient conditions for uniqueness, and first order optimality and regularity results. Under a martingale constraint we obtain an explicit solution via martingale projection and Bures--Wasserstein barycenters of the Gaussian increments. Beyond their intrinsic theoretical interest, our results provide a principled way to build representative models from collections of Gaussian stochastic systems, with applications to stochastic optimization, robust finance, and sequential statistical analysis.
\end{abstract}

\date{July 10, 2026}
\maketitle

\section{Introduction and main results}

In recent years, optimal transport has become a central tool for comparing probability measures and for extracting representative distributions from heterogeneous data. One of the most popular ways to construct such representative distributions is the so-called \emph{Wasserstein barycenter}, which provides a notion of average measure that respects the underlying geometry of the Wasserstein space of distributions. Since the landmark work \cite{AC11} of Agueh and Carlier, Wasserstein barycenters have found applications in statistics, machine learning, imaging, economics, and uncertainty quantification. An incomplete list is given by  \cite{rabin2011wasserstein,cuturi2014fast, pass2015multi, peyre2016gromov, alvarez2016fixed, chewi2020gradient, brizzi2025p, brizzi2026q, Brizzi2026hWasserstein}; see also \cite{peyre2019computational, panaretos2020invitation, friesecke2024optimal, chewi2025statistical} and the references therein for an overview. 

\medskip
In many applications however, the objects of interest are not static distributions but stochastic processes or filtered probability laws, and in such settings the classical Wasserstein framework ignores a key structural feature: time and information. This limitation is particularly important when one compares laws on path space under adaptedness constraints. In stochastic control, mathematical finance, and sequential decision problems, couplings between random evolutions should respect the underlying filtrations, so that decisions at a given time can depend only on past and present information. The \emph{adapted Wasserstein distance} was introduced in \cite{aldous2022weak,hoover1984adapted, hellwig1996sequential, pflug2014multistage, backhoff2017causal, lassalle2018causal}  precisely to capture this temporal structure. Unlike the classical Wasserstein distance, it penalizes discrepancies between stochastic processes while enforcing a dynamic consistency condition on transport plans. As a result, it is the natural metric for problems involving nested information, causal transport, and stability of stochastic systems \cite{BBBE2020A, WasSpacStoPro}. In this vein, \emph{adapted barycenters} may be viewed as dynamic analogues of Fréchet means, but in a geometry shaped not only by state discrepancy but also by the admissible flow of information.

\medskip
Adapted barycenters, introduced in \cite{acciaio2025multicausal}, are also of practical interest. They provide a principled way to average stochastic models while preserving temporal information, and thus offer natural candidates for representative models in applications where the timing of information is essential. This is relevant, for example, in scenario reduction for multi-stage stochastic optimization, in model aggregation for time-series distributions, and in robust mathematical finance, where one seeks central objects among competing adapted models. 

\medskip 
In this paper, we provide the first systematic study of adapted Wasserstein barycenters of filtered Gaussian processes. We focus here on discrete time: given a family of Gaussian processes, we study the existence and structure of minimizers of the corresponding Fr\'echet functional on the adapted Wasserstein space of filtered processes. While such properties have been extensively studied for the
classical Wasserstein distance as mentioned above, the adapted setting is intrinsically time-directed, is sensitive to the filtration carried by the laws, and typically lacks a number of tools that are available in the standard optimal transport framework. In particular, the interaction between temporal constraints and convexity properties of the transport functional creates new analytical challenges.

\medskip

Our analysis builds on the geometry of the adapted Bures--Wasserstein space recently established in \cite{GunWon, acciaio2025entropic, ABGHP26, gunasingam2026adapted}. Our main results can be summarized as follows. We show in Theorem~\ref{thm:gaussian.multicausal.optimizer} that adapted $2$-Wasserstein barycenters of Gaussian processes exist by combining the multicausal reformulation of \cite{acciaio2025multicausal} with a finite-dimensional covariance argument. An important finding is that an adapted Wasserstein barycenter need \emph{not} be an adapted Bures--Wasserstein barycenter.
In fact, as was already remarked in \cite{acciaio2025multicausal}, an adapted Wasserstein barycenter need not be naturally filtered, and Section~\ref{sec:counter} gives an explicit example. This observation is essential. There always exists an adapted barycenter with Gaussian underlying law, and one may select a representative driven by $Kd$ independent Gaussian noise coordinates at each time, where $K$ is the number of input processes. Such a representative need not lie in the class of ordinary filtered Gaussian processes, and other, non-Gaussian barycenters may coexist when uniqueness fails.

\medskip

Motivated by this observation, we characterize when this \emph{unrestricted} adapted Wasserstein problem and the \emph{restricted} adapted Wasserstein problem (where we restrict the adapted Wasserstein barycenter to be an ordinary filtered Gaussian process) agree. The characterization is a rank condition on the one-step block correlation matrices, which we call the \emph{common-noise criterion}. We investigate both problems in detail: we give sufficient conditions for uniqueness, derive a necessary first-order condition in the form of a fixed-point equation, and establish regularity of restricted minimizers. We also show that both problems decompose into classical Bures--Wasserstein problems. This yields a natural correspondence between adapted Bures--Wasserstein geometry and a product of standard Bures--Wasserstein spaces. Finally, the first-order characterization leads to computational schemes for the restricted and unrestricted problems, in the spirit of \cite{alvarez2016fixed, puccetti2020computation}.


\subsection{Statement of the main results}\label{sec:statement}
Throughout, $T\in \N = \{1, 2, \ldots\}$ is the time horizon, $d\in \N$ is the spatial dimension, and $K\in \N$ is the number of input processes. We use $s,t,u$ for time indices, $i,j$ for spatial coordinates, and $k,\ell$ for marginal processes. We refer to equation \eqref{eq:AW2} below for the formal definition of the adapted Wasserstein distance $\mathcal{AW}_2$ as well as the quotient space of filtered process $\FP_2$, and to equation \eqref{eqn:filtered.Gaussian} for the definition of the set of filtered Gaussian processes $\cFG(T,d)$. Given filtered Gaussian inputs
\[
\bX^k=\bG^{a^k,L^k}\in\cFG(T,d),\qquad k=1,\ldots,K,
\]
with weights $\lambda_k>0$ satisfying $\sum_k\lambda_k=1$, we study the adapted Wasserstein barycenter problem
\begin{equation}\label{eq:AWbar}
\inf_{\bY\in\FP_2}\sum_{k=1}^K\lambda_k\AW_2^2(\bX^k,\bY).
\end{equation}
As mentioned in the Introduction we call \eqref{eq:AWbar} the unrestricted problem, in contrast to the restricted problem
\begin{equation}\label{eq:AWbar2}
\inf_{\bY\in\cFG(T,d)}\sum_{k=1}^K\lambda_k\AW_2^2(\bX^k,\bY).
\end{equation}

Our main results can be summarized as follows:
\begin{enumerate}[label=(\roman*)]
\item We show in Theorem~\ref{thm:gaussian.multicausal.optimizer} that the unrestricted problem \eqref{eq:AWbar} admits a solution with underlying Gaussian law. It has a representative driven by $Kd$ Gaussian innovation coordinates, and its value separates over time into $T$ compact finite-dimensional optimization problems. We often call these \emph{local} optimization problems.

\item Theorem~\ref{thm:explicit.value} identifies these local problems as classical Bures--Wasserstein barycenter problems for the covariance contributions of the successive innovations. Uniqueness need not hold, even for regular input factors, but Theorem~\ref{thm:uniqueness} shows that Assumption~\ref{cond:uniqueness} provides a sufficient condition for uniqueness in $\FP_2$.

\item An unrestricted barycenter need not have an ordinary $d$-dimensional filtered Gaussian representative. Theorem~\ref{thm:FG-common-noise-criterion} characterizes exactly when such a representative exists, via the common-noise criterion \eqref{eq:common-noise.input.condition}.

\item Theorem~\ref{thm:restricted-FG-barycenters} establishes that the restricted problem \eqref{eq:AWbar2} always admits a minimizer. It is obtained from the same local optimization problems as in (i), subject to an additional rank-$d$ constraint. Proposition~\ref{thm:characterization} shows that restricted minimizers satisfy a Procrustes fixed-point condition, while Theorem~\ref{thm:regularity} proves that they inherit regularity from regular inputs.

\end{enumerate}

\subsection{Open questions} We conclude this section with open questions that arise in view of our
results:
\begin{enumerate}
\item \emph{Non-Gaussian barycenters.} Our results establish existence, a structural characterization, and sufficient conditions for uniqueness of adapted Wasserstein barycenters for Gaussian processes, relying crucially on the explicit reduction to finite-dimensional adapted Bures--Wasserstein problems. It is natural to ask whether analogous results hold for broader classes of processes. The multicausal reformulation of Theorem~\ref{thm:gaussian.multicausal.optimizer} does not require the underlying distribution to be Gaussian, and the existence of adapted barycenters in $\mathrm{FP}_2$ is guaranteed by \cite[Theorem~4.5]{acciaio2025multicausal} under general assumptions. However, uniqueness and a structural characterization beyond the Gaussian case remain open. Natural candidates for further investigation would be Markov processes, where the conditional structure is sufficiently tractable to exploit multicausal dynamic programming, yet rich enough to go beyond the Gaussian setting.
\item \emph{Continuous-time barycenters.} All results in this paper are formulated in discrete time. Extending the theory to continuous-time processes, such as diffusions driven by Brownian motion, is of significant interest for applications in mathematical finance and stochastic control. The continuous-time adapted Wasserstein distance has been studied in \cite{bartl2025wasserstein}, but the barycenter problem in that setting is open, to the best of our knowledge. Already in the Gaussian case, the question of whether the adapted barycenter of Ornstein--Uhlenbeck processes is itself a diffusion, and how the fixed-point equation \eqref{eq:FPE} extends to a continuous-time analogue involving operator-valued equations, appears nontrivial to us.
%

%
\item \emph{Entropic adapted barycenters.} The recent work \cite{acciaio2025entropic} studies entropic regularization of the adapted Wasserstein distance between Gaussian processes. A natural extension of our results is to define and study \emph{entropic adapted barycenters}, i.e.\ Fr\'echet means with respect to the entropic adapted Wasserstein distance. It would be interesting to investigate whether these regularized barycenters admit closed-form characterizations analogous to the fixed-point equation \eqref{eq:FPE}, whether uniqueness holds under weaker assumptions due to the strict convexity introduced by the entropic penalty, and how the entropic barycenter converges to the adapted Wasserstein barycenter as the regularization parameter vanishes.
\item \emph{Robust stress testing of pricing models.} In financial risk management, stress testing requires constructing plausible adverse scenarios that are consistent with a collection of reference pricing models. The adapted Wasserstein barycenter provides a natural candidate for a central scenario that respects the temporal flow of information inherent in financial models. More precisely, given a set of competing stochastic models for asset prices or risk factors, the adapted barycenter aggregates them into a single representative model while preserving the filtration structure, which is essential for hedging and dynamic trading strategies. It would be interesting to investigate how adapted barycenters can be used to define stress scenarios as perturbations of the central model within an adapted Wasserstein ball, extending the distributionally robust approach of \cite{BBBE2020A, bartl2023sensitivity} to the multi-model setting. The Gaussian framework developed in this paper offers a computationally tractable starting point for such applications, and the fixed-point iteration of Algorithm~\ref{alg:ABW} provides a concrete tool for computing these central scenarios in practice.

\end{enumerate}

\subsection{Structure of the paper}
The remainder of this paper is organized as follows. 
Section~\ref{sec:notation} introduces the filtered process, multicausal, and Gaussian notation used throughout the paper. Section~\ref{sec:unrestricted-gaussian-barycenters} studies the unrestricted adapted Wasserstein barycenter problem, including the Gaussian multicausal construction and the enlarged-noise representation of its minimizers. In Section~\ref{sec:FG-common-noise} we establish the common-noise criterion characterizing when an unrestricted barycenter admits an ordinary filtered Gaussian representative, and analyze the corresponding restricted barycenter problem. Section~\ref{sec:computation} develops computational methods for the restricted and unrestricted problems, including fixed-point iterations, semidefinite formulations, and optimality certificates. Section~\ref{sec:mgle-barycenters} treats barycenters under a martingale constraint and gives an explicit solution in terms of martingale projections and classical Bures--Wasserstein barycenters of the Gaussian increments. 
In Section~\ref{sec:OU} we apply the theory to time-varying $AR(1)$ processes, showing in particular that under nonnegative autoregressive coefficients the adapted barycenter reduces to the Knothe--Rosenblatt barycenter, which is generically not itself $AR(1)$. Finally, in Section~\ref{sec:numerics} we illustrate these findings numerically, comparing the adapted and classical Bures--Wasserstein barycenters through their second-order structure.

\section{Notation and preliminary results}
\label{sec:notation}
\label{sec:prelim}
In this section we collect the notation and background material used throughout the paper. We write $[T]=\{1,\dots,T\}$ and identify $\R^{Td}$ with $(\R^d)^T$, so that $x\in\R^{Td}$ is written as $x=(x_1,\dots,x_T)$ with $x_t\in\R^d$. Vectors are regarded as columns. We denote by $\mathscr{S}_+(n)$ (resp. $\mathscr{S}_{++}(n)$) the set of symmetric positive semidefinite (resp. positive definite) $n\times n$ matrices, and by $\mathscr{O}(n) \subset \bR^{n \times n}$ the orthogonal group. For a matrix $A$, we write $\|A\|_{2\to 2}$ for its operator norm, $\|A\|_\mathrm{F}$ for its Frobenius norm, $\|A\|_*$ for its nuclear norm (the sum of its singular values), and $\mathrm{Im}(A)$ for its column space. The barycenter weights are denoted by $\lambda_k>0$, with $\sum_{k=1}^K\lambda_k=1$. Boldface symbols such as $\boldsymbol{\epsilon}$, $\mathbf{L}$, and $\mathbf{P}$ are reserved for stacked or ensemble objects built from the $K$ marginal processes, while non-bold symbols such as $L$ and $\epsilon$ refer to single-marginal quantities. We use a wide bar for product-space or multicausal objects and an ordinary overline for barycentric means and factors. The law of a random element $X$ is denoted by $\mathcal{L}(X)$, the $n$-dimensional Gaussian distribution with mean $a$ and covariance $A$ by $\mathcal{N}_n(a,A)$ (the subscript is omitted when the dimension is clear), and the classical $2$-Wasserstein distance between probability measures with finite second moment by $\mathcal{W}_2$.

\subsection{Adapted Wasserstein distance}
\label{sec:AOT}
We recall the definitions of filtered process and adapted Wasserstein distance following the framework of \cite{WasSpacStoPro, acciaio2025multicausal}. A {\it filtered process} is a five-tuple
\[
\bX=\left(\Omega^{\bX},\cF^{\bX},\bP^{\bX},\bF^{\bX}=(\cF_t^{\bX})_{t=1}^T,X=(X_t)_{t=1}^T\right),
\]
where \((\Omega^{\bX},\cF^{\bX},\bP^{\bX})\) is a probability space, \(\bF^{\bX}\) is a filtration, and \(X\) is an \(\bR^d\)-valued process adapted to \(\bF^{\bX}\). The family of filtered processes $\bX$ with finite second moment (that is, $\bE_{\bP^{\bX}}[\|X\|_2^2] < \infty$) is denoted by \(\cFP_2\).

\medskip
Given filtered processes \(\bX, \bY \in \cFP_2\), we let $\Cpl(\bX, \bY)$ be the set of couplings on the product space $\Omega^{\bX} \times \Omega^{\bY}$ with marginals $\bP^{\bX}$ and $\bP^{\bY}$. A coupling $\pi \in \mathrm{Cpl}(\mathbb{X}, \mathbb{Y})$ is called \emph{bicausal} if, for each $t = 1, \ldots, T$, the $\sigma$-algebra $\mathcal{F}^{\mathbb{X}}_T$ is conditionally independent of $\mathcal{F}^{\mathbb{Y}}_t$ given $\mathcal{F}^{\mathbb{X}}_t$ under $\pi$, and symmetrically $\mathcal{F}^{\mathbb{Y}}_T$ is conditionally independent of $\mathcal{F}^{\mathbb{X}}_t$ given $\mathcal{F}^{\mathbb{Y}}_t$ under $\pi$; see \cite[Definition~2.1]{WasSpacStoPro}. This is equivalent to the multicausal condition, recalled below, for two marginals. We write $\mathrm{Cpl}_{\mathrm{bc}}(\mathbb{X}, \mathbb{Y})$ for the set of bicausal couplings. The \emph{adapted $2$-Wasserstein distance} between $\bX, \bY \in \cFP_2$ is defined as
\begin{equation} \label{eq:AW2}
\AW_2^2(\bX,\bY):=\inf_{\pi\in\Cpl_{\mathrm{bc}}(\bX,\bY)}\E_\pi\left[\|X-Y\|_2^2\right].
\end{equation}
We let $\FP_2$ be the quotient space $\cFP_2 / \sim$, where two filtered processes $\bX, \bY$ are identified if $\AW_2(\bX, \bY) = 0$. We often identify a filtered process $\bX$ with its equivalence class $[\bX]$. We call $(\FP_2,\AW_2)$ the \emph{Wasserstein space of (discrete-time) stochastic processes}; it is a Polish space and is isometric to a Wasserstein space over nested conditional laws.

\medskip
As shown in \cite{GunWon, acciaio2025entropic, ABGHP26, gunasingam2026adapted}, in the Gaussian setting the adapted Wasserstein distance admits a closed-form representation. Let $\mathscr{L}(T, d) \subset \bR^{Td \times Td}$ be the set of block lower-triangular matrices of the form
\[
L = \begin{pmatrix} L_{1,1} & 0 & \cdots & 0 \\ L_{2,1} & L_{2,2} & \cdots & 0 \\ \vdots & \vdots & \ddots & \vdots \\ L_{T,1} & L_{T,2} & \cdots & L_{T,T} \end{pmatrix}, \quad L_{t,s} \in \bR^{d \times d}.
\]
A \emph{filtered Gaussian process} is a filtered process $\bG^{a, L}$, parameterized by the mean $a \in \bR^{Td}$ and the Cholesky factor $L \in \mathscr{L}(T, d)$, given by
\begin{equation} \label{eqn:filtered.Gaussian}
\bG^{a, L} = ( \bR^{Td}, \cB(\bR^{Td}), \cN_{Td}(0, I), \bF^{\epsilon}, X = a + L\epsilon),
\end{equation}
where $\epsilon=(\epsilon_t)_{t=1}^T$ is the canonical standard Gaussian process and $\bF^{\epsilon}$ is the filtration generated by $\epsilon$. Explicitly,
\[
X_t=a_t+\sum_{s=1}^tL_{t,s}\epsilon_s.
\]
The filtration generated by $X$ may be strictly smaller than $\bF^{\epsilon}$ unless each $L_{t,t}$ is invertible. In the centered case $a = 0$, we write $\bG^{L} := \bG^{0, L}$. We let $\cFG(T,d)\subset\cFP_2$ be the set of filtered Gaussian processes. If $\bX^i=\bG^{a^i,L^i}\in\cFG(T,d)$, $i=1,2$, then
\begin{equation} \label{eq:AW2.Gaussian}
\AW_2^2(\bX^1, \bX^2) = \|a^1 - a^2\|_2^2 + \mathrm{dist}_{\mathrm{ABW}}^2(L^1, L^2),
\end{equation}
where $\mathrm{dist}_{\mathrm{ABW}}$ is the \emph{adapted Bures--Wasserstein (pseudo-)distance} on $\mathscr{L}(T, d)$ given by
\begin{equation}
\mathrm{dist}_{\mathrm{ABW}}^2(L^1, L^2) := \|L^1\|_{\mathrm{F}}^2 + \|L^2\|_{\mathrm{F}}^2 - 2 \sum_{t = 1}^T \| ((L^1)^{\intercal} (L^2))_{t,t}\|_*,\label{eq:ABW}
\end{equation}
or its variational form,
\begin{equation}
\mathrm{dist}_\mathrm{ABW}(L^1,L^2)= \min_{Q \in \mathscr{O}(T, d)} \| L^1 - L^2 Q \|_{\mathrm{F}}. \label{eq:Procrustes}
\end{equation}
Here, 
$\mathscr{O}(T, d) := \{ \diag(Q_1, \ldots, Q_T): Q_t \in \mathscr{O}(d)\}$ is the set of block diagonal orthogonal matrices. Both expressions will be utilized throughout this paper, and \cite{gunasingam2026adapted} shows that $\mathrm{dist}_{\mathrm{ABW}}(L^1, L^2) = 0$ if and only if $L^1 = L^2Q$ for some $Q \in \mathscr{O}(T, d)$. Under this equivalence relation $\mathrm{dist}_{\mathrm{ABW}}$ becomes a genuine distance. We denote by $\mathscr Q(L^1,L^2):=\argmin_{Q\in\mathscr O(T,d)}\|L^1-L^2Q\|_F^2$ the nonempty set of optimizers of the constrained Procrustes problem \eqref{eq:Procrustes}. Finally, from \cite{ABGHP26} we say that $L\in\mathscr{L}(T,d)$ is \emph{regular}, written $L\in\mathscr{L}^\mathrm{reg}(T,d)$, if for every $t\in[T]$ the block-column 
\[
{L}_{\cdot,t} := \begin{pmatrix} L_{1,t}  \\ \vdots \\ L_{T,t} \end{pmatrix} \in \mathbb{R}^{Td \times d},
\]
\medskip
has full column rank $d$, equivalently $(L^\intercal L)_{t,t}=L_{\cdot,t}^\intercal L_{\cdot,t}\in \mathscr{S}_{++}(d)$. Probabilistically, $L_{\cdot, t}$ records the contribution of the innovation $\eps_t$ to the path $(X_s)_{s=t}^T$. (Note that $L_{t,s}=0$ for $s>t$).
\medskip

We also observe that $\mathrm{dist}_{\mathrm{ABW}}$ can be expressed in terms of the usual \emph{Bures--Wasserstein distance} $\mathrm{dist}_{\mathrm{BW}}$, which is defined for covariance matrices $A, B \in \R^{n \times n}$ by
\begin{align*}
\mathrm{dist}_{\mathrm{BW}}^2(A, B) &:= \tr A + \tr B - 2 \tr \left( (A^{1/2}BA^{1/2})^{1/2}  \right) = \cW_2^2 ( \cN(0, A), \cN(0, B)).
\end{align*}
Indeed, by \cite[Proposition 2.9]{gunasingam2026adapted}, if $A = LL^{\intercal}$ and $B = MM^{\intercal}$, then
\begin{equation} \label{eqn:d.BW}
\mathrm{dist}_{\mathrm{BW}}^2(A, B) = \|L\|_{\mathrm{F}}^2 + \|M\|_{\mathrm{F}}^2 - 2 \|L^{\intercal}M\|_*.
\end{equation}
Define the \emph{column covariance}
\[
A_t^L:=L_{\cdot,t}L_{\cdot,t}^\intercal\in \R^{Td\times Td}.
\]
Comparing \eqref{eqn:d.BW} and \eqref{eq:ABW} for each $t$ leads to the following identity.

\begin{proposition} \label{prop:decomp}
For $L, M \in \mathscr{L}(T, d)$, we have 
\begin{equation}\label{eq:decomp}
\mathrm{dist}_{\mathrm{ABW}}^2(L, M) = \sum_{t=1}^T \mathrm{dist}_{\mathrm{BW}}^2\bigl(A_t^L,\, A_t^M\bigr).
\end{equation}

\begin{proof}
Fix $L,M\in\mathscr{L}(T,d)$. Applying \eqref{eqn:d.BW} to the column covariances $A_t^L=L_{\cdot,t}L_{\cdot,t}^{\intercal}$ and $A_t^M=M_{\cdot,t}M_{\cdot,t}^{\intercal}$ yields
\[
\mathrm{dist}_{\mathrm{BW}}^2\bigl(A_t^L, A_t^M\bigr) = \|L_{\cdot,t}\|_{\mathrm{F}}^2 + \|M_{\cdot,t}\|_{\mathrm{F}}^2 - 2 \bigl\| L_{\cdot,t}^{\intercal} M_{\cdot,t} \bigr\|_*.
\]
Summing over $t \in [T]$ gives
\begin{equation} \label{eq:decomp.sum}
\sum_{t=1}^T \mathrm{dist}_{\mathrm{BW}}^2\bigl(A_t^L, A_t^M\bigr) = \sum_{t=1}^T \|L_{\cdot,t}\|_{\mathrm{F}}^2 + \sum_{t=1}^T \|M_{\cdot,t}\|_{\mathrm{F}}^2 - 2 \sum_{t=1}^T \bigl\| L_{\cdot,t}^{\intercal} M_{\cdot,t} \bigr\|_*.
\end{equation}
It remains to identify the three sums on the right-hand side of \eqref{eq:decomp.sum} with the corresponding terms in \eqref{eq:ABW}. Since the block-columns $L_{\cdot,1}, \ldots, L_{\cdot,T}$ collect all the blocks of $L$, we have
\[
\sum_{t=1}^T \|L_{\cdot,t}\|_{\mathrm{F}}^2 = \sum_{t=1}^T \sum_{s=1}^T \|L_{t,s}\|_{\mathrm{F}}^2 = \|L\|_{\mathrm{F}}^2,
\]
and likewise
\[
\sum_{t=1}^T \|M_{\cdot,t}\|_{\mathrm{F}}^2 = \|M\|_{\mathrm{F}}^2.
\]
For the last sum, the $d\times d$ diagonal block of $L^{\intercal}M$ indexed by $t$ is
\[
(L^{\intercal} M)_{t,t}
=L_{\cdot,t}^{\intercal} M_{\cdot,t}
=\sum_{s=1}^T L_{s,t}^{\intercal} M_{s,t}.
\]
Hence
\[
\sum_{t=1}^T \bigl\| L_{\cdot,t}^{\intercal} M_{\cdot,t} \bigr\|_* = \sum_{t=1}^T \bigl\| (L^{\intercal} M)_{t,t} \bigr\|_*.
\]
Substituting these three identities into \eqref{eq:decomp.sum} gives
\[
\sum_{t=1}^T \mathrm{dist}_{\mathrm{BW}}^2\bigl(A_t^L, A_t^M\bigr) = \|L\|_{\mathrm{F}}^2 + \|M\|_{\mathrm{F}}^2 - 2 \sum_{t=1}^T \bigl\| (L^{\intercal} M)_{t,t} \bigr\|_* = \mathrm{dist}_{\mathrm{ABW}}^2(L, M),
\]
where the last equality is the definition \eqref{eq:ABW} of $\mathrm{dist}_{\mathrm{ABW}}$. This proves \eqref{eq:decomp}.
\end{proof}
\end{proposition}

\subsection{Multicausal transport and barycenters}\label{sec:multicausal}

A key ingredient is the multicausal reformulation of the barycenter problem. We briefly recall the necessary notions; see \cite{acciaio2025multicausal} for a comprehensive treatment.

\medskip
Given $K$ filtered processes $\mathbb{X}^1, \ldots, \mathbb{X}^K \in \mathrm{FP}_2$, we write $\mathrm{Cpl}(\mathbb{X}^1, \ldots, \mathbb{X}^K)$ for the set of couplings on the product space $(\widebar{\Omega},  \widebar{\cF}) := (\prod_{k=1}^K \Omega^k, \bigotimes_{k=1}^K \mathcal{F}^k)$ with marginals $\mathbb{P}^k$, $k = 1, \ldots, K$, and we denote by $\widebar{\mathcal{F}}_t := \bigotimes_{k=1}^K \mathcal{F}^k_t$ the product filtration $\widebar{\bF}$ at time $t$. We assume throughout that $\Omega^k = \Omega^k_{1:T} := \Omega^k_1 \times \cdots \times \Omega^k_T$, where each $\Omega_s^k$ is Polish and is equipped with its Borel $\sigma$-field. Thus sample points $\omega^k \in \Omega^k$ are identified with $\omega^k_{1:T} = (\omega^k_1, \ldots, \omega^k_T)$, regular conditional distributions exist, and $\mathcal{F}_t^k=\bigotimes_{s=1}^t \mathcal{B}(\Omega_s^k) \otimes \bigotimes_{s=t+1}^T \{\emptyset, \Omega_s^k\}$. It is helpful to think of $\omega_{1:t}^k$ as the driving noise of $X^k$ up to time $t$. A coupling $\pi \in \mathrm{Cpl}(\mathbb{X}^1, \ldots, \mathbb{X}^K)$ is called \emph{multicausal} if, for each $k = 1, \ldots, K$ and each $t = 1, \ldots, T$, the $\sigma$-algebra $\mathcal{F}^k_T$ is conditionally independent of $\widebar{\mathcal{F}}_t$ given $\mathcal{F}^k_t$ under $\pi$; see \cite[Definition~2.4]{acciaio2025multicausal}. We denote the set of multicausal couplings by $\mathrm{Cpl}_{\mathrm{mc}}(\mathbb{X}^1, \ldots, \mathbb{X}^K)$. For $K = 2$, multicausality reduces to bicausality. We write $\mathbb{P}^k_{t+1, \omega^k_{1:t}}$ for the regular conditional distribution of $\omega^k_{t+1}$ given $\omega^k_{1:t} := (\omega^k_1, \ldots, \omega^k_t)$ under $\mathbb{P}^k$, and $\widebar{\omega}_{1:t} := (\omega^1_{1:t}, \ldots, \omega^K_{1:t})$
. 

\medskip
By \cite[Example~4.23]{acciaio2025multicausal}, the barycenter problem \eqref{eq:AWbar} with quadratic cost is equivalent to the \emph{multicausal optimal transport problem}
\begin{equation}\label{eq:mc}
\inf_{\pi \in \mathrm{Cpl}_{\mathrm{mc}}(\mathbb{X}^1, \ldots, \mathbb{X}^K)} \mathbb{E}_\pi\left[\sum_{k=1}^K \lambda_k \left\|X^k - \phi^0(X^1, \ldots, X^K)\right\|_2^2\right],
\end{equation}
where $\phi^0(x^1, \ldots, x^K) := \sum_{k=1}^K \lambda_k x^k$ is the optimal aggregation map. It is standard to show that \eqref{eq:mc} has an optimal solution. By \cite[Theorem~4.19]{acciaio2025multicausal}, any optimizer $\pi^*$ yields a barycenter 
\begin{equation} \label{eqn:mc.to.barycenter}
\widebar{\mathbb{X}} = (\widebar{\Omega}, \widebar{\cF}, \pi^*, \widebar{\bF}, \widebar{X} := \phi^0(X^1, \ldots, X^K)).
\end{equation}
Conversely, gluing optimal bicausal couplings between the marginals $\bX^k$ and any fixed $\bY\in\FP_2$ along $\bY$ produces a multicausal coupling of $\bX^1,\dots,\bX^K$; see \cite{acciaio2025multicausal}. Moreover, by \cite[Theorem~3.2]{acciaio2025multicausal}, the multicausal problem satisfies a dynamic programming principle: define the value functions backward in time by
\begin{equation}\label{eq:DPP}
V(t, \widebar{\omega}_{1:t}) = \inf\left\{ \int V(t+1, \widebar{\omega}_{1:t}, \widebar{\omega}_{t+1})\, \pi(d\widebar{\omega}_{t+1}) \;\middle|\; \pi \in \mathrm{Cpl}(\mathbb{P}^1_{t+1, \omega^1_{1:t}}, \ldots, \mathbb{P}^K_{t+1, \omega^K_{1:t}}) \right\},
\end{equation}
with terminal condition
\[
V(T,\widebar\omega)=\sum_{k=1}^K\lambda_k\bigl\|X^k(\omega^k)-\phi^0(X^1(\omega^1),\ldots,X^K(\omega^K))\bigr\|_2^2.
\]
Solving \eqref{eq:DPP} yields an optimal kernel at each time step, and the optimal multicausal coupling is then constructed by composing the resulting kernels.

\begin{remark}
It is straightforward to show that additive constants in $X^i$ can be factored out under the quadratic cost. In particular, if the mean of $X^k$ is $a^k$, then for any barycenter $\widebar{X}$, the mean of $\widebar{X}$ is $\widebar{a} := \sum_{k = 1}^K \lambda_k a^k$. For convenience, we often assume without loss of generality that  all marginal processes have been centered.
\end{remark}

\subsection{Gaussian multicausal couplings}
We continue to use the notations in Section \ref{sec:multicausal}. Let $K$ filtered Gaussian processes $\bX^k = \bG^{a^k, L^k} \in \cFG(T, d)$, $k \in [K]$, be given. It will be proved in Section \ref{sec:unrestricted-gaussian-barycenters} that for any weights $\lambda_1, \ldots, \lambda_K$, the multi-marginal problem \eqref{eq:mc} admits an optimal multicausal coupling $\pi \in \Cpl_{\mathrm{mc}}(\bX^1, \ldots, \bX^K)$ which is jointly Gaussian. This leads, via \eqref{eqn:mc.to.barycenter}, to an $\AW_2$-barycenter $\widebar{\bX}$. To prepare for this analysis, we characterize the Gaussian elements of $\Cpl_{\mathrm{mc}}(\bX^1, \ldots, \bX^K)$, and introduce a space $\cFG_K(T, d)$ of \emph{ensemble filtered Gaussian processes} such that $\widebar{X}$ is $\AW_2$-equivalent to a member of $\cFG_K(T, d)$. Along the way, we establish notations that will be used throughout the paper.

\medskip
Since all filtered Gaussian processes share the same underlying sample space $\bR^{Td}$, we have $\widebar{\Omega} = (\bR^{Td})^K$. Let $\epsilon^k = (\epsilon_t^k)_{t = 1}^T$, $k \in [K]$, be the $k$-th noise process, so that $\epsilon_t^k(\widebar{\omega}) = \omega_t^k$ and $X^k = a^k + L^k \epsilon^k$. Let $\pi$ be a coupling of $\bX^1, \ldots, \bX^K$. By definition, it is a probability measure on $\widebar{\Omega}$ such that each of the $k$-th marginals is $\cN_{Td}(0, I)$, and represents the joint distribution of $(\epsilon^1, \ldots, \epsilon^K)$. We let
\[
\widebar{\boldsymbol{\epsilon}}_t := \begin{pmatrix} \epsilon_t^1 \\ \vdots \\ \epsilon_t^K \end{pmatrix} \in \bR^{Kd} \quad \text{and} \quad
\widebar{\boldsymbol{\epsilon}} := \begin{pmatrix} \widebar{\boldsymbol{\epsilon}}_1 \\ \vdots \\ \widebar{\boldsymbol{\epsilon}}_T \end{pmatrix}
\in\bR^{KTd}
\]
be the noises stacked first marginally, then in time. By an abuse of notation, we may regard $\pi$ as the law of $\widebar{\boldsymbol{\epsilon}}$. If $\pi$ is jointly Gaussian, it is described by its covariance matrix $\mathbf{P} := \bE_{\pi}[\widebar{\boldsymbol{\epsilon}}\widebar{\boldsymbol{\epsilon}}^{\intercal}]$, so that $\widebar{\boldsymbol{\epsilon}} \sim \pi = \cN_{KTd}(\mathbf{0}, \mathbf{P})$. Decompose $\mathbf{P}$ into $T \times T$ blocks $
\mathbf{P}_{s,t} \in \bR^{Kd \times Kd}$, each of which has $K \times K$ sub-blocks $\mathbf{P}_{s,t}^{k, \ell} \in \bR^{d \times d}$, and note that the marginal condition $\epsilon^k \sim \cN_{Td}(0, I)$ is equivalent to
\begin{equation} \label{eqn:P.block.condition}
\mathbf{P}_{s,t}^{k,k} = \delta_{s,t} I_d,
\end{equation}
where $\delta_{s,t}$ is the Kronecker delta. The following result characterizes multicausality of $\pi$ via $\mathbf{P}$; it is a direct generalization of \cite[Theorem 2.2]{acciaio2025entropic}.

\begin{lemma}[Characterization of Gaussian multicausal couplings]
\label{thm:gaussian.mc.characterization}
Let $\bX^1, \ldots, \bX^K \in \cFG(T, d)$ be filtered Gaussian processes. 
Let \(\pi = \cL(\widebar{\boldsymbol{\epsilon}}) = \cN_{KTd}(\mathbf{0}, \mathbf{P})\) be a jointly Gaussian coupling of $\bX^1,\ldots,\bX^K$, where $\mathbf{P}$ satisfies \eqref{eqn:P.block.condition}. The following are equivalent:
\begin{enumerate}
\item[(i)] \(\pi\) is multicausal.
\item[(ii)] $\mathbf{P} = (\mathbf{P}_{s,t})_{s,t = 1}^T$ is block diagonal. Equivalently, the stacked noises $\widebar{\boldsymbol{\epsilon}}_1, \ldots, \widebar{\boldsymbol{\epsilon}}_T$ are jointly independent under $\pi$.
\end{enumerate}
\end{lemma}
\begin{proof}
Suppose first that \(\pi\) is multicausal. For \(t>s\), multicausality of the
\(k\)-th coordinate at time \(s\), together with the white-noise marginal law, gives
\[
\E_\pi[\epsilon_t^k\mid\widebar{\mathcal F}_s]
=\E_\pi[\epsilon_t^k\mid\mathcal F_s^k]=0.
\]
Since \(\epsilon_s^\ell\) is \(\widebar{\mathcal F}_s\)-measurable, it follows that
\(\E_\pi[\epsilon_t^k(\epsilon_s^\ell)^\intercal]=0\) for every \(k,\ell\).
The case \(s>t\) follows by transposition. Thus all cross-time blocks of
\(\mathbf P\) vanish.

Conversely, if \(\mathbf P\) is block diagonal, joint Gaussianity implies that the
vectors \(\widebar{\boldsymbol\epsilon}_1,\ldots,
\widebar{\boldsymbol\epsilon}_T\) are independent. Hence the future noises of each
coordinate after time \(s\) are independent of \(\widebar{\mathcal F}_s\). Conditional
on \(\mathcal F_s^k\), this is exactly the multicausal conditional-independence
property.
\end{proof}

By Lemma~\ref{thm:gaussian.mc.characterization}, a Gaussian multicausal coupling of $\bX^1,\ldots,\bX^K$ has the form $\pi=\cL(\widebar{\boldsymbol{\epsilon}})=\cN_{KTd}(\mathbf{0},\mathbf{P})$, where
\[
\mathbf{P} = \diag( \mathbf{P}_1, \ldots, \mathbf{P}_T)
\]
is block diagonal and each one-step block $\mathbf{P}_t \in \bR^{Kd \times Kd}$ is an element of the set of \emph{block correlation matrices} 
\begin{equation}
\label{eq:Pkdef}
\mathscr{P}_+(K,d):=\{\mathbf{C} = (\mathbf{C}^{k,\ell})_{k,\ell = 1}^K \in\mathscr S_+(Kd):\mathbf{C}^{k,k}=I_d\textup{ for }k\in[K]\},
\end{equation}
where $\mathbf{P}_t^{k,\ell}\in\R^{d\times d}$ denotes the $(k,\ell)$-th block; the block structure ($K \times K$ blocks of size $d \times d$) will always be clear from the context. For $d=1$, $\mathscr{P}_+(K):=\mathscr{P}_+(K,1)$ is simply the set of $K\times K$ correlation matrices. Conversely, any choice of one-step blocks in $\mathscr{P}_+(K,d)$ defines a Gaussian multicausal coupling. We accordingly write
\[
\mathscr{P}_+(K,T,d):=\{\mathbf{P}=\diag(\mathbf{P}_1,\dots,\mathbf{P}_T): \mathbf{P}_1,\dots,\mathbf{P}_T\in\mathscr{P}_+(K,d)\}
\]
for the collection of covariance matrices, corresponding to Gaussian multicausal couplings.

As noted in \cite{acciaio2025multicausal}, the filtration of the $\AW_2$-barycenter is typically larger than the marginal filtrations. In our Gaussian setting, this corresponds to the fact that the $d$-dimensional process $\widebar{X}$ is driven by the stacked noises $\widebar{\boldsymbol{\epsilon}}$ which has $Kd$ spatial dimensions. This leads to the following extension of the concept of filtered Gaussian process.

\begin{definition}[Ensemble filtered Gaussian process] \label{def:extended.filtered.gaussian}
Let
\[
\mathscr L_K(T,d):=\{\mathbf L = (\mathbf{L}_{t,s})_{s,t = 1}^T \in (\bR^{d\times Kd})^{T \times T}: \mathbf L_{t,s}=0\textup{ whenever }s>t\},
\]
where each $\bL_{t,s}$ has $K$ column blocks $\bL_{t,s}^{k} \in \bR^{d \times d}$. For $a \in \bR^{Td}$ and $\mathbf{L} \in \mathscr{L}_K(T, d)$, define the filtered process
\[
\bG_K^{a,\mathbf L}:=\left(\bR^{TKd},\cB(\bR^{TKd}),\cN_{TKd}(0,I),\bF^{\boldsymbol{\epsilon}},X=a+\mathbf {L}\boldsymbol{\epsilon}\right),
\]
where $\boldsymbol{\epsilon}=(\boldsymbol{\epsilon}_t)_{t=1}^T$, $\boldsymbol{\epsilon}_t=(\epsilon_t^1,\ldots,\epsilon_t^K)$, is the canonical process and $\bF^{\boldsymbol\epsilon}$ is the filtration generated by $\boldsymbol\epsilon$. Explicitly, for each $t$ we have
\begin{equation} \label{eqn:ensemble.filtered.Gaussian}
X_t = a_t + \sum_{s = 1}^t \mathbf{L}_{t,s} \boldsymbol{\epsilon}_s = a_t + \sum_{s = 1}^t \sum_{k = 1}^K \bL_{t,s}^{k} \epsilon_s^k. 
\end{equation}
We call $\bG_K^{a,\mathbf L}$ an ensemble filtered Gaussian process and denote the corresponding space by $\cFG_K(T,d)$, and its $\AW_2$-completion by $\mathrm{FG}_K(T,d)$. Letting $K=1$ recovers the standard filtered Gaussian process. 
\end{definition}

In the context of Lemma \ref{thm:gaussian.mc.characterization}, let $\pi = \cL(\widebar{\boldsymbol{\epsilon}}) = \cN_{KTd}(\mathbf{0}, \mathbf{P})$ be a Gaussian multicausal coupling, and consider the filtered process $\widebar{\bX}$ given by \eqref{eqn:mc.to.barycenter}. Note that
\[
\widebar{X}_t = \widebar{a}_t + \sum_{k = 1}^K \lambda_k (L^k \epsilon^k)_t = \widebar{a}_t + \sum_{s = 1}^t \sum_{k = 1}^K \lambda_k L_{t,s}^k \epsilon_s^k
\]
almost has the form \eqref{eqn:ensemble.filtered.Gaussian}, except that the stacked noise $\widebar{\boldsymbol{\epsilon}}$ may not be standard normal due to the covariance $\mathbf{P}$. This apparent inconsistency can be resolved easily.

\begin{lemma} \label{lem:ensemble}
Let $\pi = \cL(\widebar{\boldsymbol{\epsilon}}) = \cN_{KTd}(\mathbf{0}, \mathbf{P})$ be a Gaussian multicausal coupling of the inputs $\bX^1, \ldots, \bX^K \in \cFG$, and define $\widebar{\bX}$ by \eqref{eqn:mc.to.barycenter}. Then there exists $\bX \in \cFG_K(T, d)$ such that
\[
\AW_2(\widebar{\bX}, \bX) = 0.
\]
\end{lemma}
\begin{proof}
Recall that $\mathbf P=\diag(\mathbf P_1,\ldots,\mathbf P_T)$. For each $t$, let $\mathbf M_t:=\mathbf P_t^{1/2}$ and set $\mathbf M:=\diag(\mathbf M_1,\ldots,\mathbf M_T)$. Define $\widebar{\mathbf L}\in\mathscr L_K(T,d)$ by $\widebar{\mathbf L}_{t,s}^k:=\lambda_kL_{t,s}^k$, set $\mathbf L:=\widebar{\mathbf L}\mathbf M$, and let $\bX:=\bG_K^{\widebar a,\mathbf L}$.

Let $\boldsymbol\epsilon$ be a standard $Kd$-dimensional Gaussian white noise and couple the two processes by
\[
\overline{\boldsymbol\epsilon}_t=\mathbf M_t\boldsymbol\epsilon_t,
\qquad t\in[T].
\]
Then $\overline{\boldsymbol\epsilon}\sim\cN_{TKd}(0,\mathbf P)$ and the two paths agree:
\[
\widebar a+\widebar{\mathbf L}\,\overline{\boldsymbol\epsilon}
=\widebar a+\widebar{\mathbf L}\mathbf M\boldsymbol\epsilon
=\widebar a+\mathbf L\boldsymbol\epsilon.
\]
It remains to verify bicausality when some $\mathbf M_t$ is singular. Write
\[
\tilde{\boldsymbol{\epsilon}}_t:=(I_{Kd}-\mathbf M_t^\dagger\mathbf M_t)\boldsymbol\epsilon_t,
\]
where $\mathbf M_t^\dagger$ is the Moore--Penrose inverse. The decomposition
\[
\boldsymbol\epsilon_t
=\mathbf M_t^\dagger\overline{\boldsymbol\epsilon}_t+\tilde{\boldsymbol{\epsilon}}_t
\]
shows that the additional information in $\sigma(\boldsymbol\epsilon_1,\ldots,\boldsymbol\epsilon_t)$ beyond $\sigma(\overline{\boldsymbol\epsilon}_1,\ldots,\overline{\boldsymbol\epsilon}_t)$ is generated by $\tilde{\boldsymbol{\epsilon}}_1,\ldots,\tilde{\boldsymbol{\epsilon}}_t$. These null-space components are jointly Gaussian and independent of the entire sequence $\overline{\boldsymbol\epsilon}$, because they are orthogonal to $\mathbf M_s\boldsymbol\epsilon_s$ at the same time and the noises are independent across time. Thus both conditional-independence requirements hold. The coupling is bicausal and has zero cost, so $\AW_2(\widebar{\bX},\bX)=0$.
\end{proof}

\medskip
\begin{remark} \label{rmk:AW2.ensemble}
Using the same column-covariance argument as in \cite{ABGHP26, gunasingam2026adapted}, the formula \eqref{eq:AW2.Gaussian} extends to $\bX^1\in\cFG_{K_1}(T,d)$ and $\bX^2\in\cFG_{K_2}(T,d)$, with rectangular one-step contractions replacing square orthogonal alignments when $K_1\ne K_2$. Equivalently, the factor contribution is the sum over $t$ of the Bures--Wasserstein distances between the corresponding column covariances.
\end{remark}

\section{The adapted barycenter}
\label{sec:unrestricted-gaussian-barycenters}
This section addresses the unrestricted barycenter problem between filtered Gaussian marginals. We prove existence and show that a representative can always be selected as an ensemble filtered Gaussian process. In general uniqueness does not hold, and we provide a sufficient condition under which it does.

\medskip
Throughout the remaining sections, unless specified otherwise, we fix $K$ marginal centered Gaussian filtered processes \(\bX^k:=\bG^{L^k}\in\cFG(T,d)\)
where $L^1, \ldots, L^K \in \mathscr{L}(T, d)$, and weights \(\lambda_1,\dots,\lambda_K>0\), satisfying \(\sum_k\lambda_k=1\). Set
\begin{equation}
\label{eqn:stacked.Lambda}
\mathbf L_{\cdot,t}:=[L_{\cdot,t}^1\ \cdots\ L_{\cdot,t}^K]\in\bR^{Td\times Kd},
\quad
\Lambda:=\diag(\lambda_1I_d,\ldots,\lambda_KI_d).
\end{equation}
For $t\in[T]$ and $\mathbf P_t\in\mathscr P_+(K,d)$, define the one-step trace functional
\begin{equation}\label{eqn:trace.functional}
\Phi_t(\mathbf P_t):=\tr(\Lambda\mathbf L_{\cdot,t}^\intercal\mathbf L_{\cdot,t})-\tr(\Lambda\mathbf L_{\cdot,t}^\intercal\mathbf L_{\cdot,t}\Lambda\mathbf P_t).
\end{equation}
For $\mathbf P=\diag(\mathbf P_1,\ldots,\mathbf P_T)\in\mathscr P_+(K,T,d)$, define the global trace functional by
\begin{equation}
\label{eq:Phi.global}
\Phi(\mathbf P):=\sum_{t=1}^T\Phi_t(\mathbf P_t).
\end{equation}

\begin{lemma}
\label{lem:trace.form.multicausal.cost}
Let \(\pi\in\Cpl_{\mathrm{mc}}(\bX^1,\ldots,\bX^K)\) be any multicausal coupling, not necessarily Gaussian. For each \(t\in[T]\), define
\[
\mathbf{P}_t:=\left(\E_\pi[\eps_t^k(\eps_t^\ell)^\intercal]\right)_{k,\ell=1}^K.
\]
Then \(\mathbf{P}_t\in\mathscr{P}_+(K,d)\) for every \(t\), and 
\begin{equation}
\label{eq:trace.form.multicausal.cost}
\mathbb{E}_\pi\left[\sum_{k=1}^K \lambda_k \left\|X^k - \phi^0(X^1, \ldots, X^K)\right\|_2^2\right]
=\Phi(\mathbf{P}) = \sum_{t=1}^T\Phi_t(\mathbf P_t).
\end{equation}
In particular every multicausal coupling satisfies
\begin{equation}
\label{eq:multicausal.lower.bound.finite.dim}
\mathbb{E}_\pi\left[\sum_{k=1}^K \lambda_k \left\|X^k - \phi^0(X^1, \ldots, X^K)\right\|_2^2\right]
\ge\sum_{t=1}^T\min_{\mathbf P_t\in\mathscr P_+(K,d)}\Phi_t(\mathbf P_t).
\end{equation}
\end{lemma}

\begin{proof}
For fixed \(t\), the matrix \(\mathbf{P}_t\) is the covariance matrix of the stacked vector \((\eps_t^1,\ldots,\eps_t^K)\) under \(\pi\).  It is therefore positive semidefinite.  Since every marginal \(\eps^k\) is a standard Gaussian white noise, its diagonal blocks are \(\mathbf P_t^{k,k}=I_d\).  Hence \(\mathbf P_t\in\mathscr{P}_+(K,d)\).

Letting \(X^k=L^k\eps^k\), we have
\[
\sum_{k=1}^K\lambda_k\E_\pi\|L^k\eps^k\|_2^2
=\sum_{k=1}^K \lambda_k\|L^k\|_\mathrm{F}^2=\sum_{t=1}^T \tr({\Lambda}\mathbf L_{\cdot, t}^\intercal\mathbf L_{\cdot, t}),
\]
because the marginals of the noises are standard Gaussians.  Moreover,
\[
\E_\pi\left\|\sum_{k=1}^K\lambda_kL^k\eps^k\right\|_2^2
=
\sum_{k,\ell=1}^K\lambda_k\lambda_\ell
\sum_{s,u=1}^T
\tr\left((L_{\cdot,s}^k)^\intercal L_{\cdot,u}^\ell
\E_\pi[\eps_u^\ell(\eps_s^k)^\intercal]\right).
\]
The cross-time terms vanish. Indeed, if $s<u$, multicausality and the white-noise marginal imply
\[
\E_\pi[\eps_u^\ell\mid\widebar{\cF}_{u-1}]
=\E_\pi[\eps_u^\ell\mid\cF_{u-1}^\ell]=0,
\]
while $\eps_s^k$ is $\widebar{\cF}_{u-1}$-measurable; the case $u<s$ follows symmetrically. Thus only $s=u$ remains, and the last display is exactly
\[
\sum_{t=1}^T
\tr(\Lambda\mathbf L_{\cdot,t}^\intercal\mathbf L_{\cdot,t}
\Lambda \mathbf P_t).
\]
Subtracting this from the first term gives $\sum_t\Phi_t(\mathbf P_t)$, proving \eqref{eq:trace.form.multicausal.cost}. The lower bound follows by minimizing each local term over $\mathscr P_+(K,d)$.
\end{proof}

\subsection{Existence of \texorpdfstring{$\FP_2$}{FP\_2}-barycenter}


\begin{theorem}[Existence of adapted barycenter]\label{thm:gaussian.multicausal.optimizer} 
For centered filtered Gaussian processes $\bX^1 = \bG^{L^1},\dots,\bX^K = \bG^{L^K}\in\cFG(T,d)$, the following statements hold:
\begin{enumerate}
\item The barycenter value is given by
\[
\inf_{\bY\in\FP_2}\sum_{k=1}^K \lambda_k\AW_2^2(\bX^k,\bY)=\sum_{t=1}^T\min_{\mathbf P_t\in\mathscr{P}_+(K,d)}\Phi_t(\mathbf P_t),
\]
and each local minimum is attained. 
\item For any local minimizers $\mathbf P^\star_1,\dots,\mathbf P^\star_T\in\mathscr{P}_+(K,d)$ and factors $\mathbf Q_t\in\bR^{Kd\times Kd}$ satisfying $\mathbf P^\star_t=\mathbf Q_t\mathbf Q_t^\intercal$, the ensemble filtered Gaussian process $\bG_K^{0,\widebar{\mathbf L}}$, with block columns
\[
\widebar{\mathbf{L}}_{\cdot, t}=\mathbf{L}_{\cdot, t}\Lambda\mathbf Q_t,
\]
for $t\in[T]$, is an adapted barycenter of $\bX^1,\dots,\bX^K$. In particular, a minimizer of \eqref{eq:AWbar} can be selected as an ensemble filtered Gaussian process. Its path covariance is $\sum_{t=1}^T\widebar{A}_t$, where the innovation covariance contributions are
\begin{equation}\label{eq:bary.cov}
\widebar{A}_t:=\widebar{\mathbf L}_{\cdot,t}\widebar{\mathbf L}_{\cdot,t}^\intercal=\mathbf L_{\cdot,t}\Lambda\mathbf P_t^\star\Lambda\mathbf L_{\cdot,t}^\intercal.
\end{equation}
\end{enumerate}
\end{theorem}

\begin{proof}
For each $t$, the set $\mathscr P_+(K,d)$ is compact and $\Phi_t$ is continuous, so a minimizer $\mathbf P_t^\star$ exists. Lemma~\ref{lem:trace.form.multicausal.cost} gives, for every multicausal coupling $\pi$,
\[
\E_\pi\!\left[\sum_{k=1}^K\lambda_k\left\|X^k-\phi^0(X^1,\ldots,X^K)\right\|_2^2\right]
\ge \sum_{t=1}^T\Phi_t(\mathbf P_t^\star).
\]
Set $\mathbf P^\star:=\diag(\mathbf P_1^\star,\ldots,\mathbf P_T^\star)$ and let $\pi^\star:=\cN_{KTd}(\mathbf0,\mathbf P^\star)$. Its one-step blocks lie in $\mathscr P_+(K,d)$ and are independent across time, so Lemma~\ref{thm:gaussian.mc.characterization} shows that $\pi^\star$ is a Gaussian multicausal coupling. Applying Lemma~\ref{lem:trace.form.multicausal.cost} once more gives equality in the preceding lower bound. Hence $\pi^\star$ solves the multicausal problem \eqref{eq:mc}, proving the value formula.

By \cite[Theorem~4.19]{acciaio2025multicausal}, the barycentric projection
$\widebar{\bX}$ of $\pi^\star$ is an adapted barycenter. Apply
Lemma~\ref{lem:ensemble} with the symmetric factors $(\mathbf P_t^\star)^{1/2}$.
This gives an ensemble factor $\widetilde{\mathbf L}$ such that
\[
\AW_2(\widebar{\bX},\bG_K^{0,\widetilde{\mathbf L}})=0,
\qquad
\widetilde{\mathbf L}_{\cdot,t}\widetilde{\mathbf L}_{\cdot,t}^\intercal
=\mathbf L_{\cdot,t}\Lambda\mathbf P_t^\star\Lambda\mathbf L_{\cdot,t}^\intercal.
\]
Now choose any $\mathbf Q_t\in\bR^{Kd\times Kd}$ with $\mathbf Q_t\mathbf Q_t^\intercal=\mathbf P_t^\star$ and define $\overline{\mathbf L}$ by
\[
\overline{\mathbf L}_{\cdot,t}:=\mathbf L_{\cdot,t}\Lambda\mathbf Q_t,
\qquad t\in[T].
\]
Then $\widebar{\mathbf L}_{\cdot,t}\overline{\mathbf L}_{\cdot,t}^\intercal
=\widetilde{\mathbf L}_{\cdot,t}\widetilde{\mathbf L}_{\cdot,t}^\intercal$ for every $t$. The column-covariance formula of Remark~\ref{rmk:AW2.ensemble} therefore gives
$\AW_2(\bG_K^{0,\widetilde{\mathbf L}},\bG_K^{0,\overline{\mathbf L}})=0$. Hence
$\AW_2(\widebar{\bX},\bG_K^{0,\overline{\mathbf L}})=0$, so the barycenter property transfers to $\bG_K^{0,\overline{\mathbf L}}$. Finally,
\[
\widebar{\mathbf L}_{\cdot,t}\widebar{\mathbf L}_{\cdot,t}^\intercal
=\mathbf L_{\cdot,t}\Lambda\mathbf P_t^\star\Lambda\mathbf L_{\cdot,t}^\intercal,
\]
which proves \eqref{eq:bary.cov}.
\end{proof}

For non-centered inputs $\bX^k=\bG^{a^k,L^k}$, deterministic translations preserve bicausality and give
\[
\AW_2^2(\bG^{a^k,L^k},\bY)
=\|a^k-b\|_2^2+\AW_2^2(\bG^{0,L^k},\bY-b),
\]
where $b$ is the mean of $Y$, and $\bY - b$ has the filtered probability space of $\bY$ and has process $Y - b$. Thus the deterministic mean and centered problems separate. Setting $\overline a:=\sum_{k=1}^K\lambda_ka^k$, the mean term is minimized uniquely at $b=\widebar{a}$, and
\begin{equation}\label{eq:unrestricted.value}
\inf_{\bY\in\FP_2}\sum_{k=1}^K\lambda_k\AW_2^2(\bX^k,\bY)
=
\sum_{k=1}^K\lambda_k\|a^k- \widebar{a}\|_2^2
+
\sum_{t=1}^T\min_{\mathbf P_t\in\mathscr P_+(K,d)}\Phi_t(\mathbf P_t).
\end{equation}
If \(\mathbf P_t^\star\) are local minimizers and \(\mathbf Q_t=(\mathbf P_t^\star)^{1/2}\), we write the corresponding ensemble barycenter factor as
\begin{equation}\label{eq:extended.barycenter.factor}
\overline{\mathbf L}_{\cdot,t}:=\mathbf L_{\cdot,t}\Lambda\mathbf Q_t,
\qquad t\in[T].
\end{equation}
Equivalently, the one-step trace functional has the expanded form
\begin{align}
\label{eq:Phi_t}
\Phi_t(\mathbf P_t)
&:=\tr(\Lambda\mathbf L_{\cdot,t}^{\intercal}\mathbf L_{\cdot,t})
-\tr(\Lambda\mathbf L_{\cdot,t}^{\intercal}\mathbf L_{\cdot,t}\Lambda \mathbf P_t) \\
\label{eq:Phi_t.expanded}
&=\sum_{k=1}^K\lambda_k\|L_{\cdot,t}^k\|_F^2
-\sum_{k,\ell=1}^K\lambda_k\lambda_\ell
\tr\!\left((L_{\cdot,t}^k)^{\intercal}L_{\cdot,t}^{\ell}\mathbf P_t^{\ell,k}\right).
\end{align}

\subsection{An explicit value formula}
\label{sec:explicit.value}

Theorem~\ref{thm:gaussian.multicausal.optimizer} expresses the unrestricted value through the $T$ local minimizations $\min_{\mathbf P_t\in\mathscr{P}_+(K,d)}\Phi_t(\mathbf P_t)$, whose attainment it also provides.  We now identify these local problems: each is a \emph{classical} Bures--Wasserstein barycenter problem for the column covariance matrices
\begin{equation}
\label{eq:A_t_k}
A_t^k:=L_{\cdot,t}^k(L_{\cdot,t}^k)^\intercal\in\mathscr S_+(Td),
\qquad k\in[K],\ t\in[T].
\end{equation}
This removes the optimization over couplings from the value formula entirely.

\begin{theorem}[Local problems as classical Bures--Wasserstein barycenter problems]
\label{thm:explicit.value}
For every $t\in[T]$,
\begin{equation}
\label{eq:local.BW.barycenter.identity}
\min_{\mathbf P_t\in\mathscr{P}_+(K,d)}\Phi_t(\mathbf P_t)
=
\min_{\Sigma\in\mathscr S_+(Td)}
\sum_{k=1}^K\lambda_k\,\mathrm{dist}_{\BW}^2(A_t^k,\Sigma),
\end{equation}
and the map $\mathbf P_t\mapsto\mathbf L_{\cdot,t}\Lambda \mathbf P_t\Lambda\mathbf L_{\cdot,t}^\intercal$ sends the minimizers of the left-hand side onto the minimizers of the right-hand side; in particular, the right-hand minimum is attained.  Every minimizer $\overline\Sigma_t$ satisfies the trace identity
\begin{equation}
\label{eq:barycenter.trace.identity}
\tr(\overline\Sigma_t)
=
\sum_{k=1}^K\lambda_k
\tr\Bigl(\bigl(\overline\Sigma_t^{1/2}A_t^k\overline\Sigma_t^{1/2}\bigr)^{1/2}\Bigr),
\end{equation}
so that
\begin{equation}
\label{eq:explicit.local.value}
\min_{\mathbf P_t\in\mathscr{P}_+(K,d)}\Phi_t(\mathbf P_t)
=
\sum_{k=1}^K\lambda_k\|L_{\cdot,t}^k\|_F^2-\tr(\overline\Sigma_t).
\end{equation}
Consequently, with general means $a^1, \ldots, a^K$ the unrestricted value \eqref{eq:unrestricted.value} takes the explicit form
\begin{equation}
\label{eq:explicit.unrestricted.value}
\inf_{\bY\in\FP_2}\sum_{k=1}^K\lambda_k\AW_2^2(\bX^k,\bY)
=
\sum_{k=1}^K\lambda_k\|a^k-\overline a\|_2^2
+
\sum_{k=1}^K\lambda_k\|L^k\|_F^2
-
\sum_{t=1}^T\tr(\overline\Sigma_t).
\end{equation}
\end{theorem}

\begin{proof}
{\sc Step 1} (lower bound).  Fix $\mathbf P\in\mathscr{P}_+(K,d)$ and choose $\mathbf N\in\bR^{Kd\times Kd}$ with $\mathbf N\mathbf N^\intercal=\mathbf P$; its block rows $\mathbf N^k\in\bR^{d\times Kd}$ satisfy $\mathbf N^k(\mathbf N^k)^\intercal=\mathbf P^{k,k}=I_d$, so $\|\mathbf N^k\|_{\mathrm{op}}=1$.  With
\[
\mathbf Y:=\mathbf L_{\cdot,t}\Lambda\mathbf N=\sum_{k=1}^K\lambda_kL_{\cdot,t}^k\mathbf N^k,
\qquad
\Sigma(\mathbf P):=\mathbf L_{\cdot,t}\Lambda\mathbf P\Lambda\mathbf L_{\cdot,t}^\intercal=\mathbf Y\mathbf Y^\intercal,
\]
we have $\|\mathbf Y\|_F^2=\tr(\Lambda\mathbf L_{\cdot,t}^\intercal\mathbf L_{\cdot,t}\Lambda\mathbf P)$, i.e.\ $\Phi_t(\mathbf P)=\sum_k\lambda_k\|L_{\cdot,t}^k\|_F^2-\|\mathbf Y\|_F^2$, and by the duality of the operator and nuclear norms,
\[
\|\mathbf Y\|_F^2
=\sum_{k=1}^K\lambda_k\bigl\langle \mathbf N^k,(L_{\cdot,t}^k)^\intercal\mathbf Y\bigr\rangle
\le\sum_{k=1}^K\lambda_k\bigl\|(L_{\cdot,t}^k)^\intercal\mathbf Y\bigr\|_* .
\]
Hence, by \eqref{eqn:d.BW} applied with the factors $L_{\cdot,t}^k$ and $\mathbf Y$ of $A_t^k$ and $\Sigma(\mathbf P)$,
\begin{equation}
\label{eq:step1.coupling.bound}
\sum_{k=1}^K\lambda_k \mathrm{dist}_{\BW}^2\bigl(A_t^k,\Sigma(\mathbf P)\bigr)
=
\sum_{k=1}^K\lambda_k\|L_{\cdot,t}^k\|_F^2+\|\mathbf Y\|_F^2
-2\sum_{k=1}^K\lambda_k\bigl\|(L_{\cdot,t}^k)^\intercal\mathbf Y\bigr\|_*
\le
\Phi_t(\mathbf P);
\end{equation}
probabilistically, $\Sigma(\mathbf P)$ is the covariance of the barycentric projection $\sum_\ell\lambda_\ell L_{\cdot,t}^\ell\eps_t^\ell$ under any coupling with one-step covariance $\mathbf P$.

\medskip
{\sc Step 2} (upper bound). Fix $\Sigma\in\mathscr S_+(Td)$. For each $k$, choose a thin singular value decomposition
\[
(L_{\cdot,t}^k)^\intercal\Sigma^{1/2}=U^kS^k(V^k)^\intercal,
\qquad U^k\in\mathscr O(d),\quad (V^k)^\intercal V^k=I_d,
\]
where zero singular directions are completed orthonormally if the matrix has rank smaller than $d$. Set $\mathbf M^k:=U^k(V^k)^\intercal\in\bR^{d\times Td}$. Then $\mathbf M^k(\mathbf M^k)^\intercal=I_d$ and $\langle\mathbf M^k,(L_{\cdot,t}^k)^\intercal\Sigma^{1/2}\rangle=\|(L_{\cdot,t}^k)^\intercal\Sigma^{1/2}\|_*$. Using $\|L_{\cdot,t}^k\mathbf M^k\|_F=\|L_{\cdot,t}^k\|_F$ and \eqref{eqn:d.BW},
\[
\|L_{\cdot,t}^k\mathbf M^k-\Sigma^{1/2}\|_F^2
=\|L_{\cdot,t}^k\|_F^2+\tr(\Sigma)-2\|(L_{\cdot,t}^k)^\intercal\Sigma^{1/2}\|_*
=\mathrm{dist}_{\BW}^2(A_t^k,\Sigma).
\]
Set $\mathbf P:=(\mathbf M^k(\mathbf M^\ell)^\intercal)_{k,\ell=1}^K\in\mathscr{P}_+(K,d)$ and $\overline B:=\sum_{k=1}^K\lambda_kL_{\cdot,t}^k\mathbf M^k$.  The weighted bias--variance decomposition around $\overline B$ (whose cross term vanishes), together with the expansion \eqref{eq:Phi_t.expanded} of $\|\overline B\|_F^2=\tr(\Lambda\mathbf L_{\cdot,t}^\intercal\mathbf L_{\cdot,t}\Lambda\mathbf P)$, gives
\begin{equation}
\label{eq:step2.master.identity}
\sum_{k=1}^K\lambda_k \mathrm{dist}_{\BW}^2(A_t^k,\Sigma)
=\sum_{k=1}^K\lambda_k\|L_{\cdot,t}^k\mathbf M^k-\Sigma^{1/2}\|_F^2
=\Phi_t(\mathbf P)+\|\overline B-\Sigma^{1/2}\|_F^2 .
\end{equation}

\medskip
{\sc Step 3} (identity and correspondence).  The left minimum in \eqref{eq:local.BW.barycenter.identity} is attained by Theorem~\ref{thm:gaussian.multicausal.optimizer}; let $\mathbf P^\star$ be a minimizer.  Combining Step~1 at $\mathbf P^\star$ with \eqref{eq:step2.master.identity},
\[
\min_{\mathbf P}\Phi_t(\mathbf P)
\;\ge\;\sum_k\lambda_k\mathrm{dist}_{\BW}^2\bigl(A_t^k,\Sigma(\mathbf P^\star)\bigr)
\;\ge\;\inf_{\Sigma}\sum_k\lambda_k\mathrm{dist}_{\BW}^2(A_t^k,\Sigma)
\;\ge\;\min_{\mathbf P}\Phi_t(\mathbf P),
\]
so all three quantities coincide: \eqref{eq:local.BW.barycenter.identity} holds and $\Sigma(\mathbf P^\star)$ attains the right-hand minimum.  Conversely, if $\overline\Sigma_t$ is any minimizer of the right-hand side, construct $\mathbf P$ and $\overline B$ from $\Sigma=\overline\Sigma_t$ as in Step~2; then \eqref{eq:step2.master.identity} forces $\Phi_t(\mathbf P)=\min_{\mathbf P'}\Phi_t(\mathbf P')$ and $\overline B=\overline\Sigma_t^{1/2}$, whence $\mathbf L_{\cdot,t}\Lambda\mathbf P\Lambda\mathbf L_{\cdot,t}^\intercal=\overline B\overline B^\intercal=\overline\Sigma_t$.  Thus the correspondence is onto; in particular $\rank(\overline\Sigma_t)\le Kd$, since $\overline B$ is a sum of $K$ matrices of rank at most $d$.

\medskip
{\sc Step 4} (trace identity and value formulas). For $s\ge0$,
\[
\bigl((s\overline\Sigma_t^{1/2})A_t^k(s\overline\Sigma_t^{1/2})\bigr)^{1/2}
=s\bigl(\overline\Sigma_t^{1/2}A_t^k\overline\Sigma_t^{1/2}\bigr)^{1/2}.
\]
Consequently,
\[
s\ \mapsto\ \sum_{k=1}^K\lambda_k\mathrm{dist}_{\BW}^2(A_t^k,s^2\overline\Sigma_t)
=\sum_{k=1}^K\lambda_k\tr(A_t^k)+s^2\tr(\overline\Sigma_t)
-2s\sum_{k=1}^K\lambda_k\tr\bigl((\overline\Sigma_t^{1/2}A_t^k\overline\Sigma_t^{1/2})^{1/2}\bigr)
\]
is a quadratic polynomial minimized over $[0,\infty)$ at the interior point $s=1$; setting its derivative at $s=1$ to zero yields \eqref{eq:barycenter.trace.identity}, and substituting back gives \eqref{eq:explicit.local.value} because $\tr(A_t^k)=\|L_{\cdot,t}^k\|_F^2$. Finally, substituting \eqref{eq:explicit.local.value} into \eqref{eq:unrestricted.value} yields \eqref{eq:explicit.unrestricted.value}.
\end{proof}

\begin{remark}[Explicit and special cases]
\label{rem:explicit.cases}
For $K=2$, writing $C:=\mathbf P_t^{2,1}$, we have $\mathbf P_t\in\mathscr{P}_+(2,d)$ if and only if $\|C\|_{2\to 2}\le1$, and by \eqref{eq:Phi_t.expanded} and $\lambda_k-\lambda_k^2=\lambda_1\lambda_2$,
\[
\Phi_t(\mathbf P_t)
=\lambda_1\lambda_2\Bigl(\|L_{\cdot,t}^1\|_F^2+\|L_{\cdot,t}^2\|_F^2\Bigr)
-2\lambda_1\lambda_2\tr\bigl((L_{\cdot,t}^1)^\intercal L_{\cdot,t}^2\,C\bigr).
\]
Maximizing the trace over contractions (duality of the operator and nuclear norms) yields the closed form
\[
\min_{\mathbf P_t\in\mathscr{P}_+(2,d)}\Phi_t(\mathbf P_t)
=\lambda_1\lambda_2
\Bigl(\|L_{\cdot,t}^1\|_F^2+\|L_{\cdot,t}^2\|_F^2-2\|(L_{\cdot,t}^1)^\intercal L_{\cdot,t}^2\|_*\Bigr)
=\lambda_1\lambda_2\,\mathrm{dist}_{\BW}^2(A_t^1,A_t^2),
\]

\medskip
When $A^k_t$ is positive definite for some $k$, $\overline{\Sigma}_t$ is the unique barycenter covariance and satisfies the Agueh--Carlier fixed-point equation $\overline\Sigma_t=\sum_k\lambda_k(\overline\Sigma_t^{1/2}A_t^k\overline\Sigma_t^{1/2})^{1/2}$; see \cite{AC11}. Under the full-column convention used here, $\rank(A_t^k)\le d<Td$ whenever $T>1$, so all column covariances are singular and minimizers need not be unique. Uniqueness of minimizers and availability of a closed form for the \emph{value} are separate issues: for $K=2$, the value has the closed form displayed above, but the set of minimizers may still not be unique. In general,
\[
\min_{\mathbf P_t\in\mathscr P_+(K,d)}\Phi_t(\mathbf P_t)
=\sum_{k=1}^K\lambda_k\|L_{\cdot,t}^k\|_F^2
-\max_{\mathbf P_t\in\mathscr P_+(K,d)}
\tr(\Lambda\mathbf L_{\cdot,t}^\intercal\mathbf L_{\cdot,t}\Lambda\mathbf P_t),
\]
the value of a semidefinite program; no closed-form expression is known in general.
\end{remark}

\subsection{A uniqueness condition}

Theorem \ref{thm:gaussian.multicausal.optimizer} constructs an unrestricted barycenter as an extended filtered Gaussian process. We now provide a
condition under which the barycenter is unique. The condition is a simultaneous strict regularity condition on the
local stacked matrices \(\mathbf L_{\cdot,t}\).  For \(K=1\), it reduces exactly to the
usual local regularity condition $L^1\in\mathscr{L}^{\mathrm{reg}}(T,d)$, that is
\((L_{\cdot,t}^1)^\intercal L_{\cdot,t}^1\in\mathscr S_{++}(d)\) for each $t$. For \(K>1\), it additionally requires the block-columns of the different inputs to admit compatible orthogonal orientations. This compatibility is precisely what removes the sign and rotation ambiguity in the local optimizers.

\begin{assumption}[Strict regularity]
\label{cond:uniqueness}
For every \(t\in[T]\), there exist
\[
\mathbf Q_t=\diag(Q_t^1,\ldots,Q_t^K),
\qquad Q_t^k\in\mathscr O(d),
\]
such that every \(d\times d\)-block
\begin{equation}
\label{eq:local-strict-regularity-matrix}
(\mathbf Q_t^\intercal \mathbf L_{\cdot,t}^\intercal
\mathbf L_{\cdot,t}\mathbf Q_t)_{k,\ell}\in\mathscr S_{++}(d),\quad k,\ell\in[K].
\end{equation}
Equivalently,
\begin{equation}
\label{eq:local-strict-regularity-blocks}
(Q_t^k)^\intercal (L_{\cdot,t}^k)^\intercal L_{\cdot,t}^\ell Q_t^\ell
\in\mathscr S_{++}(d),
\quad k,\ell\in[K].
\end{equation}
\end{assumption}

\begin{theorem}[Uniqueness of the barycenter]
\label{thm:uniqueness}
Let $\bX^k:=\bG^{L^k}$ for each $k\in[K]$, and suppose that the factors $L^1,\dots,L^K$ satisfy Assumption~\ref{cond:uniqueness}. For each $t\in[T]$, define
\begin{equation}
\label{eq:strict-regularity-barycenter-factor}
\widebar{L}_{\cdot,t}
:=
\sum_{k=1}^K\lambda_k L_{\cdot,t}^kQ_t^k.
\end{equation}
Then \(\widebar{L}\in\mathscr L(T,d)\), and \(\bG^{\overline L}\in\cFG(T,d)\) is the unique unrestricted $\AW_2$-barycenter of \(\bX^1,\ldots,\bX^K\) in \(\FP_2\).  
\end{theorem}

\begin{proof}
Fix \(t\in[T]\), and let $\mathbf Q_t$ satisfy Assumption~\ref{cond:uniqueness}. Set \(\widetilde{\mathbf{P}}_t:=\mathbf Q_t^\intercal \mathbf{P}_t\mathbf Q_t\).  Since \(\mathbf Q_t\) is
block diagonal orthogonal, the mapping
$\mathbf{P}_t\mapsto\widetilde{\mathbf{P}}_t$ is a bijection \(\mathscr{P}_+(K,d)\) onto itself. Thus if
\(G_t^{k,\ell}\) denotes the \((k,\ell)\)-block of
\(\mathbf Q_t^\intercal\mathbf L_{\cdot,t}^\intercal\mathbf L_{\cdot,t}\mathbf Q_t\), then
writing $C_t:=\tr(\Lambda\mathbf L_{\cdot,t}^\intercal\mathbf L_{\cdot,t})$, equation~\eqref{eqn:trace.functional} gives
\begin{equation}
\label{eq:trace-block-expanded-proof}
\Phi_t(\mathbf Q_t\widetilde{\mathbf P}_t\mathbf Q_t^\intercal)
=C_t-
\sum_{k,\ell=1}^K\lambda_k\lambda_\ell
\tr\!\left(G_t^{k,\ell}\widetilde{\mathbf P}_t^{\ell,k}\right).
\end{equation}
Thus minimizing $\Phi_t$ is equivalent to maximizing the sum on the right-hand side of \eqref{eq:trace-block-expanded-proof}.

By Assumption~\ref{cond:uniqueness}, every \(G_t^{k,\ell}\) is symmetric
positive definite.  Moreover, for \(\widetilde{\mathbf{P}}_t\in\mathscr{P}_+(K,d)\), each off-diagonal
block satisfies \(\|\widetilde{\mathbf{P}}_t^{\ell,k}\|_{\mathrm{2\to 2}}\le1\), because
\[
\begin{pmatrix}
I_d&\widetilde{\mathbf{P}}_t^{k,\ell}\\
\widetilde{\mathbf{P}}_t^{\ell,k}&I_d
\end{pmatrix}\succeq0.
\]
For any \(M\in\mathscr S_{++}(d)\) and any \(C\in\bR^{d\times d}\) with
\(\|C\|_{2\to 2}\le1\),
\begin{equation}
\label{eq:trace-contraction-inequality}
\tr(MC)
=\tr\!\left(M\frac{C+C^\intercal}{2}\right)
\le \tr(M),
\end{equation}
and equality holds if and only if \(C=I_d\).  Indeed,
\(I_d-(C+C^\intercal)/2\succeq0\); if the trace in
\eqref{eq:trace-contraction-inequality} is an equality, then
\(M^{1/2}(I_d-(C+C^\intercal)/2)M^{1/2}=0\), hence \((C+C^\intercal)/2=I_d\).  The
contraction bound then forces the skew-symmetric part of \(C\) to vanish.

Applying \eqref{eq:trace-contraction-inequality} term by term to the trace sum in
\eqref{eq:trace-block-expanded-proof}, its unique maximizer is characterized by
\[
\widetilde{\mathbf{P}}_t^{k,\ell}=I_d,
\qquad k,\ell\in[K].
\]
Thus the unique minimizer of \eqref{eqn:trace.functional} is
\begin{equation}
\label{eq:unique-local-P-star}
\mathbf{P}_t^\star
=
\mathbf Q_t
\begin{pmatrix}I_d\\ \vdots\\ I_d\end{pmatrix}
\begin{pmatrix}I_d&\cdots&I_d\end{pmatrix}
\mathbf Q_t^\intercal.
\end{equation}
Hence for each $k,\ell\in[K]$ we have $(\mathbf P_t^\star)^{k,\ell}=Q_t^k(Q_t^\ell)^\intercal$.

Let \(\pi\) be any optimal multicausal coupling and let \(\mathbf P_t(\pi)\) be
the covariance of its stacked time-\(t\) innovations. Lemma~\ref{lem:trace.form.multicausal.cost}
and the uniqueness of every local minimizer give
\[
\sum_{t=1}^T\Phi_t\bigl(\mathbf P_t(\pi)\bigr)
=\sum_{t=1}^T\min_{\mathbf P_t\in\mathscr P_+(K,d)}\Phi_t(\mathbf P_t),
\]
so every nonnegative local gap vanishes and \(\mathbf P_t(\pi)=\mathbf P_t^\star\)
for all \(t\). Consequently,
\[
\E_\pi\bigl\|(Q_t^k)^\intercal\epsilon_t^k
-(Q_t^\ell)^\intercal\epsilon_t^\ell\bigr\|_2^2
=2d-2\tr\!\left((Q_t^k)^\intercal
(\mathbf P_t^\star)^{k,\ell}Q_t^\ell\right)=0.
\]
Hence, with \(\epsilon_t^0:=(Q_t^1)^\intercal\epsilon_t^1\), the coupling satisfies
\[
\epsilon_t^k=Q_t^k\epsilon_t^0,
\qquad k\in[K],\ t\in[T],
\]
almost surely. Since the first marginal is Gaussian white noise,
\(\epsilon_1^0,\ldots,\epsilon_T^0\) are independent standard \(d\)-dimensional
Gaussian vectors. The complete joint law is therefore uniquely determined, and the
Gaussian coupling constructed in Theorem~\ref{thm:gaussian.multicausal.optimizer} shows
that this law is attainable. Thus \(\pi\) is the unique optimal multicausal coupling.
Since every \(Q_t^k\) is orthogonal,
\begin{equation}
\label{eq:strict-regularity-filtration-identity}
\cF_t^k
=\sigma(\epsilon_1^k,\ldots,\epsilon_t^k)
=\sigma(\epsilon_1^0,\ldots,\epsilon_t^0)
=\overline{\cF}_t,
\qquad k\in[K],\ t\in[T],
\end{equation}
up to \(\pi\)-null sets.  The pointwise barycentric projection under $\pi$ is
\[
\widebar{X}=\sum_{k=1}^K\lambda_kX^k
=
\sum_{t=1}^T\sum_{k=1}^K\lambda_kL_{\cdot,t}^kQ_t^k\epsilon^0_t
=
\sum_{t=1}^T\overline L_{\cdot,t}\epsilon^0_t.
\]
Since each \(L_{\cdot,t}^k\) has zero block rows before time \(t\), the same is true of
\( \widebar{L}_{\cdot,t}\); hence \(\widebar{L}\in\mathscr L(T,d)\) and \(\bG^{ \widebar{L}}\in\cFG(T,d)\).  Therefore the filtered process \(\widebar{\bX}:=(\widebar\Omega,\widebar\cF,\pi,\widebar\bF,\widebar{X})\) is an adapted barycenter.  Consider the coupling of \(\widebar{\bX}\) and \(\bG^{\widebar{L}}\) given by the joint law of \((\widebar\omega,\epsilon^0(\widebar\omega))\) under \(\pi\).  Its paths agree, since \(\widebar{X}=\overline L\epsilon^0\), and it is bicausal: by \eqref{eq:strict-regularity-filtration-identity}, the filtration \(\widebar\cF_t\) of \(\widebar{\bX}\) and the filtration \(\sigma(\epsilon_1^0,\ldots,\epsilon_t^0)\) of \(\bG^{\widebar L}\) coincide up to null sets, so that both conditional independence requirements hold trivially.  Therefore \(\AW_2(\widebar{\bX},\bG^{\widebar L})=0\), and \(\bG^{\widebar L}\) is an adapted barycenter.

Now suppose \(\bY\in\FP_2\) is another $\AW_2$-barycenter. For each $k\in[K]$, choose a bicausal coupling of \(\bX^k\) and
\(\bY\) attaining $\AW_2^2(\bX^k,\bY)$, and glue these couplings along $\bY$ to obtain a coupling \(\boldsymbol{\pi}\) of
\((\bX^1,\ldots,\bX^K,\bY)\); by the gluing construction recalled in Section~\ref{sec:multicausal}, the $(\bX^1,\dots,\bX^K)$-marginal \(\pi'\) of $\boldsymbol{\pi}$ is multicausal. Writing $V^\star$ for the optimal barycenter value, the pointwise
barycentric identity gives
\begin{align*}
V^\star=\E_{\boldsymbol{\pi}}\sum_{k=1}^K\lambda_k\|X^k-Y\|_2^2
&= \E_{\boldsymbol{\pi}}\sum_{k=1}^K\lambda_k\|X^k-\widebar{X}\|_2^2+\E_{\boldsymbol{\pi}}\|Y-\widebar X\|_2^2\\
&=\E_{\pi'}[c(X^1,\ldots,X^K)]
+
\E_{\boldsymbol{\pi}}\|Y-\widebar X\|_2^2,
\end{align*}
where $\widebar X=\sum_{k=1}^K\lambda_kX^k$.
Since \(\pi'\) is multicausal, we have $\E_{\pi'}[c(X^1,\ldots,X^K)]\ge V^\star$ by the multicausal reformulation \eqref{eq:mc}, while the second term on the right-hand side is nonnegative. Both terms must therefore be at their optimal values: $\pi'$ is an optimal multicausal coupling, and $Y=\widebar{X}$, $\boldsymbol{\pi}$-almost surely. By uniqueness of the multicausal optimizer established in the first part of the proof, we conclude that $\pi'=\pi$.
It remains to show that the coupling between $\widebar{\bX}$ and
\(\bY\) is bicausal.  Fix $k\in[K]$. Since \(\boldsymbol{\pi}\) was obtained by gluing optimal bicausal
couplings between \(\bX^k\) and \(\bY\), the \((\bX^k,\bY)\)-marginal of \(\boldsymbol{\pi}\) is bicausal.  Thus, for every \(t\in[T]\),
\[
\cF_T^k\perp\!\!\!\perp \cF_t^Y\mid \cF_t^k,
\qquad
\cF_T^Y\perp\!\!\!\perp \cF_t^k\mid \cF_t^Y.
\]
However, by \eqref{eq:strict-regularity-filtration-identity} --- a statement about the $(\bX^1,\ldots,\bX^K)$-marginal $\pi$ of $\boldsymbol\pi$ --- we have $\cF_t^k=\overline{\cF}_t$ up to $\boldsymbol\pi$-null sets.  Therefore the preceding bicausality relations become
\[
\widebar{\cF}_T\perp\!\!\!\perp \cF_t^Y\mid \widebar{\cF}_t,
\qquad
\cF_T^Y\perp\!\!\!\perp \widebar{\cF}_t\mid \cF_t^Y.
\]
Hence $(\bG^{\widebar{L}},\bY)$ are bicausally coupled. Hence $\AW_2(\bY,\bG^{\widebar L})=0$. Since \(\bY\) was arbitrary, \(\bG^{\widebar L}\) is the unique unrestricted
\(\AW_2\)-barycenter in \(\FP_2\).

\end{proof}

\begin{example}
\label{ex:regular-inputs-nonunique-unrestricted}
Let \(T=2\), \(d=1\), \(K=2\), and \(\lambda_1=\lambda_2=1/2\).  Consider 
filtered Gaussian inputs $\bX^1:=\bG^{L^1}$ and $\bX^2:=\bG^{L^2}$ with
\[
L^1=
\begin{pmatrix}
1&0\\
1&1
\end{pmatrix},
\qquad
L^2=
\begin{pmatrix}
1&0\\
-1&1
\end{pmatrix}.
\]
Observe that $L^1,L^2\in\mathscr{L}^{\textup{reg}}(2,1)$. We compute the unrestricted value from the local trace problems in \eqref{eq:unrestricted.value}.  For
\(\rho\in[-1,1]\), set
\[
P(\rho):=
\begin{pmatrix}1&\rho\\ \rho&1\end{pmatrix}\in\mathscr{P}_+(2).
\]
At time \(t=1\),
\[
\mathbf L_{\cdot,1}=
\begin{pmatrix}
1&1\\
1&-1
\end{pmatrix},
\qquad
\mathbf L_{\cdot,1}^\intercal\mathbf L_{\cdot,1}=2I_2.
\]
Thus
\begin{align*}
&\tr\!\left(\Lambda\mathbf L_{\cdot,1}^\intercal\mathbf L_{\cdot,1}\right)
-
\tr\!\left(\Lambda\mathbf L_{\cdot,1}^\intercal\mathbf L_{\cdot,1}\Lambda P(\rho)\right) \\
&\qquad =2-\tr\!\left(\frac12 I_2P(\rho)\right)=1,
\qquad \rho\in[-1,1].
\end{align*}
At time \(t=2\),
\[
\mathbf L_{\cdot,2}=
\begin{pmatrix}
0&0\\
1&1
\end{pmatrix},
\qquad
\mathbf L_{\cdot,2}^\intercal\mathbf L_{\cdot,2}=
\begin{pmatrix}
1&1\\
1&1
\end{pmatrix},
\]
and hence
\begin{align*}
&\tr\!\left(\Lambda\mathbf L_{\cdot,2}^\intercal\mathbf L_{\cdot,2}\right)
-
\tr\!\left(\Lambda\mathbf L_{\cdot,2}^\intercal\mathbf L_{\cdot,2}\Lambda P(\rho)\right) \\
&\qquad =1-\frac14\tr\!\left(
\begin{pmatrix}
1&1\\
1&1
\end{pmatrix}
\begin{pmatrix}
1&\rho\\
\rho&1
\end{pmatrix}
\right)
=\frac{1-\rho}{2}.
\end{align*}
By \eqref{eq:unrestricted.value}, the unrestricted centered value is therefore
\[
\sum_{t=1}^2\inf_{\mathbf{P}_t\in\mathscr{P}_+(2)}\Phi_t(\mathbf{P}_t)=1+0=1,
\]
attained by an arbitrary
\(P_1=P(\rho)\) at time one and by \(P_2=P(1)\) at time two.

Taking \(P_1=P(1)\), \(P_2=P(1)\) gives the Cholesky factor \[
\widebar L^{(+)}=
\begin{pmatrix}
1&0\\
0&1
\end{pmatrix}.
\]
Taking \(P_1=P(-1)\), \(P_2=P(1)\) gives the Cholesky factor
\[
\widebar L^{(-)}=
\begin{pmatrix}
0&0\\
1&1
\end{pmatrix}.
\]
 Since
\(\|L^1\|_F^2=\|L^2\|_F^2=3\) and
\(\|\widebar L^{(+)}\|_F^2=\|\widebar L^{(-)}\|_F^2=2\), while the sums of the absolute
time-diagonal inner products are equal to \(2\),
\[
\AW_2^2(\bG^{0,L^1},\bG^{0,\widebar L^{(+)}})
=\AW_2^2(\bG^{0,L^2},\bG^{0,\widebar L^{(+)}})=1,
\]
and
\[
\AW_2^2(\bG^{0,L^1},\bG^{0,\widebar L^{(-)}})
=\AW_2^2(\bG^{0,L^2},\bG^{0,\widebar L^{(-)}})=1.
\]
Thus, both attain the barycenter value \(1\).  They are not \(\AW_2\)-equivalent:
the first-time marginal of \(\bG^{0,\widebar L^{(+)}}\) is \(\cN(0,1)\), whereas the
first-time marginal of \(\bG^{0,\widebar L^{(-)}}\) is \(\delta_0\).  Hence regularity of
the input filtered Gaussians alone does not imply uniqueness of the unrestricted
\(\FP_2\)-barycenter.

\end{example}

\begin{example}
\label{ex:strict-regularity-not-necessary}
Assumption~\ref{cond:uniqueness} is only sufficient.  Let \(T=2\),
\(d=1\), \(K=3\), and \(\lambda_1=\lambda_2=\lambda_3=1/3\).  Set
\[
L^1=\begin{pmatrix}1&0\\1&1\end{pmatrix},
\qquad
L^2=\begin{pmatrix}1&0\\-1&1\end{pmatrix},
\qquad
L^3=\begin{pmatrix}1&0\\0&1\end{pmatrix}.
\]
Condition \eqref{eq:local-strict-regularity-blocks} fails at time \(1\), since
\((L_{\cdot,1}^1)^\intercal L_{\cdot,1}^2=0\), and in dimension one the orthogonal
matrices \(Q_t^k\) can only change signs.

Nevertheless the unrestricted barycenter is unique.  Write a generic element of
\(\mathscr{P}_+(3)\) as
\[
P=\begin{pmatrix}
1&\rho_{12}&\rho_{13}\\
\rho_{12}&1&\rho_{23}\\
\rho_{13}&\rho_{23}&1
\end{pmatrix}.
\]
At time \(1\), the only nonzero pairwise inner products are
\[
(L_{\cdot,1}^1)^\intercal L_{\cdot,1}^3=1,
\qquad
(L_{\cdot,1}^2)^\intercal L_{\cdot,1}^3=1.
\]
Thus minimizing the time-\(1\) local objective is equivalent to maximizing
\(\rho_{13}+\rho_{23}\).  Since \(P\succeq0\) implies
\(\rho_{13}\le1\) and \(\rho_{23}\le1\), any maximizer must satisfy
\(\rho_{13}=\rho_{23}=1\).  Then
\[
\det P=-(\rho_{12}-1)^2,
\]
so positive semidefiniteness forces \(\rho_{12}=1\).  Hence the unique time-\(1\)
optimizer is the all-ones matrix.  At time \(2\), all three block-columns equal
\((0,1)^\intercal\), so the unique optimizer is again the all-ones matrix.
Although Assumption~\ref{cond:uniqueness} fails, the uniqueness argument in the proof of Theorem~\ref{thm:uniqueness} only uses that each local optimizer \(\mathbf{P}_t^\star\) is unique and of the form \eqref{eq:unique-local-P-star}; both properties hold here (with \(Q_t^k=1\)), so the barycenter is unique.  The barycenter factor is
\[
\widebar L_{\cdot,1}
=\frac13\left(\begin{pmatrix}1\\1\end{pmatrix}
+\begin{pmatrix}1\\-1\end{pmatrix}
+\begin{pmatrix}1\\0\end{pmatrix}\right)
=\begin{pmatrix}1\\0\end{pmatrix},
\qquad
\widebar L_{\cdot,2}=\begin{pmatrix}0\\1\end{pmatrix},
\]
so \(\widebar L=I_2\).  Thus, strict regularity is not necessary for uniqueness.
\end{example}

\begin{remark}[Filtration of the adapted barycenter]
Theorem~\ref{thm:gaussian.multicausal.optimizer} constructs an unrestricted barycenter in the enlarged class \(\cFG_K(T,d)\). It does not imply that the same equivalence class can be represented by a process driven by \(d\)-dimensional innovations. The next example shows that this reduction can fail.
\end{remark}

\subsection{A barycenter outside \texorpdfstring{\(\cFG(T,d)\)}{FG(2,1)}}\label{sec:counter}
The next example is the basic obstruction.  It shows that the enlarged innovation dimension in Theorem~\ref{thm:gaussian.multicausal.optimizer} cannot be removed.

\begin{example}[An unrestricted barycenter outside \(\cFG(2,1)\)]
\label{ex:barycenter.not.in.FG}
Let \(T=2\), \(d=1\), \(K=3\), and \(\lambda_1=\lambda_2=\lambda_3=1/3\).  Consider centered filtered Gaussian processes \(\bX^k=\bG^{L^k}\in\cFG(2,1)\), where
\begin{equation} \label{eq:ex.marginals}
L^1=\begin{pmatrix}1&0\\0&0\end{pmatrix},
\qquad
L^2=\begin{pmatrix}-\frac12&0\\[1mm]\frac{\sqrt3}{2}&0\end{pmatrix},
\qquad
L^3=\begin{pmatrix}-\frac12&0\\[1mm]-\frac{\sqrt3}{2}&0\end{pmatrix}.
\end{equation}
Only the first block-column is nonzero, so that $\Phi_2\equiv0$.  The three vectors \(L_{\cdot,1}^1,L_{\cdot,1}^2,L_{\cdot,1}^3\) are unit vectors in \(\bR^2\) with pairwise inner product \(-1/2\).  For \(\mathbf P_1\in\mathscr{P}_+(3)\),
\[
\Phi_1(\mathbf P_1)=\frac23+\frac19(\mathbf P_1^{1,2}+\mathbf P_1^{1,3}+\mathbf P_1^{2,3}).
\]
Since \(\mathbf P_1\ge0\) and has unit diagonal,
\[
0\le (1,1,1)\mathbf P_1(1,1,1)^\intercal=3+2(\mathbf P_1^{1,2}+\mathbf P_1^{1,3}+\mathbf P_1^{2,3}),
\]
so \(\mathbf P_1^{1,2}+\mathbf P_1^{1,3}+\mathbf P_1^{2,3}\ge-3/2\).  Equality is attained by
\[
\mathbf P_1^\star=\begin{pmatrix}
1&-\frac12&-\frac12\\
-\frac12&1&-\frac12\\
-\frac12&-\frac12&1
\end{pmatrix}.
\]
Therefore
\begin{equation}
\label{eq:example.unrestricted.value}
\inf_{\bY\in\FP_2}\sum_{k=1}^3\frac13\AW_2^2(\bX^k,\bY)=\frac12.
\end{equation}

Next, we consider the restricted problem which minimizes within $\cFG(2, 1)$.  Let \(\bY=\bG^{a,L}\), with
\[
a=\begin{pmatrix}a_1\\a_2\end{pmatrix},
\qquad
L=\begin{pmatrix}x&0\\y&z\end{pmatrix}.
\]
By \eqref{eq:AW2.Gaussian} and \eqref{eq:ABW},
\begin{align*}
\sum_{k=1}^3\frac13\AW_2^2(\bX^k,\bY)
&=a_1^2+a_2^2+1+x^2+y^2+z^2\\
&\quad -\frac23\left(|x|+\left|-\frac12x+\frac{\sqrt3}{2}y\right|+\left|-\frac12x-\frac{\sqrt3}{2}y\right|\right).
\end{align*}
Writing \((x,y)=\rho(\cos\theta,\sin\theta)\), the three absolute values are the
absolute projections of a vector of length \(\rho\) onto three directions separated
by angle \(2\pi/3\).  We claim that their sum is at most \(2\rho\).  Dividing by
\(\rho\) when \(\rho>0\), this is the inequality
\[
|\cos\theta|+|\cos(\theta-2\pi/3)|+|\cos(\theta+2\pi/3)|\le 2.
\]
The left-hand side is invariant under \(\theta\mapsto-\theta\) and under
\(\theta\mapsto\theta+2\pi/3\), so it suffices to consider
\(0\le\theta\le\pi/3\).  If \(0\le\theta\le\pi/6\), the sum equals
\(2\cos\theta\le2\).  If \(\pi/6\le\theta\le\pi/3\), the sum equals
\(\cos\theta+\sqrt3\sin\theta=2\cos(\theta-\pi/3)\le2\).  The case
\(\rho=0\) is immediate.  Hence
\[
\sum_{k=1}^3\frac13\AW_2^2(\bX^k,\bY)
\ge a_1^2+a_2^2+z^2+1+\rho^2-\frac43\rho\ge\frac59.
\]
The lower bound is attained by \(a=0\), \(z=0\), \(x=2/3\), and \(y=0\).  Thus,
\begin{equation}
\label{eq:example.restricted.value}
\inf_{\bY\in\cFG(2,1)}\sum_{k=1}^3\frac13\AW_2^2(\bX^k,\bY)=\frac59.
\end{equation}

{Moreover, equality holds precisely when \(a=0\), \(z=0\),
\(\rho=2/3\), and \((x,y)\) is parallel to one of the vectors
\(L_{\cdot,1}^k\).  Thus, modulo the sign invariance of the Gaussian
factor, the restricted problem has the three minimizers
\(\bG^{\bar L^k},\) where 
\(\bar L^k:=\frac23 L^k\) for $k=1,2,3$. Equivalently, their path covariances are
\[
\bar L^k(\bar L^k)^\intercal
  =\frac49 L^k(L^k)^\intercal.
\]
Hence each restricted barycenter is supported in one of the three input
directions, with standard deviation reduced by a factor of \(2/3\).}
Since \(5/9>1/2\), no unrestricted barycenter admits a representative in \(\cFG(2,1)\). The situation is depicted graphically in Figure \ref{fig:barycenter-weight-sweep} which also shows several other barycenters with different weights.
\end{example}

\begin{remark}[The filtration obstruction]
In Example~\ref{ex:barycenter.not.in.FG}, the optimal matrix \(\mathbf P_1^\star\) has rank \(2\), although \(d=1\).  The covariance of the barycentric projection at time \(1\) is
\[
\mathbf L_{\cdot,1}\Lambda \mathbf P_1^\star\Lambda\mathbf L_{\cdot,1}^\intercal=\frac14I_2.
\]
Thus, any Gaussian linear representative of this unrestricted barycenter requires two independent Gaussian directions at time \(1\), while a process in \(\cFG(2,1)\) has only one fresh scalar Gaussian innovation at that time. In the language of Theorem~\ref{thm:explicit.value}, $\overline\Sigma_1=\frac14I_2$ is a Bures--Wasserstein barycenter of the three rank-one covariances $A_1^1,A_1^2,A_1^3$, and the unrestricted value is $\sum_k\frac13\tr(A_1^k)-\tr(\overline\Sigma_1)=1-\frac12=\frac12$, in accordance with \eqref{eq:example.unrestricted.value}.
\end{remark}

\medskip

\begin{figure}[htbp]
\centering
\begin{tikzpicture}[line cap=round,line join=round]
  \tikzset{
    input one/.style={blue!85!black,line width=0.55pt},
    input two/.style={orange!95!black,line width=0.55pt},
    input three/.style={green!60!black,line width=0.55pt},
    unrestricted/.style={black,line width=1.35pt},
    restricted/.style={black!42,densely dotted,line width=1.15pt},
    coordinate axis/.style={gray!22,line width=0.22pt}
  }

  \newcommand{\baryaxes}{%
    \draw[coordinate axis] (-1.08,0)--(1.08,0);
    \draw[coordinate axis] (0,-1.08)--(0,1.08);
  }
  \newcommand{\baryinputs}{%
    \draw[input one] (-1,0)--(1,0);
    \draw[input two] (-0.5,0.8660)--(0.5,-0.8660);
    \draw[input three] (-0.5,-0.8660)--(0.5,0.8660);
  }
  \newcommand{\weightlabel}[1]{%
    \node[font=\scriptsize,anchor=south] at (0,1.13)
      {$\boldsymbol\lambda=#1$};
  }

  \path[use as bounding box] (-6.60,-1.90) rectangle (6.60,9.35);

  \begin{scope}[x=1cm,y=1cm,yshift=8.10cm]
    \baryaxes
    \baryinputs
    \draw[unrestricted] (-0.5,-0.8660)--(0.5,0.8660);
    \draw[restricted] (-0.5,-0.8660)--(0.5,0.8660);
    \weightlabel{(0,0,1)}
  \end{scope}

  \begin{scope}[x=1cm,y=1cm,xshift=-2.45cm,yshift=5.55cm]
    \baryaxes
    \baryinputs
    \draw[unrestricted] (-0.7000,-0.5196)--(0.7000,0.5196);
    \draw[restricted] (-0.7000,-0.5196)--(0.7000,0.5196);
    \weightlabel{(0.4,0,0.6)}
  \end{scope}

  \begin{scope}[x=1cm,y=1cm,xshift=2.45cm,yshift=5.55cm]
    \baryaxes
    \baryinputs
    \draw[unrestricted] (-0.1000,-0.8660)--(0.1000,0.8660);
    \draw[restricted] (-0.1000,-0.8660)--(0.1000,0.8660);
    \weightlabel{(0,0.4,0.6)}
  \end{scope}

  \begin{scope}[x=1cm,y=1cm,xshift=-4.00cm,yshift=2.85cm]
    \baryaxes
    \fill[white,rotate=30] (0,0)
      ellipse[x radius=0.6708,y radius=0.3000];
    \baryinputs
    \draw[unrestricted,rotate=30] (0,0)
      ellipse[x radius=0.6708,y radius=0.3000];
    \draw[restricted] (-0.5000,-0.5196)--(0.5000,0.5196);
    \draw[restricted] (-0.7000,-0.1732)--(0.7000,0.1732);
    \weightlabel{(0.4,0.2,0.4)}
  \end{scope}

  \begin{scope}[x=1cm,y=1cm,yshift=2.85cm]
    \baryaxes
    \fill[white] (0,0) circle[radius=0.5000];
    \baryinputs
    \draw[unrestricted] (0,0) circle[radius=0.5000];
    \draw[restricted] (-0.6667,0)--(0.6667,0);
    \draw[restricted] (-0.3333,0.5774)--(0.3333,-0.5774);
    \draw[restricted] (-0.3333,-0.5774)--(0.3333,0.5774);
    \weightlabel{(1/3,1/3,1/3)}
  \end{scope}

  \begin{scope}[x=1cm,y=1cm,xshift=4.00cm,yshift=2.85cm]
    \baryaxes
    \fill[white] (0,0)
      ellipse[x radius=0.3000,y radius=0.6708];
    \baryinputs
    \draw[unrestricted] (0,0)
      ellipse[x radius=0.3000,y radius=0.6708];
    \draw[restricted] (-0.2000,-0.6928)--(0.2000,0.6928);
    \draw[restricted] (-0.2000,0.6928)--(0.2000,-0.6928);
    \weightlabel{(0.2,0.4,0.4)}
  \end{scope}

  \begin{scope}[x=1cm,y=1cm,xshift=-5.25cm]
    \baryaxes
    \baryinputs
    \draw[unrestricted] (-1,0)--(1,0);
    \draw[restricted] (-1,0)--(1,0);
    \weightlabel{(1,0,0)}
  \end{scope}

  \begin{scope}[x=1cm,y=1cm,xshift=-1.75cm]
    \baryaxes
    \baryinputs
    \draw[unrestricted] (-0.8000,0.3464)--(0.8000,-0.3464);
    \draw[restricted] (-0.8000,0.3464)--(0.8000,-0.3464);
    \weightlabel{(0.6,0.4,0)}
  \end{scope}

  \begin{scope}[x=1cm,y=1cm,xshift=1.75cm]
    \baryaxes
    \baryinputs
    \draw[unrestricted] (-0.7000,0.5196)--(0.7000,-0.5196);
    \draw[restricted] (-0.7000,0.5196)--(0.7000,-0.5196);
    \weightlabel{(0.4,0.6,0)}
  \end{scope}

  \begin{scope}[x=1cm,y=1cm,xshift=5.25cm]
    \baryaxes
    \baryinputs
    \draw[unrestricted] (-0.5,0.8660)--(0.5,-0.8660);
    \draw[restricted] (-0.5,0.8660)--(0.5,-0.8660);
    \weightlabel{(0,1,0)}
  \end{scope}

  \begin{scope}[x=1cm,y=1cm,yshift=-1.55cm]
    \draw[input one] (-6.10,0)--(-5.55,0);
    \node[anchor=west,font=\scriptsize] at (-5.40,0)
      {$\mathbb X^1$};

    \draw[input two] (-4.35,0)--(-3.80,0);
    \node[anchor=west,font=\scriptsize] at (-3.65,0)
      {$\mathbb X^2$};

    \draw[input three] (-2.60,0)--(-2.05,0);
    \node[anchor=west,font=\scriptsize] at (-1.90,0)
      {$\mathbb X^3$};

    \draw[unrestricted] (-0.60,0)--(0.20,0);
    \node[anchor=west,font=\scriptsize] at (0.35,0)
      {unrestricted};

    \draw[restricted] (3.50,0)--(4.30,0);
    \node[anchor=west,font=\scriptsize] at (4.45,0)
      {restricted};
  \end{scope}
\end{tikzpicture}

\caption{Graphical illustration of Example \ref{ex:barycenter.not.in.FG}. The marginals $L^i$ in \eqref{ex:barycenter.not.in.FG}, which correspond to three unit vectors in $\bR^2$, are represented as line segments with pairwise angle $\pi/2$. When the weight $\lambda$ is close to $(1/3, 1/3, 1/3)$, the unrestricted barycenter requires a higher dimensional driving noise depicted by the ellipses, and hence lies outside $\cFG(2, 1)$.}
\label{fig:barycenter-weight-sweep}
\end{figure}

\section{Restricted Adapted Barycenter}
\label{sec:FG-common-noise}
Theorem~\ref{thm:gaussian.multicausal.optimizer} constructs an unrestricted barycenter as an ensemble filtered Gaussian process, and Example~\ref{ex:barycenter.not.in.FG} shows that the enlarged noise cannot be removed from that Gaussian representation in general. In this section we first characterize exactly when the unrestricted problem admits an ordinary filtered Gaussian barycenter in $\cFG(T,d)$, and we then solve the barycenter problem restricted to $\cFG(T,d)$, together with a fixed-point characterization and a regularity result for its minimizers.

\subsection{A necessary and sufficient condition}
Recall the column covariance matrices \(A_t^k=L_{\cdot,t}^k(L_{\cdot,t}^k)^\intercal\in\mathscr S_+(Td)\) from \eqref{eq:A_t_k}.

\begin{lemma}[Rank-\(d\) matrices in \(\mathscr{P}_+(K,d)\)]
\label{lem:Pk-rank-common-noise}
Let \(\mathbf P\in\mathscr{P}_+(K,d)\).  Then \(\rank(\mathbf P)=d\) if and only if there exist orthogonal matrices \(\mathbf Q^1,\ldots,\mathbf Q^K\in\mathscr O(d)\) such that
\begin{equation}
\label{eq:rankd.common.noise}
\mathbf P^{k,\ell}=\mathbf Q^k(\mathbf Q^\ell)^\intercal,
\qquad k,\ell\in[K].
\end{equation}
\end{lemma}

\begin{proof}
Since every diagonal block of \(\mathbf P\) is \(I_d\), every \(\mathbf P\in\mathscr{P}_+(K,d)\) has rank at least \(d\).  Suppose first that \(\rank(\mathbf P)=d\).  Choose a factorization \(\mathbf P=\mathbf R\mathbf R^\intercal\) with \(\mathbf R\in\bR^{Kd\times d}\), and write \(R^k\in\bR^{d\times d}\) for the \(k\)-th block row of \(\mathbf R\).  The diagonal block condition gives
\[
R^k(R^k)^\intercal=\mathbf P^{k,k}=I_d.
\]
Thus, each \(R^k\) is orthogonal, and \eqref{eq:rankd.common.noise} holds with \(\mathbf Q^k=R^k\).  Conversely, if \eqref{eq:rankd.common.noise} holds, then \(\mathbf P=\mathbf R\mathbf R^\intercal\), where \(\mathbf R\) has block rows \(\mathbf Q^1,\ldots,\mathbf Q^K\).  Since \(\mathbf Q^1\) is invertible, \(\rank(\mathbf R)=d\), and hence \(\rank(\mathbf P)=d\).
\end{proof}

\begin{condition}[Common-noise input condition]
\label{cond:common-noise.input}
The input factors \(L^1,\ldots,L^K\) and weights \(\lambda_1,\ldots,\lambda_K\) are said to satisfy the common-noise input condition if, for every \(t\in[T]\),
\begin{equation}
\label{eq:common-noise.input.condition}
\min_{\mathbf P_t\in\mathscr{P}_+(K,d)}\Phi_t(\mathbf P_t)
=
\min_{\substack{\mathbf P_t\in\mathscr{P}_+(K,d)\\ \rank(\mathbf P_t)=d}}\Phi_t(\mathbf P_t).
\end{equation}
Equivalently, each local problem admits an optimizer of rank \(d\).  This condition depends only on the matrices \(L_{\cdot,t}^k\) and the weights \(\lambda_k\), through the trace functional \(\Phi_t\).
\end{condition}

\begin{theorem}[Common-noise criterion for filtered Gaussian barycenters]
\label{thm:FG-common-noise-criterion}
The following statements are equivalent:
\begin{enumerate}
\item[(i)] The common-noise input condition \eqref{eq:common-noise.input.condition} holds.
\item[(ii)] For every \(t\in[T]\), there exist orthogonal matrices \(\mathbf Q_t^1,\ldots,\mathbf Q_t^K\in\mathscr O(d)\) such that the matrix
\begin{equation}
\label{eq:common-noise.optimizer.matrix}
(\mathbf P_t^\star)^{k,\ell}=\mathbf Q_t^k(\mathbf Q_t^\ell)^\intercal,
\qquad k,\ell\in[K],
\end{equation}
minimizes \(\Phi_t\) over \(\mathscr{P}_+(K,d)\).
\item[(iii)] There exists an unrestricted barycenter \(\overline{\bX}\in\cFG_K(T,d)\) constructed in Theorem~\ref{thm:gaussian.multicausal.optimizer} which is \(\AW_2\)-equivalent to an element of \(\cFG(T,d)\).
\item[(iv)] The unrestricted adapted barycenter problem admits a barycenter representative in \(\cFG(T,d)\).
\end{enumerate}
When these equivalent conditions hold, the matrix \(\mathbf P^\star=\diag(\mathbf P_1^\star,\ldots,\mathbf P_T^\star)\in\mathscr{P}_+(K,T,d)\) minimizes \(\Phi\), and
$
\bG^{\widebar a,\widebar L}\in \cFG(T,d)
$
is an $\FP_2$-barycenter, where $\widebar{L}\in\mathscr{L}(T,d)$ is defined by
\[
\widebar L_{\cdot,t}=\sum_{k=1}^K\lambda_kL_{\cdot,t}^k\mathbf Q_t^k,
\quad t\in[T].
\]
Moreover, if $\widebar{\bX}:=\bG_K^{\widebar a,\widebar{\mathbf L}}$ denotes the barycenter representative built from the optimizers $\mathbf P_t^\star$ in Theorem~\ref{thm:gaussian.multicausal.optimizer}, then
\begin{equation}
\label{eq:AW-zero-extended-common-noise}
\AW_2(\widebar{\bX},\bG^{\widebar{a},\widebar{L}})=0.
\end{equation}
\end{theorem}

\begin{proof}
The equivalence of (i) and (ii) follows from Lemma~\ref{lem:Pk-rank-common-noise}: every rank-\(d\) element of \(\mathscr{P}_+(K,d)\) is exactly a common-noise matrix of the form \eqref{eq:common-noise.optimizer.matrix}, and conversely every such matrix has rank \(d\).

Assume (ii), and let \(\widebar{\bX}=\bG_K^{\widebar a,\widebar{\mathbf L}}\) be the enlarged-noise barycenter representative built from the optimizers \(\mathbf P_t^\star\) in Theorem~\ref{thm:gaussian.multicausal.optimizer}.  We prove \eqref{eq:AW-zero-extended-common-noise}, which establishes (iii).  Fix \(t\), and set
\[
\mathbf R_t:=\begin{pmatrix}\mathbf Q_t^1\\ \vdots\\ \mathbf Q_t^K\end{pmatrix}\in\bR^{Kd\times d},
\qquad
\mathbf U_t:=\frac{1}{\sqrt{K}}\mathbf R_t.
\]
Then \(\mathbf U_t^\intercal \mathbf U_t=I_d\), \(\mathbf P_t^\star=\mathbf R_t\mathbf R_t^\intercal=K\mathbf U_t\mathbf U_t^\intercal\), and therefore \((\mathbf P_t^\star)^{1/2}=\sqrt{K}\mathbf U_t\mathbf U_t^\intercal\).  Let \((\beps_t)_{t=1}^T\) be the canonical \(Kd\)-dimensional white noise of \(\widebar{\bX}\), and couple it to the \(d\)-dimensional white noise by
\[
\eps_t:=\mathbf U_t^\intercal\beps_t,
\qquad t\in[T].
\]
Then \((\eps_t)_{t=1}^T\) is a standard \(d\)-dimensional Gaussian white noise.  Using \eqref{eq:extended.barycenter.factor}, we obtain pathwise
\begin{align*}
\widebar{\mathbf L}_{\cdot,t}\beps_t
&=\mathbf L_{\cdot,t}\Lambda(\mathbf P_t^\star)^{1/2}\beps_t \\
&=\mathbf L_{\cdot,t}\Lambda K^{1/2}\mathbf U_t\mathbf U_t^\intercal\beps_t
=\mathbf L_{\cdot,t}\Lambda\mathbf R_t\eps_t
=\overline L_{\cdot,t}\eps_t.
\end{align*}
Hence the two candidate barycenter paths are equal under this coupling. We verify bicausality of the zero-cost coupling.  Let \(\mathcal F_t^{\beps} :=\sigma(\beps_1,\ldots,\beps_t)\) and \(\mathcal F_t^{\eps}:=\sigma(\eps_1,\ldots,\eps_t)\).  Since \(\eps_s\) is a function of \(\beps_s\), the causal direction from the enlarged filtration to the reduced filtration is immediate.  For the reverse direction, write
\[
\beps_s=\mathbf U_s\eps_s+(I_{Kd}-\mathbf U_s\mathbf U_s^\intercal)\beps_s.
\]
For each \(s\), the two terms on the right are independent centered Gaussian vectors, and the decompositions are independent over different times.  Consequently, conditional on \(\mathcal F_t^\eps\), the additional information in \(\mathcal F_t^{\beps}\) is generated by orthogonal components up to time \(t\), which are independent of \(\eps_{t+1},\ldots,\eps_T\).  Therefore \(\mathcal F_T^\eps\) is conditionally independent of \(\mathcal F_t^{\beps}\) given \(\mathcal F_t^\eps\).  The coupling is bicausal, and since the two paths are equal, its cost is zero.  This proves \eqref{eq:AW-zero-extended-common-noise}.  In particular, \(\overline L\in\mathscr L(T,d)\) since each \(L^k_{\cdot,t}\) has zero block rows before time \(t\), and \(\bG^{\overline a,\overline L}\in\cFG(T,d)\) is an unrestricted barycenter representative.  This establishes (iii).

The implication (iii) \(\Rightarrow\) (iv) is immediate.  It remains to prove necessity.  Assume (iv), and let \(\bG^{b,M}\in\cFG(T,d)\) be an unrestricted barycenter representative.  By the mean contribution in \eqref{eq:unrestricted.value}, necessarily \(b=\widebar a\).  By the distance formula \eqref{eq:AW2.Gaussian} together with Proposition~\ref{prop:decomp} and \eqref{eqn:d.BW}, its objective value is
\begin{equation}
\label{eq:common.noise.necessity.objective}
\sum_{k=1}^K\lambda_k\|a^k-\widebar a\|_2^2
+
\sum_{t=1}^T
\sum_{k=1}^K\lambda_k
\mathrm{dist}_{\BW}^2(A_t^k,M_{\cdot,t}M_{\cdot,t}^\intercal).
\end{equation}
Fix \(t\).  For each \(k\), choose an orthogonal Procrustes optimizer \(\mathbf Q_t^k\in\mathscr O(d)\) such that
\[
\mathrm{dist}_{\BW}^2(A_t^k,M_{\cdot,t}M_{\cdot,t}^\intercal)
=
\|L_{\cdot,t}^k\mathbf Q_t^k-M_{\cdot,t}\|_F^2.
\]
Set
\[
\widebar{M}_{\cdot,t}:=\sum_{k=1}^K\lambda_kL_{\cdot,t}^k\mathbf Q_t^k,
\qquad
\mathbf P_t^{k,\ell}:=\mathbf Q_t^k(\mathbf Q_t^\ell)^\intercal.
\]
Then \(\mathbf P_t\in\mathscr{P}_+(K,d)\) has rank \(d\), and completing the square gives
\begin{align*}
\sum_{k=1}^K\lambda_k
\mathrm{dist}_{\BW}^2(A_t^k,M_{\cdot,t}M_{\cdot,t}^\intercal) 
&\quad =
\sum_{k=1}^K\lambda_k\|L_{\cdot,t}^k\mathbf Q_t^k-M_{\cdot,t}\|_F^2 \\
&\quad =
\sum_{k=1}^K\lambda_k\|L_{\cdot,t}^k\mathbf Q_t^k-\widebar{M}_{\cdot,t}\|_F^2
+
\|M_{\cdot,t}-\widebar{M}_{\cdot,t}\|_F^2 \\
&\quad =
\Phi_t(\mathbf P_t)+\|M_{\cdot,t}-\widebar{M}_{\cdot,t}\|_F^2
\ge
\min_{\widetilde{\mathbf P}_t\in\mathscr{P}_+(K,d)}\Phi_t(\widetilde{\mathbf P}_t).
\end{align*}
The equality with \(\Phi_t(\mathbf P_t)\) is precisely the trace expansion \eqref{eq:Phi_t.expanded}.  Summing the last inequality over \(t\), and comparing with the unrestricted value formula \eqref{eq:unrestricted.value}, shows that equality must hold at every time.  Indeed, the total objective of \(\bG^{\widebar a,M}\) equals the unrestricted barycenter value, while every displayed local gap is nonnegative.  Hence each \(\mathbf P_t\) is a rank-\(d\) minimizer of \(\Phi_t\) over \(\mathscr{P}_+(K,d)\).  Thus, the common-noise input condition holds, proving (i).
\end{proof}

\begin{lemma}[Regularity and common noise]
\label{lem:strict-regularity-implies-common-noise}
Suppose that $L^1,\ldots,L^K$ satisfy Assumption~\ref{cond:uniqueness}, with orthogonal matrices $\mathbf Q_t^k$ as therein.  Then, for every $t\in[T]$, the unique minimizer of $\Phi_t$ over $\mathscr{P}_+(K,d)$ is the matrix $\mathbf P_t^\star$ of \eqref{eq:unique-local-P-star}, with blocks $(\mathbf P_t^\star)^{k,\ell}=\mathbf Q_t^k(\mathbf Q_t^\ell)^\intercal$, and $\rank(\mathbf P_t^\star)=d$.  In particular, Condition~\ref{cond:common-noise.input} holds, and the unique barycenter $\bG^{\widebar L}$ of Theorem~\ref{thm:uniqueness} is the  representative in $\cFG(T, d)$ provided by Theorem~\ref{thm:FG-common-noise-criterion}.
\end{lemma}

\begin{proof}
That $\mathbf P_t^\star$ is the unique minimizer of $\Phi_t$ over $\mathscr{P}_+(K,d)$, with the stated blocks, was established in the first part of the proof of Theorem~\ref{thm:uniqueness}.  By Lemma~\ref{lem:Pk-rank-common-noise}, a matrix with blocks $\mathbf Q_t^k(\mathbf Q_t^\ell)^\intercal$ has rank $d$.  Hence every local problem admits a rank-$d$ minimizer, which is Condition~\ref{cond:common-noise.input}, and Theorem~\ref{thm:FG-common-noise-criterion} applies with these $\mathbf Q_t^k$: the resulting representative has mean zero and factor $\widebar L_{\cdot,t}=\sum_{k=1}^K\lambda_kL^k_{\cdot,t}\mathbf Q_t^k$, which is exactly \eqref{eq:strict-regularity-barycenter-factor}.
\end{proof}

\begin{remark}[Closed form under compatible orientations]
\label{rem:compatible-orientation-closed-form}
The argument in the proof of Theorem~\ref{thm:uniqueness} also gives a useful local closed form under a slightly weaker hypothesis.  Fix \(t\), and suppose there exist \(\mathbf Q_t^1,\ldots,\mathbf Q_t^K\in\mathscr O(d)\) such that
\[
(\mathbf Q_t^k)^\intercal(L_{\cdot,t}^k)^\intercal L_{\cdot,t}^\ell \mathbf Q_t^\ell\succeq0,
\qquad k,\ell\in[K].
\]
Then \eqref{eq:trace-contraction-inequality} shows that
\[
\mathbf P_t^\star=\bigl(\mathbf Q_t^k(\mathbf Q_t^\ell)^\intercal\bigr)_{k,\ell=1}^K
\]
minimizes \(\Phi_t\) over \(\mathscr{P}_+(K,d)\), and
\[
\min_{\mathbf P_t\in\mathscr{P}_+(K,d)}\Phi_t(\mathbf P_t)
=\sum_{k=1}^K\lambda_k\|L_{\cdot,t}^k\|_F^2
-\left\|\sum_{k=1}^K\lambda_kL_{\cdot,t}^k\mathbf Q_t^k\right\|_F^2 .
\]
In particular, Assumption~\ref{cond:uniqueness} implies this condition at every time, with the optimizer unique.
\end{remark}

\subsection{Barycentric interpolation}
\label{sec:two.point}
When \(K=2\), the common-noise input condition holds automatically.  Indeed, the local problem in \eqref{eq:unrestricted.value} amounts, by \eqref{eq:Phi_t.expanded}, to maximizing \(\tr(AC)\) over the off-diagonal block \(C=\mathbf P_t^{2,1}\), which ranges over all contractions, where \(A=(L^1_{\cdot,t})^\intercal L^2_{\cdot,t}\).  The maximum \(\|A\|_*\) is attained at an orthogonal matrix \(C=\mathbf Q\) obtained by extending the partial isometry in the singular value decomposition of \(A\), and the resulting optimizer
\[
\mathbf P_t=\begin{pmatrix} I_d & \mathbf Q^\intercal\\ \mathbf Q & I_d\end{pmatrix}
\]
has rank \(d\).  Thus the enlarged innovation dimension of Theorem~\ref{thm:gaussian.multicausal.optimizer} is not needed for two marginals.  The two-point barycenter is therefore an adapted analogue of displacement interpolation and of constant-speed geodesics in Wasserstein space \cite{AGS08,V08}.  The existence of the filtered Gaussian geodesic used here is a consequence of the Procrustes representation \eqref{eq:Procrustes}.

Let
\[
\bX_0=\bG^{a_0,L_0},
\qquad
\bX_1=\bG^{a_1,L_1}
\]
be elements of \(\cFG(T,d)\).  Fix \(u\in[0,1]\) and \(Q\in\mathscr Q(L_0,L_1)\).  Define
\begin{equation}
\label{eq:two.point.interpolation}
\hat a_u:=(1-u)a_0+ua_1,
\qquad
\hat L_u:=(1-u)L_0+uL_1Q.
\end{equation}

\begin{theorem}[Barycentric interpolation]
\label{thm:two.point.barycenter.filtered.gaussian}
One has
\begin{equation}
\label{eq:two.point.value.filtered.gaussian}
\inf_{\bZ\in\cFG}
\left\{(1-u)\AW_2^2(\bX_0,\bZ)+u\AW_2^2(\bX_1,\bZ)\right\}
=u(1-u)\AW_2^2(\bX_0,\bX_1),
\end{equation}
and $\bG^{\hat a_u,\hat L_u}$ is a minimizer.  Conversely, if $\bZ=\bG^{b,M}$ is a minimizer, then $b=\hat a_u$ and there exist $Q\in\mathscr Q(L_0,L_1)$ and $R\in\mathscr O(T,d)$ such that
\begin{equation}
\label{eq:minimizer.factor.characterization}
M=((1-u)L_0+uL_1Q)R.
\end{equation}
Thus, the set of two-point barycenters in $\FG$ is the collection of equivalence classes generated by \eqref{eq:two.point.interpolation} as $Q$ ranges over $\mathscr Q(L_0,L_1)$.
\end{theorem}

\begin{proof}
We verify the barycenter statement and the minimizer characterization directly.  Let \(\bZ=\bG^{b,M}\).  By \eqref{eq:AW2.Gaussian}, the problem separates into the mean part and the factor part.  The mean part
\[
(1-u)\|a_0-b\|_2^2+u\|a_1-b\|_2^2
\]
is minimized uniquely at \(b=\hat a_u\), with value \(u(1-u)\|a_0-a_1\|_2^2\).

For $u=0$ or $u=1$, the factor statement is immediate. In the remainder of the proof assume $0<u<1$.

For the factor part, we pass to the Procrustes form \eqref{eq:Procrustes} once and for all, so that the adapted distance enters the computation only through Frobenius norms:
\begin{equation}
\label{eq:two.point.procrustes.form}
\begin{aligned}
&(1-u)\mathrm{dist}_{\mathrm{ABW}}^2(L_0,M)
+u\,\mathrm{dist}_{\mathrm{ABW}}^2(L_1,M)\\
&\quad=\inf_{R_0,R_1\in\mathscr O(T,d)}
\Bigl\{(1-u)\|L_0-MR_0\|_F^2
+u\|L_1-MR_1\|_F^2\Bigr\}.
\end{aligned}
\end{equation}
Substitute \(\widetilde M:=MR_0\) and \(Q:=R_1^\intercal R_0\); as \((R_0,R_1)\) ranges over \(\mathscr O(T,d)^2\), the pair \((\widetilde M,Q)\) ranges over \(\{MR:R\in\mathscr O(T,d)\}\times\mathscr O(T,d)\).  Since the Frobenius norm is invariant under right multiplication by an orthogonal matrix,
\[
\|L_1-MR_1\|_F=\|(L_1-MR_1)R_1^\intercal R_0\|_F=\|L_1Q-\widetilde M\|_F,
\]
and completing the square,
\[
(1-u)\|L_0-\widetilde M\|_F^2+u\|L_1Q-\widetilde M\|_F^2
=\|\widetilde M-((1-u)L_0+uL_1Q)\|_F^2
+u(1-u)\|L_0-L_1Q\|_F^2.
\]
Minimizing the right-hand side over \(M\) (equivalently, over \(\widetilde M\), since \((1-u)L_0+uL_1Q\in\mathscr L(T,d)\) is attainable) and then over \(Q\in\mathscr O(T,d)\), the infimum in \eqref{eq:two.point.procrustes.form} over all \(M\) equals
\[
u(1-u)\min_{Q\in\mathscr O(T,d)}\|L_0-L_1Q\|_F^2
=u(1-u)\,\mathrm{dist}_{\mathrm{ABW}}^2(L_0,L_1),
\]
attained exactly when \(Q\in\mathscr Q(L_0,L_1)\) and \(\widetilde M=(1-u)L_0+uL_1Q\).  Adding the mean contribution proves \eqref{eq:two.point.value.filtered.gaussian}.

If \(M\) is a minimizer, equality must hold in the preceding completion of the square for some \(R_0,R_1\).  Hence \(Q=R_1^\intercal R_0\in\mathscr Q(L_0,L_1)\) and \(MR_0=(1-u)L_0+uL_1Q\).  Therefore \(M=((1-u)L_0+uL_1Q)R_0^\intercal\), which is \eqref{eq:minimizer.factor.characterization} with \(R=R_0^\intercal\).
\end{proof}

\begin{corollary}[Minimizing geodesics in \(\FG\)]
\label{cor:barycentric.interpolation.geodesic}
Let \(Q\in\mathscr Q(L_0,L_1)\), and define \(\hat{\bX}_u=\bG^{\hat a_u,\hat L_u}\) by \eqref{eq:two.point.interpolation}.  Then, in the quotient \(\FG\),
\[
\AW_2(\bX_0,\hat{\bX}_u)=u\AW_2(\bX_0,\bX_1),
\qquad
\AW_2(\hat{\bX}_u,\bX_1)=(1-u)\AW_2(\bX_0,\bX_1).
\]
Consequently, \(u\mapsto[\hat{\bX}_u]\) is a constant-speed minimizing geodesic from \([\bX_0]\) to \([\bX_1]\).
\end{corollary}

\begin{proof}
Let \(D=\mathrm{dist}_{\mathrm{ABW}}(L_0,L_1)=\|L_0-L_1Q\|_F\).  Then
\[
\mathrm{dist}_{\mathrm{ABW}}(L_0,\hat L_u)\le\|L_0-\hat L_u\|_F=uD,
\]
and, using the same alignment \(Q\),
\[
\mathrm{dist}_{\mathrm{ABW}}(\hat L_u,L_1)\le\|\hat L_u-L_1Q\|_F=(1-u)D.
\]
The triangle inequality forces equality in both estimates.  The same linear identities hold for the means.  The formula \eqref{eq:AW2.Gaussian} then proves the stated distance identities.
\end{proof}

\subsection{The restricted problem}
\label{sec:restricted.filtered.gaussian.problem}
The results in the remainder of this section concern the constrained problem in which the candidate barycenter is required to lie in the ordinary filtered Gaussian class \(\cFG(T,d)\), or equivalently in the finite-dimensional (quotient) factor space \(\mathscr L(T,d)/\mathscr O(T,d)\).  They are not assertions about arbitrary unrestricted barycenter representatives in \(\FP_2\).  This distinction is essential: the unrestricted problem is solved in Section~\ref{sec:unrestricted-gaussian-barycenters}, while the restricted problem imposes the rank-\(d\) common-noise constraint at every time.

\medskip
After centering the inputs, by \eqref{eq:AW2.Gaussian} the restricted problem reduces to the \emph{restricted factor problem}
\begin{equation}\label{eq:restricted.ABW.factor.problem}
\inf_{L\in\mathscr L(T,d)}\sum_{k=1}^K\lambda_k\mathrm{dist}_{\mathrm{ABW}}^2(L,L^k).
\end{equation}
We first spell out the restricted objective; this formula is the starting point for the constrained problem when the common-noise input condition is not assumed.

\begin{proposition}[Restricted objective decomposition]
\label{prop:FG-restricted-barycenter}
For \(\bY=\bG^{b,M}\in\cFG(T,d)\),
\begin{equation}
\label{eq:restricted.functional.decomposition}
\sum_{k=1}^K\lambda_k\AW_2^2(\bX^k,\bY)
=
\sum_{k=1}^K\lambda_k\|a^k-b\|_2^2+
\sum_{t=1}^T\sum_{k=1}^K\lambda_k\mathrm{dist}_{\BW}^2(L_{\cdot,t}^k(L_{\cdot,t}^k)^\intercal,M_{\cdot,t}M_{\cdot,t}^\intercal).
\end{equation}
Consequently,
\begin{equation}
\label{eq:restricted.value}
\inf_{\bY\in\cFG}\sum_{k=1}^K\lambda_k\AW_2^2(\bX^k,\bY)
=
\sum_{k=1}^K\lambda_k\|a^k-\overline a\|_2^2+
\sum_{t=1}^T\min_{\substack{\mathbf P_t\in\mathscr{P}_+(K,d)\\ \rank(\mathbf P_t)=d}}\Phi_t(\mathbf P_t).
\end{equation}
Equivalently, the second sum on the right-hand side of \eqref{eq:restricted.value} is the infimum of the global trace functional \(\Phi(\mathbf P)\) over those \(\mathbf P\in\mathscr{P}_+(K,T,d)\) whose one-step blocks satisfy \(\rank(\mathbf P_t)=d\) for every \(t\).
In the language of Theorem~\ref{thm:explicit.value}, the restricted problem amounts to imposing the rank constraint \(\rank(\Sigma)\le d\) in the local Bures--Wasserstein barycenter problems \eqref{eq:local.BW.barycenter.identity}.
For fixed \(t\), the local rank-\(d\) infimum admits the equivalent factor form
\begin{equation}
\label{eq:local.restricted.factor.problem}
\min_{\substack{\mathbf P_t\in\mathscr{P}_+(K,d)\\ \rank(\mathbf P_t)=d}}\Phi_t(\mathbf P_t)
=
\inf_{ Q^1,\ldots,Q^K\in\mathscr O(d)}
\left\{
\sum_{k=1}^K\lambda_k\|L_{\cdot,t}^k\|_F^2-
\left\|\sum_{k=1}^K\lambda_kL_{\cdot,t}^k Q^k\right\|_F^2
\right\}.
\end{equation}
For fixed \( Q^1,\ldots, Q^K\), the minimizing block-column is
\begin{equation}
\label{eq:restricted.column.formula}
Y_t=\sum_{k=1}^K\lambda_kL_{\cdot,t}^k Q^k.
\end{equation}
\end{proposition}

\begin{proof}
By \eqref{eq:AW2.Gaussian} and \eqref{eq:ABW},
\begin{align*}
\AW_2^2(\bX^k,\bG^{b,M})
&=\|a^k-b\|_2^2+
\sum_{t=1}^T\bigl(\|L_{\cdot,t}^k\|_F^2+\|M_{\cdot,t}\|_F^2
-2\|(L_{\cdot,t}^k)^\intercal M_{\cdot,t}\|_*\bigr).
\end{align*}
Using \eqref{eqn:d.BW} with \(L_{\cdot,t}^k(L_{\cdot,t}^k)^\intercal\) and \(M_{\cdot,t}M_{\cdot,t}^\intercal\) gives \eqref{eq:restricted.functional.decomposition}.  The mean term is minimized uniquely at \(b=\overline a\).

Fix \(t\), and write \(Y\) for a possible value of \(M_{\cdot,t}\).  Again by \eqref{eqn:d.BW},
\[
\sum_{k=1}^K\lambda_k\mathrm{dist}_{\BW}^2(A_t^k,YY^\intercal)
=
\inf_{ Q^1,\ldots, Q^K\in\mathscr O(d)}
\sum_{k=1}^K\lambda_k\|L_{\cdot,t}^k Q^k-Y\|_F^2.
\]
For fixed \( Q^1,\ldots, Q^K\), completing the square gives
\begin{align*}
\sum_{k=1}^K\lambda_k\|L_{\cdot,t}^k Q^k-Y\|_F^2
&=\left\|Y-\sum_{k=1}^K\lambda_kL_{\cdot,t}^k Q^k\right\|_F^2 \\
&\quad+\sum_{k=1}^K\lambda_k\|L_{\cdot,t}^k\|_F^2
-\left\|\sum_{k=1}^K\lambda_kL_{\cdot,t}^k Q^k\right\|_F^2.
\end{align*}
The first term is minimized uniquely by \eqref{eq:restricted.column.formula}. If \(\mathbf P_t^{k,\ell}= Q^k(Q^\ell)^\intercal\), then \(\mathbf P_t\in\mathscr{P}_+(K,d)\), \(\rank(\mathbf P_t)=d\), and the remaining expression is \(\Phi_t(\mathbf P_t)\). Conversely, Lemma~\ref{lem:Pk-rank-common-noise} shows that every rank-\(d\) element of \(\mathscr{P}_+(K,d)\) has this form. This proves \eqref{eq:local.restricted.factor.problem}. Since each local minimizer has zero block rows before time \(t\), the local block-columns assemble into an element of \(\mathscr L(T,d)\), and summing over \(t\) gives \eqref{eq:restricted.value}.
\end{proof}

\begin{theorem}[Restricted filtered Gaussian barycenters]
\label{thm:restricted-FG-barycenters}
The restricted problem
\[
\inf_{\bY\in\cFG(T,d)}\sum_{k=1}^K\lambda_k\AW_2^2(\bX^k,\bY)
\]
admits a minimizer, and its value is \eqref{eq:restricted.value}.  Moreover, \(\bG^{\widebar a,\widebar L}\in\cFG(T,d)\) is a restricted minimizer if and only if, for every \(t\in[T]\), there exist \( Q_t^1,\ldots,Q_t^K\in\mathscr O(d)\) such that
\begin{equation}
\label{eq:restricted.theorem.column}
\widebar L_{\cdot,t}=\sum_{k=1}^K\lambda_kL_{\cdot,t}^k Q_t^k
\end{equation}
and the common-noise matrix
\begin{equation}
\label{eq:restricted.theorem.P}
\mathbf P_t^{k,\ell}= Q_t^k( Q_t^\ell)^\intercal,
\qquad k,\ell\in[K],
\end{equation}
attains
\[
\min_{\substack{\mathbf P_t\in\mathscr{P}_+(K,d)\\ \rank(\mathbf P_t)=d}}\Phi_t(\mathbf P_t).
\]
A restricted minimizer is an unrestricted adapted barycenter if and only if the matrices \(\mathbf P_t\) in \eqref{eq:restricted.theorem.P} can be chosen to minimize \(\Phi_t\) over all of \(\mathscr{P}_+(K,d)\) for every \(t\).  Equivalently, Condition~\ref{cond:common-noise.input} holds and the restricted minimizer is built from unrestricted local optimizers.
\end{theorem}

\begin{proof}
Proposition~\ref{prop:FG-restricted-barycenter} gives the value formula and reduces the problem to independent local minimizations.  The set
\[
\{\mathbf P\in\mathscr{P}_+(K,d):\rank(\mathbf P)=d\}
\]
is compact, because it is the intersection of the compact set \(\mathscr{P}_+(K,d)\) with the closed condition \(\rank(\mathbf P)\le d\), and every element of \(\mathscr{P}_+(K,d)\) has rank at least \(d\).  Thus, each local rank-\(d\) infimum is attained.  Lemma~\ref{lem:Pk-rank-common-noise} gives the common-noise representation \eqref{eq:restricted.theorem.P}, and the completed-square identity in the proof of Proposition~\ref{prop:FG-restricted-barycenter} gives \eqref{eq:restricted.theorem.column}.  These block-columns assemble into \(\widebar L\in\mathscr L(T,d)\), proving existence and one direction of the characterization.

Conversely, suppose \(\bG^{\widebar a,\widebar L}\) is a restricted minimizer.  The mean must be \(\widebar a\).  By the separation in \eqref{eq:restricted.functional.decomposition}, each block-column \(\widebar L_{\cdot,t}\) minimizes the local factor problem.  Choose Procrustes optimizers \(\mathbf Q_t^k\in\mathscr O(d)\) such that
\[
\mathrm{dist}_{\BW}^2(L_{\cdot,t}^k(L_{\cdot,t}^k)^\intercal,\widebar L_{\cdot,t}\widebar L_{\cdot,t}^\intercal)
=\|L_{\cdot,t}^k\mathbf Q_t^k-\widebar L_{\cdot,t}\|_F^2.
\]
Set \(\widetilde L_{\cdot,t}:=\sum_k\lambda_kL_{\cdot,t}^k Q_t^k\).  Completing the square gives
\begin{align*}
&\sum_{k=1}^K\lambda_k\|L_{\cdot,t}^k Q_t^k-\widebar L_{\cdot,t}\|_F^2 \\
&\quad=
\sum_{k=1}^K\lambda_k\|L_{\cdot,t}^k Q_t^k-\widetilde L_{\cdot,t}\|_F^2
+\|\widebar L_{\cdot,t}-\widetilde L_{\cdot,t}\|_F^2.
\end{align*}
The first term on the right is the value of \(\Phi_t\) at the common-noise matrix \eqref{eq:restricted.theorem.P}; hence it is at least the rank-\(d\) infimum.  The left-hand side is exactly the local minimum.  Therefore the nonnegative second term vanishes, \(\widebar L_{\cdot,t}=\widetilde L_{\cdot,t}\), and \(\mathbf P_t\) attains the rank-\(d\) infimum.

The final assertion follows by comparing \eqref{eq:restricted.value} with the unrestricted value formula \eqref{eq:unrestricted.value}.  Each local rank-\(d\) infimum is at least the unrestricted local infimum.  Hence equality of the restricted and unrestricted values is equivalent to equality at every time, which is precisely the condition that the chosen \(\mathbf P_t\)'s minimize \(\Phi_t\) on all of \(\mathscr{P}_+(K,d)\).  The equivalence with Theorem~\ref{thm:FG-common-noise-criterion} follows from the common-noise criterion.
\end{proof}

\subsection{Fixed-point optimality for the restricted problem}\label{sec:characterization}

The finite-dimensional restricted factor problem has the usual Procrustes optimality condition.  In contrast with the earlier unqualified formulation, this condition is not a general uniqueness theorem.

\begin{proposition}[Restricted Procrustes fixed-point condition]\label{thm:characterization}
Let $\widebar L$ be a minimizer of \eqref{eq:restricted.ABW.factor.problem}.  Then there exist Procrustes optimizers
\[
Q^k\in\mathscr Q(\widebar L,L^k):=\argmin_{Q\in\mathscr O(T,d)}\|\widebar L-L^kQ\|_F^2,
\qquad k=1,\ldots,K,
\]
such that
\begin{equation}\label{eq:FPE}
\widebar L=\sum_{k=1}^K\lambda_kL^kQ^k.
\end{equation}
Conversely, if $\widebar L$ and $(Q^1,\ldots,Q^K)$ solve the joint global minimization problem
\[
\inf_{L\in\mathscr L(T,d),\,Q^1,\ldots,Q^K\in\mathscr O(T,d)}
\sum_{k=1}^K\lambda_k\|L-L^kQ^k\|_F^2,
\]
with each $Q^k\in\mathscr Q(\widebar L,L^k)$, then $\widebar L$ is a restricted minimizer.
\end{proposition}

\begin{proof}
Choose $Q^k\in\mathscr Q(\widebar L,L^k)$ for a restricted minimizer $\widebar L$.  For these fixed orthogonal matrices, the function
\[
L\mapsto\sum_{k=1}^K\lambda_k\|L-L^kQ^k\|_F^2
\]
is strictly convex and quadratic on the linear space $\mathscr L(T,d)$, with unique minimizer $\sum_{k=1}^K\lambda_kL^kQ^k$.  Since $\widebar L$ minimizes the full objective, it must minimize this fixed-$Q$ quadratic problem, yielding \eqref{eq:FPE}.  The converse is immediate from the stated joint global optimality.
\end{proof}

\subsection{Regularity of restricted minimizers}\label{sec:regularity}

In this subsection we show that restricted adapted Bures--Wasserstein barycenters of regular factors are themselves regular. Recall from Section~\ref{sec:AOT} that $L\in\mathscr L^{\mathrm{reg}}(T,d)$ if and only if $(L^\intercal L)_{t,t}=L_{\cdot,t}^\intercal L_{\cdot,t}$ is positive definite for all $t=1,\ldots,T$, equivalently each block column $L_{\cdot,t}$ has full column rank $d$.

\begin{theorem}[Regularity of restricted barycenters]\label{thm:regularity}
Let $L^1,\ldots,L^K\in\mathscr L^{\mathrm{reg}}(T,d)$ and let $\lambda_1,\ldots,\lambda_K>0$ with $\sum_{k=1}^K\lambda_k=1$. Then every restricted minimizer $M$ of \eqref{eq:restricted.ABW.factor.problem} lies in $\mathscr L^{\mathrm{reg}}(T,d)$.
\end{theorem}

\begin{proof}
By Proposition~\ref{prop:decomp}, the restricted barycenter problem decomposes column by column. For $t\in[T]$, let
\[
\mathcal \mathbf C_t:=\left\{
Y=(Y_1^\intercal,\ldots,Y_T^\intercal)^\intercal\in\bR^{Td\times d}
:
Y_s=0\text{ for }s<t
\right\}
\]
be the space of admissible $t$-th block columns. The $t$-th block column of any restricted minimizer is therefore a minimizer of the block-column barycenter functional
\begin{equation}\label{eq:block_col_bary}
F_t(Y):=\sum_{k=1}^K \lambda_k \min_{Q^k\in\mathscr O(d)}
\|Y-L_{\cdot,t}^kQ^k\|_F^2,
\qquad Y\in\mathcal \mathbf C_t.
\end{equation}
Since $L^k\in\mathscr L^{\mathrm{reg}}(T,d)$, each block column $L_{\cdot,t}^k$ has full column rank $d$. We need to show that the minimizer $M_{\cdot,t}$ of \eqref{eq:block_col_bary} has also full column rank $d$ for every $t$. In order to do so, we show by a contradiction argument that if the minimizer $M_{\cdot,t}$ does not have full column rank, then we can construct a competitor for \eqref{eq:block_col_bary} such that $F_t(M_{\cdot,t}^\delta)<F_t(M_{\cdot,t})$.

\medskip
Suppose that $\rank(M_{\cdot,t})<d$ for some $t$. Then there exists a unit vector $v\in\bR^d$ with $M_{\cdot,t}v=0$. By the fixed-point equation \eqref{eq:FPE} of Proposition~\ref{thm:characterization}, the restricted minimizer satisfies
\[
M=\sum_{k=1}^K\lambda_k L^kQ^k
\]
for some $Q^k\in\mathscr Q(M,L^k)$. Writing $Q^k_t\in\mathscr O(d)$ for the $t$-th diagonal block of $Q^k$, this gives
\[
M_{\cdot,t}=\sum_{k=1}^K\lambda_kL_{\cdot,t}^kQ^k_t.
\]
Thus,
\begin{equation}\label{eq:reg_kernel}
0=M_{\cdot,t}v=\sum_{k=1}^K\lambda_kL_{\cdot,t}^kQ^{k}_tv.
\end{equation}
We now construct a perturbation that strictly decreases the objective \eqref{eq:block_col_bary}. For each $k$, set
\[
B^k:=(L_{\cdot,t}^k)^\intercal M_{\cdot,t}\in\bR^{d\times d}.
\]
Since $M_{\cdot,t}v=0$, we have $B^kv=0$ for all $k$. Write
\[
v^\perp:=\{u\in\bR^d:u^\intercal v=0\}
\]
for the orthogonal complement of $\operatorname{Span}(v)$ and
\[
B^k|_{v^\perp}:v^\perp\to\bR^d
\]
for the restriction of $B^k$ to $v^\perp$. Since $\dim(v^\perp)=d-1$, we have
\[
\rank(B^k|_{v^\perp})\le d-1
\]
for every $k$.

\medskip
{\sc Step 1.} \textit{Choice of the perturbation direction.}
Fix $k=1$. Since $\rank(B^1|_{v^\perp})\le d-1$, the orthogonal complement $\Image(B^1|_{v^\perp})^\perp$ in $\bR^d$ is nontrivial. Let $w\in\Image(B^1|_{v^\perp})^\perp$ be a unit vector. Since $L_{\cdot,t}^1$ has full column rank $d$, the matrix
\[
(L_{\cdot,t}^1)^\intercal L_{\cdot,t}^1\in\bR^{d\times d}
\]
is positive definite and thus invertible. Define
\[
h:=L_{\cdot,t}^1\bigl((L_{\cdot,t}^1)^\intercal L_{\cdot,t}^1\bigr)^{-1}w
\in\bR^{Td}.
\]
The vector $h$ has zero blocks before time $t$, so $hv^\intercal\in\mathcal \mathbf C_t$. Moreover, $(L_{\cdot,t}^1)^\intercal h=w$, and by the choice of $w$,
\begin{equation}\label{eq:c1perp}
c_{1,\perp}
:=
\textup{projection of }(L_{\cdot,t}^1)^\intercal h
\textup{ onto }\Image(B^1|_{v^\perp})^\perp
=w\ne0.
\end{equation}
For $k=2,\ldots,K$, write $c_{k,\perp}$ for the projection of $(L_{\cdot,t}^k)^\intercal h$ onto $\Image(B^k|_{v^\perp})^\perp$.

\medskip
{\sc Step 2.} \textit{The perturbation.}
Consider
\[
M_{\cdot,t}^\delta:=M_{\cdot,t}+\delta hv^\intercal
\]
for $\delta>0$. Since $hv^\intercal\in\mathcal \mathbf C_t$, also $M_{\cdot,t}^\delta\in\mathcal \mathbf C_t$. The Procrustes distance for each $k$ satisfies
\[
\min_{\mathbf Q^k\in\mathscr O(d)}
\|M_{\cdot,t}^\delta-L_{\cdot,t}^k\mathbf Q^k\|_F^2
=
\|M_{\cdot,t}^\delta\|_F^2+\|L_{\cdot,t}^k\|_F^2
-2\tr(S_{k,t}^\delta),
\]
where $S_{k,t}^\delta$ is the diagonal matrix of singular values of
\[
(L_{\cdot,t}^k)^\intercal M_{\cdot,t}^\delta.
\]
We now turn to
\[
(L_{\cdot,t}^k)^\intercal M_{\cdot,t}^\delta
=
B^k+\delta\bigl((L_{\cdot,t}^k)^\intercal h\bigr)v^\intercal.
\]
Choose a basis of $\bR^d$ in which $v=e_d$. In this basis,
\[
B^k=(\hat B^k\mid 0),
\qquad
\hat B^k\in\bR^{d\times(d-1)},
\]
and the perturbation gives
\[
B^k+\delta(L_{\cdot,t}^k)^\intercal h\,e_d^\intercal
=
(\hat B^k\mid \delta(L_{\cdot,t}^k)^\intercal h).
\]
Let $c_k:=(L_{\cdot,t}^k)^\intercal h$. The directional derivative formula for the nuclear norm, applied at $(\hat B^k\mid0)$ in the direction $(0\mid c_k)$, gives
\begin{equation}\label{eq:nuclear_expansion}
\tr(S_{k,t}^\delta)
=\tr(S_{k,t})
+\delta\bigl\|\operatorname{proj}_{\Image(\hat B^k)^\perp}c_k\bigr\|_2
+o(\delta)
=\tr(S_{k,t})+\delta\|c_{k,\perp}\|_2+o(\delta)
\end{equation}
as $\delta\to0^+$, where $S_{k,t}$ is the diagonal matrix of singular values of $(L_{\cdot,t}^k)^\intercal M_{\cdot,t}$. Indeed, in a singular value decomposition of $(\hat B^k\mid0)$, the nonzero right singular vectors have zero last coordinate, so the smooth part of the directional derivative vanishes; the null-space contribution is precisely the norm of the component of $c_k$ orthogonal to $\Image(\hat B^k)$.

\medskip
{\sc Step 3.} \textit{Conclusion.}
Using $M_{\cdot,t}v=0$, we have
\[
\|M_{\cdot,t}^\delta\|_F^2
=
\|M_{\cdot,t}\|_F^2+\delta^2\|h\|_2^2.
\]
Summing over $k$ yields
\begin{align*}
\sum_{k=1}^K \lambda_k
\min_{Q^k}
\|M_{\cdot,t}^\delta-L_{\cdot,t}^kQ^k\|_F^2
&=
\sum_{k=1}^K \lambda_k
\min_{Q^k}
\|M_{\cdot,t}-L_{\cdot,t}^kQ^k\|_F^2
\\
&\quad
+\delta^2\|h\|_2^2
-2\delta\sum_{k=1}^K\lambda_k\|c_{k,\perp}\|_2
+o(\delta).
\end{align*}
By \eqref{eq:c1perp}, $\|c_{1,\perp}\|_2=\|w\|_2=1>0$. Since $\lambda_1>0$, the term
\[
-2\delta\sum_{k=1}^K\lambda_k\|c_{k,\perp}\|_2
\le
-2\delta\lambda_1
\]
is strictly negative and of order $\delta$, dominating $\delta^2\|h\|_2^2$ for sufficiently small $\delta>0$. This gives
\[
\sum_{k=1}^K \lambda_k
\min_{Q^k}
\|M_{\cdot,t}^\delta-L_{\cdot,t}^kQ^k\|_F^2
<
\sum_{k=1}^K \lambda_k
\min_{Q^k}
\|M_{\cdot,t}-L_{\cdot,t}^kQ^k\|_F^2,
\]
contradicting the minimality of $M_{\cdot,t}$.

Therefore every block column $M_{\cdot,t}$ has full column rank $d$. Equivalently,
\[
(M^\intercal M)_{t,t}=M_{\cdot,t}^\intercal M_{\cdot,t}
\]
is positive definite for every $t=1,\ldots,T$. Hence $M\in\mathscr L^{\mathrm{reg}}(T,d)$.
\end{proof}

\begin{corollary}[Conditional uniqueness]\label{cor:conditional.uniqueness}
Assume that, for every \(t\in[T]\), the local Bures--Wasserstein barycenter set
\[
\argmin_{A\in\mathscr S_+(Td)}
\sum_{k=1}^K\lambda_k \mathrm{dist}_{\BW}^2(A,A_t^k)
\]
is a singleton \(\{\Sigma_t^\star\}\), and that each \(\Sigma_t^\star\) has rank at most \(d\).  Then the restricted factor barycenter is unique modulo \(\mathscr O(T,d)\).  Without uniqueness of these local problems, uniqueness of the restricted barycenter need not hold.
\end{corollary}

\begin{proof}
By the onto correspondence in Theorem~\ref{thm:explicit.value}, each $\Sigma_t^\star$ is supported on the admissible future-row subspace, because every input column covariance $A_t^k$ is supported there. Since $\rank(\Sigma_t^\star)\le d$, choose a factor $\overline L_{\cdot,t}\in\bR^{Td\times d}$ with zero block rows before time $t$ and $\overline L_{\cdot,t}\overline L_{\cdot,t}^{\intercal}=\Sigma_t^\star$, and assemble these columns into a block-lower triangular factor $\overline L$. Proposition~\ref{prop:decomp} shows that $\overline L$ minimizes the restricted objective. If $L$ and $M$ have the same column covariances $L_{\cdot,t}L_{\cdot,t}^{\intercal}=M_{\cdot,t}M_{\cdot,t}^{\intercal}$ for all $t$, then for each $t$ there is $Q_t\in\mathscr O(d)$ such that $L_{\cdot,t}=M_{\cdot,t}Q_t$; hence $L=MQ$ with $Q=\diag(Q_1,\ldots,Q_T)$. This gives uniqueness modulo $\mathscr O(T,d)$.
\end{proof}

\begin{remark}
By Theorem~\ref{thm:explicit.value}, the local Bures--Wasserstein barycenter sets appearing in Corollary~\ref{cor:conditional.uniqueness} are exactly the minimizers of the unrestricted local problems \eqref{eq:local.BW.barycenter.identity}.  The rank assumption of the corollary is therefore precisely Condition~\ref{cond:common-noise.input} combined with uniqueness of the local barycenters.
\end{remark}

\section{Computation of adapted barycenters}
\label{sec:computation}
In this section we turn the structural results of Sections~\ref{sec:unrestricted-gaussian-barycenters} and~\ref{sec:FG-common-noise} into algorithms.  By \eqref{eq:unrestricted.value}, the mean contribution is computed separately, and we work with the centered inputs $\bG^{L^k}$ throughout.  We first recall the alternating Procrustes iteration for the restricted problem, which is based on the fixed-point condition \eqref{eq:FPE}.  We then show that the unrestricted local problems in \eqref{eq:unrestricted.value} are semidefinite programs, derive an optimality certificate from conic duality, and extend the restricted iteration to an alternating algorithm for the unrestricted problem in which the orthogonal alignments are replaced by rectangular matrices with orthonormal rows.
\subsection{The restricted problem: a Procrustes fixed-point iteration}
\label{sec:computation.restricted}

The fixed-point condition \eqref{eq:FPE} of Proposition~\ref{thm:characterization} suggests an alternating minimization scheme for the restricted factor problem \eqref{eq:restricted.ABW.factor.problem}: given a candidate factor, align each input to it by solving the Procrustes problems defining $\mathscr Q(\cdot\,,\cdot)$, and then update the candidate to the weighted average of the aligned inputs.  By \eqref{eq:Procrustes}, the alignment step decouples over the time-diagonal blocks and is solved by singular value decompositions: for $A=((L^k)^\intercal L)_{t,t}$ with singular value decomposition $A=USV^\intercal$, the block $(O_k)_t=UV^\intercal$ maximizes $\tr(O_t^\intercal A)$ over $\mathscr O(d)$, and the maximizer is unique whenever $A$ is nonsingular.  The procedure is summarized in Algorithm~\ref{alg:ABW}.

\begin{algorithm}[H]
\caption{Restricted barycenter via alternating Procrustes minimization}\label{alg:ABW}
\begin{algorithmic}[1]
\REQUIRE Factors $L^1, \ldots, L^K \in \mathscr L(T,d)$, weights $\lambda_1, \ldots, \lambda_K > 0$, tolerance $\delta > 0$
\ENSURE A restricted minimizer candidate $\widebar L \in \mathscr L(T,d)$
\STATE Initialize $L^{(0)} \leftarrow \sum_{k=1}^K \lambda_k L^k$, $j \leftarrow 0$
\REPEAT
    \FOR{$k = 1, \ldots, K$}
        \FOR{$t = 1, \ldots, T$}
            \STATE Compute the SVD $((L^k)^\intercal L^{(j)})_{t,t} = U_t^k S_t^k (V_t^k)^\intercal$
            \STATE Set $(Q_t^k)^{(j)}\leftarrow U_t^k (V_t^k)^\intercal$
        \ENDFOR
    \ENDFOR
    \STATE Update $L^{(j+1)} \leftarrow \sum_{k=1}^K \lambda_k L^k (Q^k)^{(j)}$, $j \leftarrow j + 1$
\UNTIL{$\|L^{(j)} - L^{(j-1)}\|_F < \delta$}
\RETURN $\widebar L \leftarrow L^{(j)}$
\end{algorithmic}
\end{algorithm}

Each iteration of Algorithm~\ref{alg:ABW} consists of two exact block minimizations of
\[
F(L;Q^1,\ldots,Q^K):=\sum_{k=1}^K\lambda_k\|L-L^kQ^k\|_F^2.
\]
Let $G(L):=\min_{Q^1,\ldots,Q^K}F(L;Q^1,\ldots,Q^K)$ denote the restricted factor objective. Since the alignment step is exact and $L^{(j+1)}$ is the weighted average of the aligned factors,
\[
G(L^{(j)})-G(L^{(j+1)})
\ge F(L^{(j)};Q^{(j)})-F(L^{(j+1)};Q^{(j)})
=\|L^{(j+1)}-L^{(j)}\|_F^2.
\]
Thus the objective values are non-increasing and convergent, and the increments vanish. The iterates are bounded because each update is a weighted average of fixed factors multiplied by orthogonal matrices. For any convergent subsequence, compactness of the orthogonal groups gives a further subsequence of alignments converging to Procrustes optimizers at the limit; the limit then satisfies \eqref{eq:FPE}. Such a fixed point need not be a global restricted minimizer, even when the global restricted barycenter is unique.

Accordingly, an output of Algorithm~\ref{alg:ABW} represents an unrestricted adapted barycenter only after global optimality for the restricted problem has been verified and its associated common-noise matrices are shown to minimize each $\Phi_t$ over all of $\mathscr P_+(K,d)$. When the common-noise condition fails, as in Example~\ref{ex:barycenter.not.in.FG}, no restricted method can attain the unrestricted value, and a different scheme is needed. This is the subject of the remainder of this section.

\subsection{The unrestricted problem as a semidefinite program}
\label{sec:computation.unrestricted}

The key observation is that the one-step trace functional $\Phi_t$ from \eqref{eq:Phi_t} is \emph{affine} in $\mathbf P_t$.  Writing
\begin{equation}
\label{eq:sdp.cost.matrix}
\mathbf C_t:=\Lambda\mathbf L_{\cdot,t}^\intercal\mathbf L_{\cdot,t}\Lambda\in\mathscr S_+(Kd),
\end{equation}
we have $\Phi_t(\mathbf P_t)=\sum_{k=1}^K\lambda_k\|L_{\cdot,t}^k\|_F^2-\tr(\mathbf C_t\mathbf P_t)$.  Since $\mathscr{P}_+(K,d)$ is the intersection of the positive semidefinite cone with the affine constraints $\mathbf P^{k,k}=I_d$, the local problem
\begin{equation}
\label{eq:local.sdp}
\min_{\mathbf P_t\in\mathscr{P}_+(K,d)}\Phi_t(\mathbf P_t)
=
\sum_{k=1}^K\lambda_k\|L_{\cdot,t}^k\|_F^2
-
\max_{\mathbf P_t\in\mathscr{P}_+(K,d)}\tr(\mathbf C_t\mathbf P_t)
\end{equation}
is a semidefinite program, which can be solved to arbitrary accuracy in polynomial time by interior-point methods; see \cite{BV04}.  In view of \eqref{eq:local.restricted.factor.problem} and Lemma~\ref{lem:Pk-rank-common-noise}, the restricted local problem is the same maximization with the additional constraint $\rank(\mathbf P_t)=d$. In this language, Condition~\ref{cond:common-noise.input} states precisely that the relaxation is tight at every time $t$.

\medskip
The unrestricted local problem also admits a variational form that exactly parallels \eqref{eq:local.restricted.factor.problem}, with the orthogonal matrices $ Q^k\in\mathscr O(d)$ replaced by rectangular matrices with orthonormal rows.

\begin{proposition}
\label{prop:sdp.factorized}
For every $t\in[T]$,
\begin{equation}
\label{eq:local.sdp.factorized}
\min_{\mathbf P_t\in\mathscr{P}_+(K,d)}\Phi_t(\mathbf P_t)
=
\min
\left\{
\sum_{k=1}^K\lambda_k\bigl\|L_{\cdot,t}^k\mathbf M^k-\mathbf Y\bigr\|_F^2
\right\},
\end{equation}
where the minimum on the right-hand side is taken over
\[
\mathbf Y\in\bR^{Td\times Kd},\qquad
\mathbf M^k\in\bR^{d\times Kd},\qquad
\mathbf M^k(\mathbf M^k)^\intercal=I_d,\qquad k\in[K].
\]
The two problems are related by $\mathbf P_t=\mathbf M\mathbf M^\intercal$, where $\mathbf M\in\bR^{Kd\times Kd}$ stacks the block rows $\mathbf M^1,\ldots,\mathbf M^K$; this correspondence is onto $\mathscr{P}_+(K,d)$.  For fixed $\mathbf M^1,\ldots,\mathbf M^K$, the inner minimum over $\mathbf Y$ is attained uniquely at
\begin{equation}
\label{eq:sdp.Y.update}
\mathbf Y=\sum_{k=1}^K\lambda_kL_{\cdot,t}^k\mathbf M^k=\mathbf L_{\cdot,t}\Lambda\mathbf M,
\end{equation}
which has zero block rows before time $t$.  Requiring in addition that the row spaces of $\mathbf M^1,\ldots,\mathbf M^K$ coincide with a fixed $d$-dimensional subspace, equivalently $\rank(\mathbf P_t)=d$, recovers the restricted local problem \eqref{eq:local.restricted.factor.problem}.
\end{proposition}

\begin{proof}
Let $\mathbf M^1,\ldots,\mathbf M^K$ have orthonormal rows and set $\mathbf P_t:=\mathbf M\mathbf M^\intercal$, so that $\mathbf P_t^{k,\ell}=\mathbf M^k(\mathbf M^\ell)^\intercal$ and $\mathbf P_t^{k,k}=I_d$; thus $\mathbf P_t\in\mathscr{P}_+(K,d)$.  Conversely, if $\mathbf P_t\in\mathscr{P}_+(K,d)$ and $\mathbf M\in\bR^{Kd\times Kd}$ is any matrix with $\mathbf M\mathbf M^\intercal=\mathbf P_t$, for instance $\mathbf M=\mathbf P_t^{1/2}$, then its block rows satisfy $\mathbf M^k(\mathbf M^k)^\intercal=\mathbf P_t^{k,k}=I_d$; hence the correspondence is onto.  For fixed $\mathbf M^1,\ldots,\mathbf M^K$, completing the square gives
\[
\sum_{k=1}^K\lambda_k\|L_{\cdot,t}^k\mathbf M^k-\mathbf Y\|_F^2
=
\Bigl\|\mathbf Y-\sum_{k=1}^K\lambda_kL_{\cdot,t}^k\mathbf M^k\Bigr\|_F^2
+
\sum_{k=1}^K\lambda_k\|L_{\cdot,t}^k\|_F^2
-
\Bigl\|\sum_{k=1}^K\lambda_kL_{\cdot,t}^k\mathbf M^k\Bigr\|_F^2,
\]
where we used $\|L_{\cdot,t}^k\mathbf M^k\|_F^2=\tr(L_{\cdot,t}^k\mathbf M^k(\mathbf M^k)^\intercal(L_{\cdot,t}^k)^\intercal)=\|L_{\cdot,t}^k\|_F^2$.  The first term vanishes exactly at \eqref{eq:sdp.Y.update}, and by \eqref{eq:Phi_t.expanded},
\[
\Bigl\|\sum_{k=1}^K\lambda_kL_{\cdot,t}^k\mathbf M^k\Bigr\|_F^2
=
\sum_{k,\ell=1}^K\lambda_k\lambda_\ell
\tr\bigl((L_{\cdot,t}^k)^\intercal L_{\cdot,t}^\ell\,\mathbf M^\ell(\mathbf M^k)^\intercal\bigr)
=\tr(\mathbf C_t\mathbf P_t).
\]
Hence the inner minimum equals $\Phi_t(\mathbf P_t)$, and minimizing over the $\mathbf M^k$ is equivalent to minimizing $\Phi_t$ over $\mathscr{P}_+(K,d)$ by the onto correspondence.  Since each $L_{\cdot,t}^k$ has zero block rows before time $t$, so does the optimal $\mathbf Y$.  Finally, if the row spaces of all $\mathbf M^k$ coincide with a $d$-dimensional subspace $V\subseteq\bR^{Kd}$, choose an orthonormal basis of $V$ and write $\mathbf M^k=\mathbf Q^k B$ with $\mathbf Q^k\in\mathscr O(d)$ and $B\in\bR^{d\times Kd}$ the basis matrix; then $\mathbf P_t^{k,\ell}=\mathbf Q^k(\mathbf Q^\ell)^\intercal$ and $\rank(\mathbf P_t)=d$, and conversely every rank-$d$ element of $\mathscr{P}_+(K,d)$ arises this way by Lemma~\ref{lem:Pk-rank-common-noise}.  Substituting into \eqref{eq:local.sdp.factorized} and absorbing $B$ into $\mathbf Y$ recovers \eqref{eq:local.restricted.factor.problem}.
\end{proof}

The convexity of \eqref{eq:local.sdp} yields a global optimality certificate; this should be contrasted with the restricted problem, where no such certificate is available due to the nonconvex rank constraint.

\begin{proposition}[Optimality certificate]
\label{prop:sdp.certificate}
Let $t\in[T]$ and $\mathbf P_t\in\mathscr{P}_+(K,d)$. Suppose there exist symmetric matrices $Z_t^1,\ldots,Z_t^K\in\bR^{d\times d}$ such that
\begin{equation}
\label{eq:sdp.certificate}
\mathbf S_t:=\diag(Z_t^1,\ldots,Z_t^K)-\mathbf C_t\succeq0
\qquad\text{and}\qquad
\mathbf S_t\mathbf P_t=0.
\end{equation}
Then $\mathbf P_t$ minimizes $\Phi_t$ over $\mathscr{P}_+(K,d)$. Conversely, every minimizer of $\Phi_t$ over $\mathscr{P}_+(K,d)$ admits a certificate \eqref{eq:sdp.certificate}.
\end{proposition}

\begin{proof}
Suppose \eqref{eq:sdp.certificate} holds and let $\mathbf P_t'\in\mathscr{P}_+(K,d)$ be arbitrary. Since the diagonal blocks of every element of $\mathscr{P}_+(K,d)$ are $I_d$,
\[
\tr(\mathbf C_t\mathbf P_t')
=\tr\bigl(\diag(Z_t^1,\ldots,Z_t^K)\mathbf P_t'\bigr)-\tr(\mathbf S_t\mathbf P_t')
=\sum_{k=1}^K\tr(Z_t^k)-\tr(\mathbf S_t\mathbf P_t')
\le\sum_{k=1}^K\tr(Z_t^k),
\]
because $\tr(\mathbf S_t\mathbf P_t')\ge0$ for positive semidefinite $\mathbf S_t$ and $\mathbf P_t'$. On the other hand, $\mathbf S_t\mathbf P_t=0$ gives $\tr(\mathbf C_t\mathbf P_t)=\sum_k\tr(Z_t^k)$. Hence $\tr(\mathbf C_t\mathbf P_t)\ge\tr(\mathbf C_t\mathbf P_t')$ for all feasible $\mathbf P_t'$, so $\mathbf P_t$ minimizes $\Phi_t$. For the converse, note that \eqref{eq:local.sdp} is a semidefinite program whose feasible set contains the positive definite matrix $I_{Kd}$. By Slater's condition, strong conic duality holds, the dual is attained, and complementary slackness is exactly \eqref{eq:sdp.certificate}; see \cite[Section~5.9]{BV04}.
\end{proof}

\begin{remark}[Relation to classical Bures--Wasserstein barycenters]
\label{rem:classical.BW.solvers}
By Theorem~\ref{thm:explicit.value}, solving the local semidefinite program \eqref{eq:local.sdp} is equivalent to computing a classical Bures--Wasserstein barycenter of the column covariances $A_t^1,\ldots,A_t^K$, and the optimal value is available in the explicit form \eqref{eq:explicit.local.value} once any barycenter covariance $\widebar\Sigma_t$ is known.  Note, however, that here $\rank(A_t^k)\le d$, so for $t<T$ all marginals are degenerate and the classical fixed-point theory for the barycenter equation \cite{AC11,alvarez2016fixed}, which requires nonsingular covariances, does not apply directly; the semidefinite formulation \eqref{eq:local.sdp} and Algorithm~\ref{alg:unrestricted} remain applicable without any nondegeneracy assumption.  Conversely, at a certified fixed point of Algorithm~\ref{alg:unrestricted} the matrix $\mathbf Y_t\mathbf Y_t^\intercal$ is a Bures--Wasserstein barycenter covariance of $A_t^1,\ldots,A_t^K$ by Theorem~\ref{thm:explicit.value}, so the algorithm doubles as a method for computing classical barycenters of degenerate Gaussian measures.
\end{remark}

\subsection{An alternating algorithm for the unrestricted problem}
\label{sec:computation.alternating}

While \eqref{eq:local.sdp} can be handed to a general-purpose semidefinite solver, the factorized form \eqref{eq:local.sdp.factorized} leads to a lightweight alternating scheme that exactly parallels Algorithm~\ref{alg:ABW}. For fixed $\mathbf M^1,\ldots,\mathbf M^K$, the optimal $\mathbf Y$ is the weighted average \eqref{eq:sdp.Y.update}. For fixed $\mathbf Y$, the problems over the individual $\mathbf M^k$ decouple, and
\[
\argmin_{\mathbf M\mathbf M^\intercal=I_d}\|L_{\cdot,t}^k\mathbf M-\mathbf Y\|_F^2
=
\argmax_{\mathbf M\mathbf M^\intercal=I_d}\tr\bigl(\mathbf M^\intercal(L_{\cdot,t}^k)^\intercal\mathbf Y\bigr).
\]

If $(L_{\cdot,t}^k)^\intercal\mathbf Y=USV^\intercal$ is a thin singular value decomposition with $U\in\mathscr O(d)$, $V\in\bR^{Kd\times d}$, and $V^\intercal V=I_d$, completing zero singular directions orthonormally when necessary, then $\mathbf M=UV^\intercal$ is a maximizer. It is unique whenever $(L_{\cdot,t}^k)^\intercal\mathbf Y$ has full row rank $d$. Note that the iterate $\mathbf Y$ is precisely the candidate block-column of the ensemble barycenter factor: at a global optimum, $\mathbf Y=\mathbf L_{\cdot,t}\Lambda\mathbf M=\overline{\mathbf L}_{\cdot,t}$ as in \eqref{eq:extended.barycenter.factor}, where $\mathbf M$ is a general factor of $\mathbf P_t^\star$ rather than the symmetric square root. By Lemma~\ref{lem:ensemble}, the resulting ensemble processes are $\AW_2$-equivalent. The procedure is summarized in Algorithm~\ref{alg:unrestricted}.

\begin{algorithm}[H]
\caption{Unrestricted barycenter via alternating minimization}\label{alg:unrestricted}
\begin{algorithmic}[1]
\REQUIRE Factors $L^1, \ldots, L^K \in \mathscr L(T,d)$, weights $\lambda_1, \ldots, \lambda_K > 0$, tolerance $\delta > 0$
\ENSURE Local optimizer candidates $\mathbf P_t \in \mathscr{P}_+(K,d)$ and an ensemble factor $\overline{\mathbf L} \in \mathscr L_K(T,d)$
\FOR{$t = 1, \ldots, T$}
    \STATE Initialize $\mathbf M_t^{k,(0)} \leftarrow$ the $k$-th block row of $I_{Kd}$ for $k\in[K]$, $\mathbf Y_t^{(0)} \leftarrow \sum_{k=1}^K \lambda_k L_{\cdot,t}^k \mathbf M_t^{k,(0)}$, $j \leftarrow 0$
    \REPEAT
        \FOR{$k = 1, \ldots, K$}
            \STATE Compute a thin SVD $(L_{\cdot,t}^k)^\intercal \mathbf Y_t^{(j)} = U_t^k S_t^k (V_t^k)^\intercal$, completing zero singular directions orthonormally
            \STATE Set $\mathbf M_t^{k,(j+1)} \leftarrow U_t^k (V_t^k)^\intercal$
        \ENDFOR
        \STATE Update $\mathbf Y_t^{(j+1)} \leftarrow \sum_{k=1}^K \lambda_k L_{\cdot,t}^k \mathbf M_t^{k,(j+1)}$, $j \leftarrow j + 1$
    \UNTIL{$\|\mathbf Y_t^{(j)} - \mathbf Y_t^{(j-1)}\|_F < \delta$}
    \STATE Set $\mathbf P_t \leftarrow \mathbf M_t\mathbf M_t^\intercal$ and $\overline{\mathbf L}_{\cdot,t} \leftarrow \mathbf Y_t^{(j)}$, where $\mathbf M_t$ stacks $\mathbf M_t^{1,(j)},\ldots,\mathbf M_t^{K,(j)}$
\ENDFOR
\RETURN $(\mathbf P_1,\ldots,\mathbf P_T)$ and $\overline{\mathbf L}$
\end{algorithmic}
\end{algorithm}
As for Algorithm~\ref{alg:ABW}, each sweep of Algorithm~\ref{alg:unrestricted} performs two exact block minimizations of the objective in \eqref{eq:local.sdp.factorized}. Hence the objective values are non-increasing and convergent; the exact quadratic minimization over $\mathbf Y_t$ also implies $\|\mathbf Y_t^{(j+1)}-\mathbf Y_t^{(j)}\|_F\to0$. Compactness of the row-Stiefel factors shows that every convergent subsequence has a further subsequence whose limiting factors and candidate satisfy
\begin{equation}
\label{eq:FPE.unrestricted}
\mathbf Y_t=\sum_{k=1}^K\lambda_kL_{\cdot,t}^k\mathbf M_t^k,
\qquad
\mathbf M_t^k\in\argmax_{\mathbf M(\mathbf M)^\intercal=I_d}
\tr\bigl(\mathbf M^\intercal(L_{\cdot,t}^k)^\intercal\mathbf Y_t\bigr),
\quad k\in[K],
\end{equation}

which is the exact analogue of \eqref{eq:FPE}.  Since the factorized problem is nonconvex, a fixed point need not be a global minimizer.  The next lemma shows, however, that at every fixed point the complementarity condition in \eqref{eq:sdp.certificate} holds automatically for a canonical dual candidate, so that global optimality reduces to a single positive semidefiniteness check. In Figure~\ref{fig:algorithm-convergence}, we apply both algorithms to Example~\ref{ex:barycenter.not.in.FG}. For both the unrestricted and restricted problems the objective values converge to the optimal values, and the algorithms behave similarly.

\begin{figure}[H]
\centering
\includegraphics[width=0.72\textwidth]{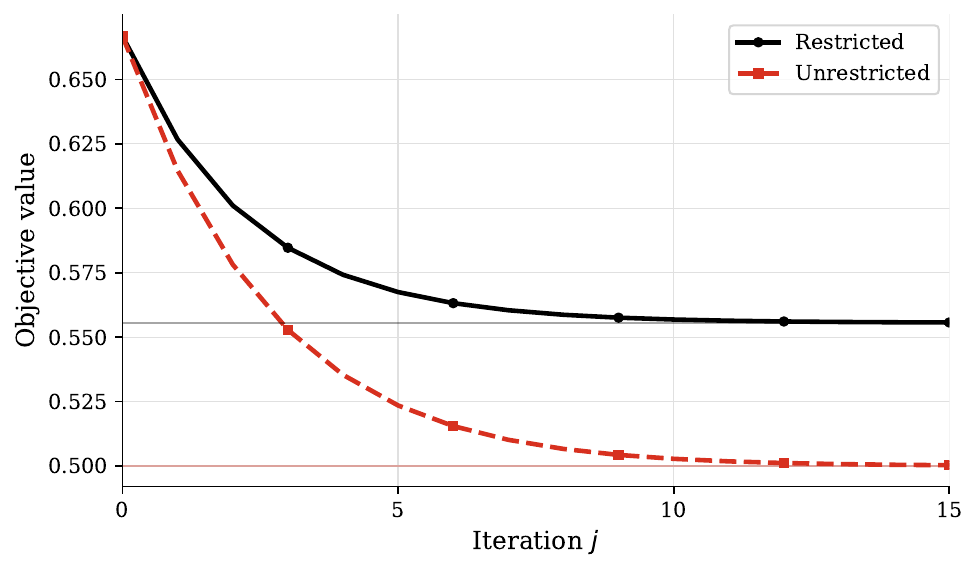}
\caption{Objective value against iteration for under-relaxed versions of the
restricted Procrustes iteration (solid black circles) and the unrestricted
barycenter iteration (dashed red squares), applied to
Example~\ref{ex:barycenter.not.in.FG}. The
two sequences start from the common objective value $2/3$. The restricted
sequence approaches $5/9$, whereas the unrestricted sequence approaches the
strictly smaller value $1/2$.}
\label{fig:algorithm-convergence}
\end{figure}

\begin{lemma}
\label{lem:fixed.point.certificate}
Let $(\mathbf M_t,\mathbf Y_t)$ satisfy \eqref{eq:FPE.unrestricted} and set $\mathbf P_t:=\mathbf M_t\mathbf M_t^\intercal$,
\[
Z_t^k:=\lambda_k(L_{\cdot,t}^k)^\intercal\mathbf Y_t(\mathbf M_t^k)^\intercal,
\qquad k\in[K],
\qquad
\mathbf S_t:=\diag(Z_t^1,\ldots,Z_t^K)-\mathbf C_t.
\]
Then each $Z_t^k$ is symmetric positive semidefinite and $\mathbf S_t\mathbf P_t=0$.  Consequently, if $\mathbf S_t\succeq0$, then $\mathbf P_t$ is a global minimizer of $\Phi_t$ over $\mathscr{P}_+(K,d)$.
\end{lemma}

\begin{proof}
Write $C^k:=(L_{\cdot,t}^k)^\intercal\mathbf Y_t$. Since $\mathbf M_t^k$ is a maximizer in \eqref{eq:FPE.unrestricted}, the equality conditions in the trace--nuclear-norm inequality give the polar optimality identities
\[
C^k(\mathbf M_t^k)^\intercal\succeq0,
\qquad
C^k(\mathbf M_t^k)^\intercal\mathbf M_t^k=C^k.
\]
Thus $Z_t^k=\lambda_kC^k(\mathbf M_t^k)^\intercal$ is symmetric positive semidefinite. Using $\mathbf Y_t=\mathbf L_{\cdot,t}\Lambda\mathbf M_t$, the $(k,\ell)$-block of $\mathbf C_t\mathbf P_t$ is
\[
(\mathbf C_t\mathbf P_t)^{k,\ell}
=\sum_{m=1}^K\lambda_k\lambda_m(L_{\cdot,t}^k)^\intercal L_{\cdot,t}^m\mathbf M_t^m(\mathbf M_t^\ell)^\intercal
=\lambda_k(L_{\cdot,t}^k)^\intercal\mathbf Y_t(\mathbf M_t^\ell)^\intercal
=\lambda_kC^k(\mathbf M_t^\ell)^\intercal,
\]
while the $(k,\ell)$-block of $\diag(Z_t^1,\ldots,Z_t^K)\mathbf P_t$ is
\[
Z_t^k\mathbf P_t^{k,\ell}
=\lambda_kC^k(\mathbf M_t^k)^\intercal\mathbf M_t^k(\mathbf M_t^\ell)^\intercal
=\lambda_kC^k(\mathbf M_t^\ell)^\intercal.
\]
Subtracting the two displays gives $\mathbf S_t\mathbf P_t=0$.  The final assertion follows from Proposition~\ref{prop:sdp.certificate}.
\end{proof}

\begin{remark}
\label{rem:invariant.subspace}
If the row spaces of all $\mathbf M_t^{k,(j)}$ are contained in a common $d$-dimensional subspace $V\subseteq\bR^{Kd}$, then the rows of $\mathbf Y_t^{(j)}$ and of $(L_{\cdot,t}^k)^\intercal\mathbf Y_t^{(j)}$ also lie in $V$. The updated factors remain in $V$ provided that any zero-singular-vector completions in the SVD step are chosen inside $V$; this tie-breaking rule is essential when the cross-matrix is rank deficient. Under this rule, identifying $V$ with $\bR^d$ makes the restricted iteration exactly the time-$t$ column update of Algorithm~\ref{alg:ABW}. To explore the enlarged search space, Algorithm~\ref{alg:unrestricted} instead initializes $\mathbf M_t^{(0)}=I_{Kd}$, equivalently $\mathbf P_t^{(0)}=I_{Kd}$, the strictly feasible center of $\mathscr P_+(K,d)$.
\end{remark}

\begin{remark}
\label{rem:assembling.barycenter}
Suppose the matrices $\mathbf P_1,\ldots,\mathbf P_T$ returned by Algorithm~\ref{alg:unrestricted} are certified by Lemma~\ref{lem:fixed.point.certificate}, so that each $\mathbf P_t$ minimizes $\Phi_t$. By Theorem~\ref{thm:gaussian.multicausal.optimizer} and Lemma~\ref{lem:ensemble}, the ensemble filtered Gaussian process $\bG_K^{\widebar a,\widebar{\mathbf L}}$ with $\widebar{\mathbf L}_{\cdot,t}=\mathbf Y_t$ is then an unrestricted adapted barycenter representative. If moreover $\rank(\mathbf P_t)=d$ for every $t$, Lemma~\ref{lem:Pk-rank-common-noise} and Theorem~\ref{thm:FG-common-noise-criterion} yield an ordinary representative in $\cFG(T,d)$; that representative is a global restricted minimizer, although Algorithm~\ref{alg:ABW} is not guaranteed to converge to it. If some $\mathbf S_t$ has a negative eigenvalue, the fixed point is not certified, and one may restart from a perturbed initialization or use an interior-point solver for \eqref{eq:local.sdp} to obtain a global minimizer.
\end{remark}

\newcommand{\mgle}{\mathrm{mgle}}

\section{The Martingale barycenter}
\label{sec:mgle-barycenters}


\medskip
In this section, we impose additionally the constraint that the barycenter is a martingale. 
The martingale constraint is a natural structural constraint in dynamic model aggregation.
In mathematical finance, martingale laws encode risk-neutral models and the absence of
predictable gains; this is the structural constraint underlying martingale optimal transport
and robust model-independent pricing, beginning with
\cite{BeiglbockHenryLaborderePenkner2013,BeiglbockJuillet2016}.  Adapted Wasserstein
geometry is also motivated by stability of stochastic optimization problems under
nonanticipative perturbations; see \cite{BBBE2020A}.  Thus, if several calibrated Gaussian
processes are averaged into a consensus model, it is important that the averaging procedure
not introduce a drift.  This is one of the advantages of the adapted barycenter over the
classical Wasserstein barycenter: the multicausal barycenter formulation of
\cite{acciaio2025multicausal} shows that barycenters of martingales remain martingales.
Recent work on martingale projection in adapted Wasserstein distance
\cite{BlanchetWieselZhangZhang2026} and on barycentric aggregation of stochastic-process
models \cite{JaimungalPesenti2026} further supports viewing martingale-constrained
barycenters as a natural dynamic aggregation problem.  In the present filtered Gaussian
setting, the projection results of \cite[Section~4]{gunasingam2026adapted} and the common-noise criterion
from Theorem~\ref{thm:FG-common-noise-criterion} give an explicit finite-dimensional
solution.

\subsection{Filtered Gaussian martingales and projections}


\begin{definition}[Filtered Gaussian martingales]
\label{def:FG-mgle-centered-final}
A (centered) filtered Gaussian process \(\bG^{L}\in\cFG(T,d)\) is called a
\emph{filtered Gaussian martingale} if $X:=L\epsilon$
is a martingale with respect to the driving filtration \(\bF^\eps\).  We write
\(\cFG_{\mgle}(T,d)\) for this class and set
\[
\mathscr L_{\mgle}(T,d)
:=\{L\in\mathscr L(T,d):\bG^{L}\in\cFG_{\mgle}(T,d)\}.
\]
The corresponding subset of \(\FP_2\) is denoted by \(\FGMart\).
\end{definition}

By \cite[Proposition~4.2]{gunasingam2026adapted}, the martingale condition is equivalent to
\begin{equation}
\label{eq:FG-mgle-factor-condition-final}
L_{t,s}=L_{s,s},\qquad 1\le s\le t\le T.
\end{equation}
In particular,
\[
X_t=\sum_{s=1}^tL_{s,s}\eps_s,
\qquad
X_{t+1}=X_t+L_{t+1,t+1}\eps_{t+1}.
\]
Thus
\begin{equation}
\label{eq:L-mgle-subspace-final}
\mathscr L_{\mgle}(T,d)
=
\{L\in\mathscr L(T,d):L_{t,s}=L_{s,s},\,1\le s\le t\le T\}.
\end{equation}

\begin{definition}[Canonical martingale projection]
\label{def:canonical-mgle-projection-final}
For \(L\in\mathscr L(T,d)\), define \(L^{\mgle}\in\mathscr L_{\mgle}(T,d)\) by
\begin{equation}
\label{eq:canonical-mgle-projection-final}
(L^{\mgle})_{t,s}
:=
\frac{1}{T-s+1}\sum_{u=s}^T L_{u,s},
\qquad 1\le s\le t\le T.
\end{equation}
For \(\bX=\bG^{L}\), we define its corresponding martingale projection by \(
\bX^{\mgle}:=\bG^{L^{\mgle}}.
\) By \cite[Corollary~4.8]{gunasingam2026adapted}, this is the canonical representative of the
\(\AW_2\)-projection of \(\bG^{L}\) onto the Gaussian martingale class, up to the
usual right-orthogonal equivalence of factor representatives.
\end{definition}

The next identity is the key point needed for the constrained barycenter problem.  The
projection formula is given in \cite{gunasingam2026adapted}; the argument shows how the orthogonal factoring interacts with the Procrustes representation already used previously.

\begin{lemma}[Pythagorean identity]
\label{lem:mgle-projection-pythagorean-final}
Let \(L\in\mathscr L(T,d)\), and let \(M\in\mathscr L_{\mgle}(T,d)\).  Then
\begin{equation}
\label{eq:mgle-projection-pythagorean-final}
\mathrm{dist}_{\mathrm{ABW}}^2(L,M)
=
\|L-L^{\mgle}\|_F^2+
\mathrm{dist}_{\mathrm{ABW}}^2(L^{\mgle},M).
\end{equation}
Equivalently,
\begin{equation}
\label{eq:mgle-projection-pythagorean-process-final}
\AW_2^2(\bG^{L},\bG^{M})
=
\|L-L^{\mgle}\|_F^2+
\AW_2^2(\bG^{L^{\mgle}},\bG^{M}).
\end{equation}
\end{lemma}

\begin{proof}
By \cite[Corollary~4.8]{gunasingam2026adapted}, the representative \(L^{\mgle}\) in
\eqref{eq:canonical-mgle-projection-final} is the Frobenius orthogonal projection of
\(L\) onto the closed linear subspace \(\mathscr L_{\mgle}(T,d)\).  Hence, for every
\(\tilde{L}\in\mathscr L_{\mgle}(T,d)\),
\begin{equation}
\label{eq:mgle-frob-pythagorean-final}
\|L-\tilde{L}\|_F^2
=
\|L-L^{\mgle}\|_F^2+
\|L^{\mgle}-\tilde{L}\|_F^2.
\end{equation}
If \(Q=\diag(Q_1,\ldots,Q_T)\in\mathscr O(T,d)\), then \(MQ\in\mathscr L_{\mgle}(T,d)\),
because
\[
(MQ)_{t,s}=M_{t,s}Q_s=M_{s,s}Q_s=(MQ)_{s,s},
\qquad t\ge s.
\]
Applying \eqref{eq:mgle-frob-pythagorean-final} with \(\tilde{L}=MQ\) gives
\[
\|L-MQ\|_F^2
=
\|L-L^{\mgle}\|_F^2+
\|L^{\mgle}-MQ\|_F^2.
\]
Taking the minimum over \(Q\in\mathscr O(T,d)\) and using
\eqref{eq:Procrustes} proves \eqref{eq:mgle-projection-pythagorean-final}.
The process-level identity follows from \eqref{eq:AW2.Gaussian}.
\end{proof}

\subsection{Barycenter of martingales}
For arbitrary centered filtered Gaussian inputs
\[
\bX^k=\bG^{L^k}\in\cFG(T,d),\qquad k\in[K],
\]
consider the martingale-constrained Gaussian barycenter problem
\begin{equation}
\label{eq:constrained-mgle-barycenter-problem-final}
\inf_{\bY\in\cFG_{\mgle}(T,d)}
\sum_{k=1}^K\lambda_k\AW_2^2(\bX^k,\bY).
\end{equation}

We first solve the unrestricted adapted barycenter problem when the inputs are already
centered filtered Gaussian martingales.  The proof uses the restricted filtered Gaussian
problem and the common-noise criterion from the preceding section.  The point is that,
for martingale factors, the contribution of the time-\(t\) innovation is the same block
repeated on the future rows \(t,t+1,\ldots,T\).  Thus the local unrestricted problem is a
positive scalar multiple of an ordinary \(d\)-dimensional Gaussian barycenter problem,
which admits a rank-\(d\) common-noise optimizer.

\begin{theorem}[Martingale barycenter]
\label{thm:unrestricted-FG-mgle-barycenter-final}
Let
\[
\bX^k:=\bG^{L^k}\in\cFG_{\mgle}(T,d),
\qquad k\in[K].
\]
Then the unrestricted adapted barycenter problem
\begin{equation}
\label{eq:unrestricted-mgle-input-problem-final}
\inf_{\bY\in\FP_2}
\sum_{k=1}^K\lambda_k\AW_2^2(\bX^k,\bY)
\end{equation}
admits a representative \(\bG^{\widebar L}\in\cFG_{\mgle}(T,d)\).  It may be chosen
as follows.  For every \(t\in[T]\), choose orthogonal matrices
\( Q_t^1,\ldots, Q_t^K\in\mathscr O(d)\) minimizing
\begin{equation}
\label{eq:mgle-local-factor-problem-final}
\sum_{k=1}^K\lambda_k\|L^k_{t,t}\|_F^2
-
\left\|\sum_{k=1}^K\lambda_kL^k_{t,t} Q_t^k\right\|_F^2.
\end{equation}
Set
\begin{equation}
\label{eq:mgle-local-factor-average-final}
\widebar L_{t,t}:=\sum_{k=1}^K\lambda_kL^k_{t,t}\mathbf Q_t^k,
\qquad
\widebar L_{r,t}:=\widebar L_{t,t},\quad r\ge t.
\end{equation}
Then \(\bG^{0,\widebar L}\) solves \eqref{eq:unrestricted-mgle-input-problem-final}.  Equivalently,
its one-step covariance blocks are characterized by
\begin{equation}
\label{eq:mgle-local-BW-covariance-final}
\widebar L_{t,t}\widebar L_{t,t}^{\intercal}
\in
\argmin_{\Sigma\in\mathscr S_+(d)}
\sum_{k=1}^K\lambda_k
\mathrm{dist}_{\BW}^2\!\left(
L^k_{t,t}(L^k_{t,t})^{\intercal},\Sigma
\right),
\qquad t\in[T].
\end{equation}
The barycenter value is
\begin{equation}
\label{eq:mgle-unrestricted-value-final}
\sum_{t=1}^T(T-t+1)
\min_{\Sigma\in\mathscr S_+(d)}
\sum_{k=1}^K\lambda_k
\mathrm{dist}_{\BW}^2\!\left(
L^k_{t,t}(L^k_{t,t})^{\intercal},\Sigma
\right).
\end{equation}
In particular, for every martingale candidate \(\bG^{0,M}\in\cFG_{\mgle}(T,d)\),
\begin{equation}
\label{eq:AW-between-mgle-factors-final}
\AW_2^2(\bG^{0,L^k},\bG^{0,M})
=
\sum_{t=1}^T(T-t+1)
\mathrm{dist}_{\BW}^2\!\left(
L^k_{t,t}(L^k_{t,t})^{\intercal},
M_{t,t}M_{t,t}^{\intercal}
\right).
\end{equation}
If the minimizer in \eqref{eq:mgle-local-BW-covariance-final} is unique for every
\(t\), then the barycenter is unique at the level of one-step increment covariances.
This uniqueness holds, for instance, whenever for each \(t\) at least one matrix
\(L^k_{t,t}(L^k_{t,t})^{\intercal}\) is positive definite.
\end{theorem}

\begin{proof}
We first compute the martingale-restricted objective.  Let
\(\bG^{M}\in\cFG_{\mgle}(T,d)\).  By \cite[Corollary~4.5]{gunasingam2026adapted}, the adapted
Gaussian distance between martingale factors is exactly
\eqref{eq:AW-between-mgle-factors-final}.  Consequently,
\begin{align*}
&\sum_{k=1}^K\lambda_k\AW_2^2(\bG^{L^k},\bG^{M})=
\sum_{t=1}^T(T-t+1)
\sum_{k=1}^K\lambda_k
\mathrm{dist}_{\BW}^2\!\left(
L^k_{t,t}(L^k_{t,t})^{\intercal},
M_{t,t}M_{t,t}^{\intercal}
\right).
\end{align*}
The variables \(M_{t,t}M_{t,t}^{\intercal}\) are separated over \(t\), and the
coefficients \(T-t+1\) are strictly positive.  Hence the martingale-restricted
minimizers are exactly the local Bures--Wasserstein barycenters in
\eqref{eq:mgle-local-BW-covariance-final}, repeated as in
\eqref{eq:mgle-local-factor-average-final}.  The restricted value is therefore
\eqref{eq:mgle-unrestricted-value-final}.  Existence of the local covariance minimizers
follows directly from Theorem~\ref{thm:explicit.value}, applied to the static
\((T=1)\) Gaussian problem. This argument includes singular increment covariances.

It remains to prove that the same factor is an unrestricted adapted barycenter, not only a
martingale-restricted minimizer.  Fix \(t\in[T]\).  Since every input is a martingale
factor, for all \(k,\ell\in[K]\),
\begin{equation}
\label{eq:mgle-column-scaling-final}
(L^k_{\cdot,t})^{\intercal}L^\ell_{\cdot,t}
=
(T-t+1)(L^k_{t,t})^{\intercal}L^\ell_{t,t}.
\end{equation}
Thus the one-step trace functional \(\Phi_t\) from \eqref{eq:Phi_t} is
\((T-t+1)\) times the ordinary Gaussian multimarginal trace functional for the
\(d\)-dimensional covariance matrices
\[
L^1_{t,t}(L^1_{t,t})^{\intercal},\ldots,
L^K_{t,t}(L^K_{t,t})^{\intercal}.
\]
The factor formulation of this ordinary Gaussian barycenter problem is precisely the
minimization in \eqref{eq:mgle-local-factor-problem-final}; for fixed
\(\mathbf Q^1,\ldots,\mathbf Q^K\), the optimal factor is the weighted aligned average
\(\sum_k\lambda_kL^k_{t,t}\mathbf Q^k\).

Let \(\mathbf Q_t^1,\ldots,\mathbf Q_t^K\) be optimizers of
\eqref{eq:mgle-local-factor-problem-final}, and define
\begin{equation}
\label{eq:mgle-common-noise-matrix-final}
(\mathbf P_t^\star)^{k,\ell}:=\mathbf Q_t^k(\mathbf Q_t^\ell)^{\intercal},
\qquad k,\ell\in[K].
\end{equation}
By Lemma~\ref{lem:Pk-rank-common-noise}, \(\mathbf P_t^\star\in\mathscr{P}_+(K,d)\) and
\(\rank(\mathbf P_t^\star)=d\). We claim that \(\mathbf P_t^\star\) minimizes the full local
functional \(\Phi_t\) over \(\mathscr{P}_+(K,d)\), and not merely over its rank-\(d\)
elements. By the scaling identity \eqref{eq:mgle-column-scaling-final}, this is the
static \((T=1)\) trace problem for
\[
A^k:=L^k_{t,t}(L^k_{t,t})^{\intercal},\qquad k\in[K].
\]
Theorem~\ref{thm:explicit.value}, applied with horizon one, identifies its unrestricted
minimum with
\[
\min_{\Sigma\in\mathscr S_+(d)}
\sum_{k=1}^K\lambda_k\mathrm{dist}_{\BW}^2(A^k,\Sigma).
\]
On the other hand, \eqref{eq:local.restricted.factor.problem}, again with horizon one,
identifies the rank-\(d\) minimum with the same expression restricted to covariances of
the form \(YY^{\intercal}\), \(Y\in\bR^{d\times d}\). This is no restriction because
every \(\Sigma\in\mathscr S_+(d)\), singular or nonsingular, has a square root of that
form. Hence the unrestricted and rank-\(d\) minima coincide, and the optimizer
\(\mathbf P_t^\star\) constructed above is a global minimizer of \(\Phi_t\). Since this
holds at every time,
Theorem~\ref{thm:FG-common-noise-criterion} applies.  Hence the unrestricted
adapted barycenter admits an ordinary filtered Gaussian representative with
block-columns
\begin{equation}
\label{eq:mgle-common-noise-column-final}
\widebar L_{\cdot,t}=
\sum_{k=1}^K\lambda_kL^k_{\cdot,t}\mathbf Q_t^k.
\end{equation}
For \(r\ge t\), the \((r,t)\)-block of the right-hand side is
\[
\sum_{k=1}^K\lambda_kL^k_{t,t}\mathbf Q_t^k,
\]
which is independent of \(r\).  Thus \(\widebar L_{r,t}=\widebar L_{t,t}\) for all
\(r\ge t\), so \(\widebar L\in\mathscr L_{\mgle}(T,d)\).  The unrestricted value is therefore
identical to the martingale-restricted value computed above, and \(\bG^{0,\widebar L}\) is
an unrestricted adapted barycenter representative.

The final assertion follows by applying the standard uniqueness criterion for Gaussian
covariance barycenters independently at each time. Different factors of a fixed covariance
represent the same increment covariance and differ only by the usual right orthogonal
equivalence in \(\FG\).
\end{proof}

\begin{theorem}[Projection before martingale barycenter]
\label{thm:projection-before-mgle-barycenter-final}
Let \(\bX^k:=\bG^{L^k}\in\cFG(T,d)\)for each \(k\in[K]\), and set
\(
(\bX^k)^{\mgle}:=\bG^{(L^k)^{\mgle}}.
\)
Then
\begin{align}
\label{eq:projection-before-restricted-final}
&\inf_{\bY\in\cFG_{\mgle}(T,d)}
\sum_{k=1}^K\lambda_k\AW_2^2(\bX^k,\bY)\notag\\
&\quad=
\sum_{k=1}^K\lambda_k\|L^k-(L^k)^{\mgle}\|_F^2
+
\inf_{\bY\in\cFG_{\mgle}(T,d)}
\sum_{k=1}^K\lambda_k\AW_2^2((\bX^k)^{\mgle},\bY).
\end{align}
Moreover, the two infima over \(\cFG_{\mgle}(T,d)\) have exactly the same minimizers.
Applying Theorem~\ref{thm:unrestricted-FG-mgle-barycenter-final} to the
projected martingale inputs, this becomes
\begin{align}
\label{eq:projection-before-unrestricted-final}
&\inf_{\bY\in\cFG_{\mgle}(T,d)}
\sum_{k=1}^K\lambda_k\AW_2^2(\bX^k,\bY)\notag\\
&
\quad=
\sum_{k=1}^K\lambda_k\|L^k-(L^k)^{\mgle}\|_F^2
+
\inf_{\bZ\in\FP_2}
\sum_{k=1}^K\lambda_k\AW_2^2((\bX^k)^{\mgle},\bZ).
\end{align}
Consequently, a martingale representative \(\bY\in\cFG_{\mgle}(T,d)\) solves
\eqref{eq:constrained-mgle-barycenter-problem-final} if and only if \(\bY\) is an
unrestricted adapted barycenter representative of
\((\bX^1)^{\mgle},\ldots,(\bX^K)^{\mgle}\).
\end{theorem}

\begin{proof}
Let \(\bY=\bG^{M}\in\cFG_{\mgle}(T,d)\).  Applying
Lemma~\ref{lem:mgle-projection-pythagorean-final} to the pair \((L^k,M)\) gives, for
each \(k\),
\[
\AW_2^2(\bG^{L^k},\bG^{M})
=
\|L^k-(L^k)^{\mgle}\|_F^2
+
\AW_2^2(\bG^{(L^k)^{\mgle}},\bG^{M}).
\]
Multiplying by \(\lambda_k\) and summing over \(k\) yields the pointwise identity
\begin{align}
\label{eq:projection-before-pointwise-final}
&\sum_{k=1}^K\lambda_k\AW_2^2(\bG^{L^k},\bG^{M})\notag\\
&\quad=
\sum_{k=1}^K\lambda_k\|L^k-(L^k)^{\mgle}\|_F^2
+
\sum_{k=1}^K\lambda_k\AW_2^2(\bG^{(L^k)^{\mgle}},\bG^{M}).
\end{align}
The first term on the right-hand side is independent of \(M\).  Therefore taking the
infimum over \(M\in\mathscr L_{\mgle}(T,d)\) gives
\eqref{eq:projection-before-restricted-final}, and the equality of minimizer sets follows
from the same pointwise identity.

It remains only to justify the unrestricted infimum in
\eqref{eq:projection-before-unrestricted-final}.  The inputs
\((\bX^k)^{\mgle}\) are centered filtered Gaussian martingales.  By
Theorem~\ref{thm:unrestricted-FG-mgle-barycenter-final}, their unrestricted adapted
barycenter problem over \(\FP_2\) admits a representative in
\(\cFG_{\mgle}(T,d)\), and the unrestricted value is equal to the value of the same
problem restricted to \(\cFG_{\mgle}(T,d)\).  Substituting this equality into
\eqref{eq:projection-before-restricted-final} proves
\eqref{eq:projection-before-unrestricted-final}.  The final assertion follows because a
martingale candidate minimizes the left-hand side if and only if, by
\eqref{eq:projection-before-pointwise-final}, it minimizes the projected-input objective;
by Theorem~\ref{thm:unrestricted-FG-mgle-barycenter-final}, these minimizers are exactly
the unrestricted barycenter representatives of the projected martingales that lie in
\(\cFG_{\mgle}(T,d)\).
\end{proof}

\begin{corollary}
\label{cor:explicit-constrained-mgle-barycenter-final}
Let \(\bX^k=\bG^{0,L^k}\in\cFG(T,d)\), \(k\in[K]\), be arbitrary centered
filtered Gaussian inputs.  Then the constrained problem
\eqref{eq:constrained-mgle-barycenter-problem-final} is solved as follows.  First form the
martingale projections \((L^k)^{\mgle}\) by
\eqref{eq:canonical-mgle-projection-final}.  Then, for every \(t\in[T]\), choose
\begin{equation}
\label{eq:constrained-projected-local-covariance-final}
\widebar{\Sigma}_t:=\widebar L_{t,t}\widebar L_{t,t}^{\intercal}
\in
\argmin_{\Sigma\in\mathscr S_+(d)}
\sum_{k=1}^K\lambda_k
\mathrm{dist}_{\BW}^2\!\left(
(L^k)^{\mgle}_{t,t}\big((L^k)^{\mgle}_{t,t}\big)^{\intercal},
\Sigma
\right),
\end{equation}
and set \(\widebar L_{t,s}=\widebar L_{s,s}\) for \(t\ge s\). Then
\(\bG^{\widebar L}\in\cFG_{\mgle}(T,d)\) solves
\eqref{eq:constrained-mgle-barycenter-problem-final}.  Its value is
\begin{align}
\label{eq:constrained-mgle-value-explicit-final}
&\sum_{k=1}^K\lambda_k\|L^k-(L^k)^{\mgle}\|_F^2\notag\\
&\quad+
\sum_{t=1}^T(T-t+1)
\min_{\Sigma_t\in\mathscr S_+(d)}
\sum_{k=1}^K\lambda_k
\mathrm{dist}_{\BW}^2\!\left(
(L^k)^{\mgle}_{t,t}\big((L^k)^{\mgle}_{t,t}\big)^{\intercal},
\Sigma_t
\right),
\end{align}
where
\begin{equation}
\label{eq:projected-mgle-diagonal-block-final}
(L^k)^{\mgle}_{t,t}
=
\frac{1}{T-t+1}\sum_{s=t}^T L^k_{s,t}.
\end{equation}
\end{corollary}

\begin{proof}
This is the combination of Theorem~\ref{thm:projection-before-mgle-barycenter-final}
with Theorem~\ref{thm:unrestricted-FG-mgle-barycenter-final} applied to the projected
martingales \((\bX^1)^{\mgle},\ldots,(\bX^K)^{\mgle}\).  The value formula is obtained
by substituting \eqref{eq:mgle-unrestricted-value-final} into
\eqref{eq:projection-before-unrestricted-final}.
\end{proof}

\begin{remark}\label{rem:mgle-projection-first-not-output-projection-final}
Theorem~\ref{thm:projection-before-mgle-barycenter-final} does not assert that one may
first compute an unrestricted barycenter of the original inputs and then project the output
onto \(\mathscr L_{\mgle}(T,d)\).  The identity
\eqref{eq:mgle-projection-pythagorean-final} holds against martingale candidates.  In
general, the optimal Procrustes alignments in \(\mathrm{dist}_{\mathrm{ABW}}\) depend on the candidate
factor, so barycenter and projection need not commute. The result shows that the
variational statement is required for the constrained problem: once the candidate is restricted
to the martingale class, the non-martingale component of each input contributes only the
constant residual \(\|L^k-(L^k)^{\mgle}\|_F^2\).
\end{remark}

\section{Example: Autoregressive Processes}\label{sec:OU}
In this section we study univariate, finite-horizon, time-varying $AR(1)$ inputs and their adapted barycenters. This class is rich enough to exhibit the key features of the adapted setting while remaining explicit. Recall that $AR(1)$ is the discrete-time counterpart of the Ornstein--Uhlenbeck process.
\subsection*{Autoregressive Processes}
$AR(p)$ processes serve as a standard tool in time-series analysis and offer an example of iterated function systems, see for instance \cite{alsm03}. 
They arise naturally in a wide range of applications, including financial time series modelling, signal processing, and econometrics, see \cite{BrockwellDavis1991} for a thorough analysis. Here, we focus on the base case $p=1$.  As in \eqref{eqn:filtered.Gaussian}, let $\eps = (\eps_t)_{t=1}^T$ be the canonical standard Gaussian innovation process with its generated filtration $\bF^{\eps}$.
For $k = 1, \dots, K$ we define $\bX^k$ as the univariate discrete-time autoregressive process $AR(1)$ on $t = 1, \ldots, T$ by
\footnote{
Typically, $AR(1)$ processes are introduced with a time-independent autoregressive parameter. Here we allow a general finite-horizon time-varying parameter $\alpha_t^k$. Time-varying autoregressions under additional local-stationarity assumptions are studied in \cite{dahlhaus1996, dahlhaus1997}.
}
\begin{equation}\label{eq:AR_def}
\begin{aligned}
X_1^k &= \sigma_1^k \eps_1, \\
X_t^k &= \alpha_t^k X_{t-1}^k + \sigma_t^k \eps_t \quad \text{for } t = 2, \ldots, T,
\end{aligned}
\end{equation}
where $\alpha_t ^k\in \bR$ is the \emph{autoregressive coefficient} and $\sigma_t^k > 0$ is the \emph{volatility}. We may expand the terms in \eqref{eq:AR_def} to
\begin{equation}\label{eq:unrolled}
X_t^k = \sigma_t^k \eps_t + \alpha_t^k \sigma_{t-1}^k \eps_{t-1} + \alpha_t ^k\alpha_{t-1}^k \sigma_{t-2}^k \eps_{t-2} + \cdots + \left(\prod_{u=2}^t \alpha_u^k\right) \sigma_1^k \eps_1.
\end{equation}
Thus, we may rewrite \eqref{eq:AR_def} in the form $X^k = L^k\eps$ for a lower-triangular matrix $L^k \in \mathscr L(T,1)$ with entries
\begin{equation}\label{eq:Lentries}
L_{t,s}^k = \sigma_s^k \prod_{u=s+1}^t \alpha_u^k, \quad 1 \leq s \leq t \leq T,
\end{equation}
and the convention that empty products equal $1$. Note that since $\sigma_t^k > 0$ for all $t$, we have $L_{t,t}^k = \sigma_t^k > 0$, and thus $L^k \in \mathscr L^{\mathrm{reg}}(T,1)$.
\subsection*{The fixed-point equation}
Let $\lambda_1, \dots, \lambda_K > 0$ be convex weights. If $\widebar L$ is a restricted minimizer, the fixed-point condition \eqref{eq:FPE} can be written columnwise in the scalar case as follows. For each $k$ and each innovation time $s$, choose
\[
\varepsilon_s^k\in \operatorname{sign}\bigl(((L^k)^\intercal \widebar L)_{s,s}\bigr),
\qquad
\operatorname{sign}(r)=
\begin{cases}
\{+1\},& r>0,\\
\{-1\},& r<0,\\
\{+1,-1\},& r=0.
\end{cases}
\]
Then
\begin{equation}\label{eq:FPEentries}
\widebar L_{t,s} = \sum_{k=1}^K \lambda_k L^k_{t,s} \varepsilon^k_s, \quad 1 \leq s \leq t \leq T.
\end{equation}
In the special case where all the process parameters $\alpha_t^k$ are positive for $k = 1, \dots, K$ and $t = 2, \dots , T$, the barycenter of the $AR(1)$ processes $\bX^k$ reduces to their Knothe--Rosenblatt barycenter; see \cite[Section 1.5]{KnoRos_bar}.

\begin{proposition}[AR processes with nonnegative coefficients]\label{prop:AR_pos_signs}
Let \(\bX^1,\ldots,\bX^K\) be the univariate \(AR(1)\) processes defined by
\eqref{eq:AR_def}, with \(\sigma_t^k>0\) and \(\alpha_t^k\ge 0\) for every
\(k\in[K]\) and \(t=2,\ldots,T\). Then, for every \(t\in[T]\), we have
\[
\mathbf P_t^\star:=\mathbf 1_K\mathbf 1_K^\intercal
\]
where $\mathbf{1}_K:=[1,1,\dots,1]^\intercal\in\R^K$
is the unique minimizer of \(\Phi_t\) over \(\mathscr{P}_+(K)\).

Consequently, the common-noise input condition holds, the restricted and
unrestricted barycenter values coincide, and the barycenter admits the ordinary
filtered Gaussian representative
\[
\widebar{\bX}=\bG^{0,\widebar L}\in\cFG(T,1),
\qquad
\widebar L_{\cdot,t}
=
\mathbf L_{\cdot,t}\Lambda\mathbf 1_K
=
\sum_{k=1}^K\lambda_k L_{\cdot,t}^k,
\qquad t\in[T].
\]
Equivalently, under the common innovation convention,
\[
\widebar X_u=\sum_{k=1}^K\lambda_kX_u^k,
\qquad u\in[T].
\]
Thus, the adapted barycenter, both unrestricted in \(\FP_2\) and restricted to
\(\cFG(T,1)\), coincides with the Knothe--Rosenblatt barycenter.
\end{proposition}

\begin{proof}
Fix \(t\in[T]\), and set
\[
G_t:=\mathbf L_{\cdot,t}^\intercal\mathbf L_{\cdot,t}.
\]
Since \(d=1\), the set \(\mathscr{P}_+(K)\) is the set of \(K\times K\) correlation
matrices. By \eqref{eq:Phi_t}, the local objective is
\[
\Phi_t(\mathbf P)
=
\tr(\Lambda G_t)
-
\tr(\Lambda G_t\Lambda\mathbf P),
\qquad \mathbf P\in\mathscr{P}_+(K).
\]
The first trace is independent of \(\mathbf P\). Hence minimizing \(\Phi_t\) over
\(\mathscr{P}_+(K)\) is equivalent to maximizing the trace term
\[
\tr(\Lambda G_t\Lambda\mathbf P).
\]

Write
\[
g_t^{k,\ell}:=(G_t)_{k,\ell}
=
\bigl(L_{\cdot,t}^k\bigr)^\intercal L_{\cdot,t}^\ell
=
\sum_{u=t}^T L_{u,t}^kL_{u,t}^\ell .
\]
By \eqref{eq:Lentries} and the assumption \(\alpha_s^k\ge0\),
\[
L_{u,t}^k
=
\sigma_t^k\prod_{s=t+1}^u\alpha_s^k
\ge 0,
\qquad u\ge t,
\]
and \(L_{t,t}^k=\sigma_t^k>0\). Therefore, for every \(k,\ell\in[K]\),
\[
g_t^{k,\ell}
=
\sum_{u=t}^T L_{u,t}^kL_{u,t}^\ell
\ge
L_{t,t}^kL_{t,t}^\ell
=
\sigma_t^k\sigma_t^\ell
>0.
\]

Using the trace expansion \eqref{eq:Phi_t.expanded}, in the scalar case we have
\[
\tr(\Lambda G_t\Lambda\mathbf P)
=
\sum_{k,\ell=1}^K
\lambda_k\lambda_\ell g_t^{k,\ell}\mathbf P^{\ell,k}.
\]
Since \(\mathbf P\in\mathscr{P}_+(K)\) is a correlation matrix, every \(2\times2\)
principal minor is nonnegative, and hence
\[
|\mathbf P^{k,\ell}|\le 1,
\qquad k,\ell\in[K].
\]
In particular, \(\mathbf P^{\ell,k}\le 1\). Since all coefficients
\(\lambda_k\lambda_\ell g_t^{k,\ell}\) are strictly positive, it follows that
\[
\tr(\Lambda G_t\Lambda\mathbf P)
\le
\sum_{k,\ell=1}^K
\lambda_k\lambda_\ell g_t^{k,\ell}
=
\tr\!\left(
\Lambda G_t\Lambda\mathbf 1_K\mathbf 1_K^\intercal
\right).
\]
Thus, \(\mathbf P_t^\star=\mathbf 1_K\mathbf 1_K^\intercal\) maximizes the second trace
term, and therefore minimizes \(\Phi_t\).

Moreover, equality can hold only if \(\mathbf P^{k,\ell}=1\) for every off-diagonal pair
\(k\ne\ell\). Since also \(\mathbf P^{k,k}=1\) by definition of \(\mathscr{P}_+(K)\), equality
holds only for \(\mathbf P=\mathbf P_t^\star\). Hence \(\mathbf P_t^\star\) is the unique minimizer.

The matrix \(\mathbf P_t^\star\) has rank one and admits the common-noise representation
\[
(\mathbf P_t^\star)^{k,\ell}=\mathbf Q_t^k(\mathbf Q_t^\ell)^\intercal,
\qquad
\mathbf Q_t^k=1\in\mathscr O(1).
\]
Therefore the common-noise input condition holds. By
Theorem~\ref{thm:FG-common-noise-criterion}, the unrestricted barycenter has an
ordinary filtered Gaussian representative, and its factor is given by
\[
\widebar L_{\cdot,t}
=
\sum_{k=1}^K\lambda_kL_{\cdot,t}^k\mathbf Q_t^k
=
\sum_{k=1}^K\lambda_kL_{\cdot,t}^k
=
\mathbf L_{\cdot,t}\Lambda\mathbf 1_K.
\]
Since
\[
\widebar L_{t,t}
=
\sum_{k=1}^K\lambda_k\sigma_t^k
>0,
\]
this representative has the canonical Cholesky orientation.

Finally, under the common innovation convention,
\[
\widebar X_t
=
\sum_{s=1}^t\widebar L_{t,s}\eps_s
=
\sum_{s=1}^t\sum_{k=1}^K\lambda_kL_{t,s}^k\eps_s
=
\sum_{k=1}^K\lambda_kX_t^k,
\qquad t\in[T].
\]
This is precisely the Knothe--Rosenblatt barycenter of the triangular Gaussian
representatives. The restricted and unrestricted values coincide because every
local unrestricted optimizer \(\mathbf P_t^\star\) already has rank \(d=1\).
\end{proof}

\begin{remark}[The barycenter of $AR(1)$ processes is generally not $AR(1)$]\label{rem:nonAR}
The Cholesky factor of an $AR(1)$ process with parameters $(\alpha_t^k, \sigma_t^k)$ has entries $L_{t,s}^k = \sigma_s^k \prod_{u=s+1}^t \alpha_u^k$. Its defining structural property is that the ratios $L_{t,s}^k / L_{t-1,s}^k = \alpha_t^k$, $s < t$, do not depend on $s$. Suppose now that $\alpha_t^k>0$ for every $k\in[K]$ and $t=2,\ldots,T$, while $\sigma_t^k>0$ for all $k,t$. By the proof of Proposition \ref{prop:AR_pos_signs}, we have $\varepsilon_s^k = 1$ for every $k,s$, and the barycenter is $\widebar L_{t,s} = \sum_{k=1}^K \lambda_k \sigma_s^k \prod_{u=s+1}^t \alpha_u^k$. Hence, the barycenter is $AR(1)$ if and only if the ratio
\[
\frac{\widebar L_{t,s}}{\widebar L_{t-1,s}} = \frac{\sum_{k=1}^K \lambda_k \sigma_s^k \alpha_t^k \prod_{u=s+1}^{t-1} \alpha_u^k}{\sum_{k=1}^K \lambda_k \sigma_s^k \prod_{u=s+1}^{t-1} \alpha_u^k}
\]
is independent of $s$ for each $t$. This is a weighted average of $\alpha_t^k$ with weights proportional to $\lambda_k \sigma_s^k \prod_{u=s+1}^{t-1} \alpha_u^k$, which may depend on $s$. Thus, for $T \geq 3$, heterogeneous coefficients can destroy the $AR(1)$ consistency condition; generically, the barycenter is then not $AR(1)$.

\medskip
The column decomposition (Proposition \ref{prop:decomp}) gives an intuitive explanation. The barycenter problem decouples into independent Bures--Wasserstein problems for each noise column, one per time step. The $AR(1)$ property requires consistency across columns ($\widebar L_{t,s}/\widebar L_{t-1,s}$ independent of $s$), but since the columns are optimized independently, no mechanism enforces this structural constraint. In the positive-coefficient setting above, for $T = 2$ the condition is automatic because there is only one ratio to check, whereas for $T \geq 3$ it may fail.
\end{remark}

We construct an explicit example where the adapted and classical barycenter do not coincide. 

\begin{example}[The adapted Bures--Wasserstein barycenter vs.\ the classical Bures-Wasserstein barycenter]
Consider $K = 2$ processes with equal weights $\lambda_1 = \lambda_2 = \frac{1}{2}$, time horizon $T = 2$, and parameters $\sigma^1_1 = \sigma^2_1 = 1$, $\sigma^1_2 = \sigma^2_2 = 1$, $\alpha^1_2 = 0.5$, $\alpha^2_2 = -0.5$. By \eqref{eq:Lentries}, the Cholesky factors are
\[
L^1 = 
\begin{pmatrix} 
1 & 0 \\ 0.5 & 1
\end{pmatrix}, 
\qquad 
L^2 = 
\begin{pmatrix} 
1 & 0 \\ -0.5 & 1
\end{pmatrix}.
\]
We compute the adapted barycenter using the column decomposition of Proposition~\ref{prop:decomp}. For the first column ($t = 1$), the full block columns are $L_{\cdot,1}^1 = (1,\, 0.5)^\intercal$ and $L_{\cdot,1}^2 = (1,\, -0.5)^\intercal$. Since $\mathscr O(1) = \{+1, -1\}$, we need to find the signs $\varepsilon_1^1, \varepsilon_1^2 \in \{+1, -1\}$ that maximize $\|\frac{1}{2}(\varepsilon_1^1 L_{\cdot,1}^1 + \varepsilon_1^2 L_{\cdot,1}^2)\|_F^2$, as in the proof of Proposition~\ref{prop:AR_pos_signs}. By direct computation:
\begin{alignat*}{2}
&\varepsilon_1^1 = +1,\; \varepsilon_1^2 = +1: &\quad &\tfrac{1}{2}(1, 0.5)^\intercal + \tfrac{1}{2}(1, -0.5)^\intercal = (1,\, 0)^\intercal, \quad \text{norm}^2 = 1, \\
&\varepsilon_1^1 = +1,\; \varepsilon_1^2 = -1: &\quad &\tfrac{1}{2}(1, 0.5)^\intercal - \tfrac{1}{2}(1, -0.5)^\intercal = (0,\, 0.5)^\intercal, \quad \text{norm}^2 = 0.25,
\end{alignat*}
and the remaining two cases give the negatives of these, with the same norms. Hence $\varepsilon_1^1 = \varepsilon_1^2 = +1$ is optimal (up to a global sign flip), yielding $\widebar L_{\cdot,1} = (1, 0)^\intercal$. For the second column ($t = 2$), the full block columns are $L_{\cdot,2}^1=L_{\cdot,2}^2=(0,1)^\intercal$, so trivially $\widebar L_{\cdot,2}=(0,1)^\intercal$. Therefore, the adapted barycenter is
\begin{equation}\label{eq:OU:numresult}
\widebar L = \frac{1}{2}(L^1 + L^2) = \begin{pmatrix} 1 & 0 \\ 0 & 1\end{pmatrix} = I_2, 
\end{equation}
and in particular, $\widebar L$ has covariance matrix $\widebar\Sigma_\mathrm{ABW} = I_2$. In conclusion, as the off-diagonal entries of $L^1$ and $L^2$ cancel due to the opposite signs of $\alpha^1$ and $\alpha^2$, the adapted barycenter is a white noise process with no temporal correlation.

\medskip
On the other hand, the covariance matrices of the two input processes are
\[
\Sigma^1 = \begin{pmatrix} 1 & 0.5 \\ 0.5 & 1.25 \end{pmatrix}, \qquad \Sigma^2 = \begin{pmatrix} 1 & -0.5 \\ -0.5 & 1.25 \end{pmatrix}.
\]
Note that $\Sigma^1$ and $\Sigma^2$ share the same diagonal entries: the squaring $\Sigma^k = L^k(L^k)^\intercal$ preserves the sign of $\alpha^k$ in the off-diagonal entries, while the diagonal entries depend only on $|\alpha^k|$. Let $D=\diag(1,-1)$. Then $\Sigma^2=D\Sigma^1D$. The equal-weight Bures--Wasserstein objective is invariant under $\Sigma\mapsto D\Sigma D$, and its barycenter covariance is unique here because the inputs are positive definite. Consequently, $\widebar\Sigma_{\BW}=D\widebar\Sigma_{\BW}D$, which forces its off-diagonal entries to vanish. This symmetry argument, rather than arithmetic averaging of covariance matrices, explains the diagonal form below.
The covariance $\widebar\Sigma_{\BW}$ of the classical Bures--Wasserstein barycenter solves the fixed-point equation of \cite{AC11},
\begin{equation}\label{eq:AgCar_FPE}
\widebar\Sigma_{\BW} = \frac{1}{2}\bigl[(\widebar\Sigma_{\BW}^{\frac12}\Sigma^1\widebar\Sigma_{\BW}^{\frac12})^{\frac12} + (\widebar\Sigma_{\BW}^{\frac12}\Sigma^2\widebar\Sigma_{\BW}^{\frac12})^{\frac12}\bigr].
\end{equation}
We may numerically compute
\[
\widebar\Sigma_{\BW} \approx \begin{pmatrix} 0.934 & 0 \\ 0 & 1.199 \end{pmatrix} \neq I_2 = \widebar\Sigma_\mathrm{ABW}.
\]
The symmetry forces zero off-diagonal entries. The classical Bures--Wasserstein mean
nevertheless has diagonal entries different from one. By contrast, the adapted distance
operates columnwise on $L$, where the opposite autoregressive signs cancel directly in the
aligned factor average, yielding the white-noise covariance $\Sigma_\mathrm{ABW}=I_2$.
\end{example}

\section{Numerical illustrations}\label{sec:numerics}
We illustrate the difference between the adapted and classical Bures--Wasserstein barycenters through numerical experiments. We consider $K = 10$ univariate $AR(1)$ processes on a time horizon $T = 30$, with parameters paired symmetrically
\[
\alpha^k \in \{0.92,\, 0.85,\, 0.75,\, 0.60,\, 0.50,\, {-0.92},\, {-0.85},\, {-0.75},\, {-0.60},\, {-0.50}\},
\]
\[
\sigma^k \in \{1.0,\, 1.3,\, 0.8,\, 1.5,\, 1.1,\, 1.0,\, 1.3,\, 0.8,\, 1.5,\, 1.1\},
\]
and equal weights $\lambda_k = 1/10$. 
The symmetric pairing of the $\alpha$-parameters is the key design choice: each positive $\alpha^k$ is matched with its negative counterpart and the same volatility $\sigma^k$. For constant coefficients,
\[
L(-\alpha)_{t,s}=(-1)^{t-s}L(\alpha)_{t,s}.
\]
Thus, with the common orientations selected in this experiment, odd-lag entries of the adapted barycenter factor cancel, whereas even-lag entries reinforce. For the classical problem, if $D=\diag(1,-1,1,-1,\ldots)$, then $\Sigma(-\alpha)=D\Sigma(\alpha)D$. Uniqueness of the positive-definite Gaussian barycenter and invariance of its objective therefore imply $D\widebar\Sigma_{\BW}D=\widebar\Sigma_{\BW}$. Hence cross-parity entries vanish, but same-parity, even-lag dependence remains. Both barycenters consequently display a checkerboard structure rather than a diagonal one.

\medskip
The adapted barycenter factor, denoted $\widebar L_\mathrm{ABW}$ in this section, is computed via Algorithm~\ref{alg:ABW}, while the classical barycenter $\widebar\Sigma_{\BW}$ is computed via the Bures--Wasserstein iteration \eqref{eq:AgCar_FPE}. Here the common-noise property can be verified exactly. For every $k,\ell$ and $t$,
\[
\bigl(L_{\cdot,t}^k\bigr)^\intercal L_{\cdot,t}^\ell
=\sigma^k\sigma^\ell\sum_{j=0}^{T-t}(\alpha^k\alpha^\ell)^j>0,
\]
because $|\alpha^k\alpha^\ell|<1$. The trace argument in the proof of Proposition~\ref{prop:AR_pos_signs} therefore shows that $\mathbf P_t^\star=\mathbf 1_K\mathbf 1_K^\intercal$ is the unique local optimizer at every time, even though some individual factor entries are negative. Thus the common-noise condition holds and
\[
\widebar L_\mathrm{ABW}=\sum_{k=1}^K\lambda_kL^k
\]
is both a restricted and an unrestricted adapted barycenter factor. The two barycenters are used to generate sample paths from $\widebar X = \widebar L_\mathrm{ABW}\eps$ and $\widebar X_{\BW} = \widebar L_{\BW} \eps$, where $\widebar L_{\BW}$ is the Cholesky factor of $\widebar\Sigma_{\BW}$ and the same noise realization is used within each displayed pair.

\subsection*{Sample paths}
Figures \ref{fig:processes_adapted} and \ref{fig:processes_classical} show barycenter paths (thick lines) together with sample paths from the $10$ input processes (thin coloured lines). Figure \ref{fig:paths_comparison} overlays several paired barycenter paths; within each black/red pair, the same innovation realization is used. The classical paths are often more dispersed over this simulated horizon, but this pathwise observation is realization-dependent and should be interpreted together with the variance comparison below.

\begin{figure}[htbp]
\centering
\begin{subfigure}{0.49\textwidth}
\centering
\includegraphics[width=\linewidth]{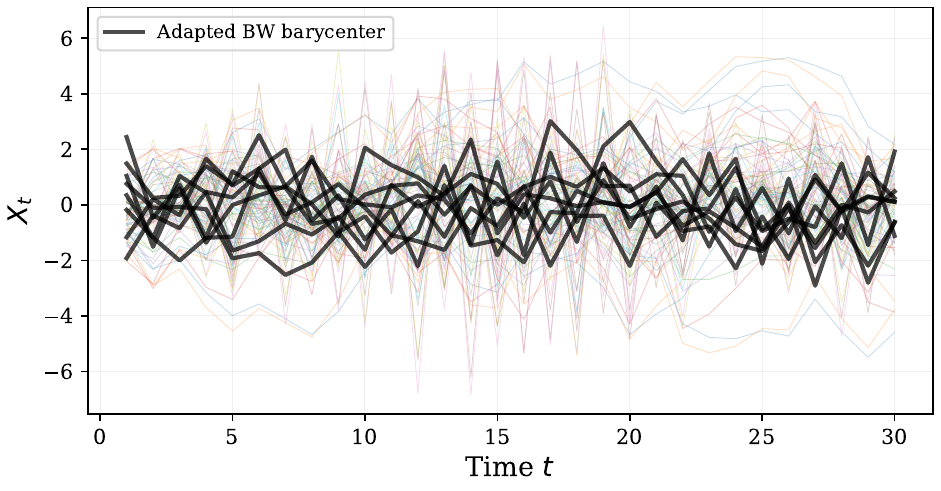}
\caption{Adapted barycenter}
\label{fig:processes_adapted}
\end{subfigure}
\begin{subfigure}{0.49\textwidth}
\centering
\includegraphics[width=\linewidth]{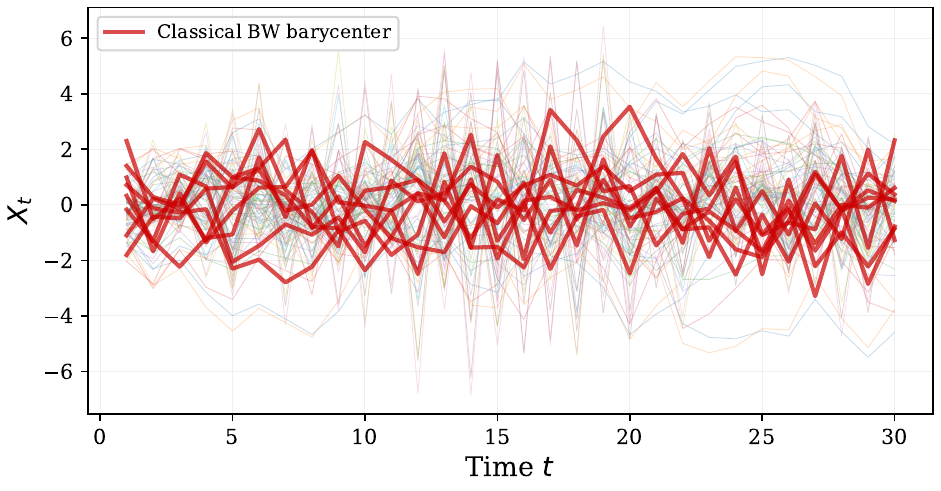}
\caption{Classical barycenter}
\label{fig:processes_classical}
\end{subfigure}
\caption{Sample paths of the input processes (thin, coloured) and the corresponding barycenters (thick).}
\end{figure}
\begin{figure}[htbp]
\centering
\includegraphics[width=0.49\textwidth]{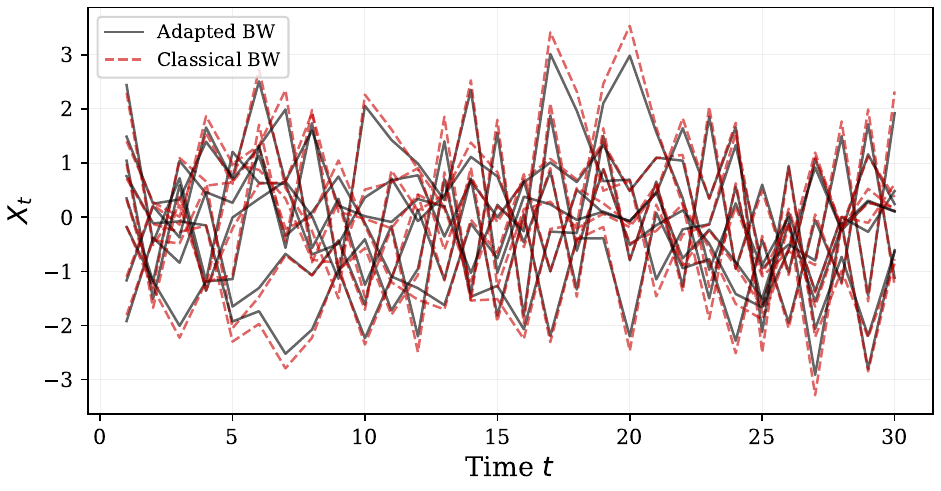}
\caption{Comparison of paired barycenter sample paths. Within each black/red pair, the same innovation realization is used. Solid black: adapted; dashed red: classical.}
\label{fig:paths_comparison}
\end{figure}
\subsection*{Second-order structure}
The key differences between the two barycenters are visible in their second-order statistics. Symmetric sign pairing imposes a parity pattern: odd-lag entries vanish, while even-lag entries remain. Consequently, both $\widebar L_\mathrm{ABW}$ and $\widebar L_{\BW}$ have a checkerboard sparsity pattern, and both barycenter processes retain material even-lag autocovariance. Their differences therefore concern both marginal variances and the magnitude of the surviving same-parity dependence.

\medskip
Figure~\ref{fig:variance} compares the marginal variance $\mathrm{Var}(X_t)$ over time. For the displayed parameters, the adapted variance is larger at the first time point ($1.2996$ versus approximately $1.1529$) and lower from the second time point onward. The growing separation later in the horizon is consistent with cancellation of odd-lag terms in the adapted factor, but it is a feature of this parameter choice rather than a general ordering theorem for adapted and classical barycenter variances.

\medskip
Figure~\ref{fig:covdecay} shows $\Cov(X_1, X_t)$, which measures how the initial state correlates with future values. The parity structure is explicit: odd lags are zero, while the nonzero even-lag values decrease in magnitude with the lag in this experiment. The adapted and classical curves differ on those surviving lags; the figure records an empirical comparison and does not imply a universal covariance ordering.

\begin{figure}[htbp]
\centering
\begin{subfigure}{0.48\textwidth}
\centering
\includegraphics[width=\linewidth]{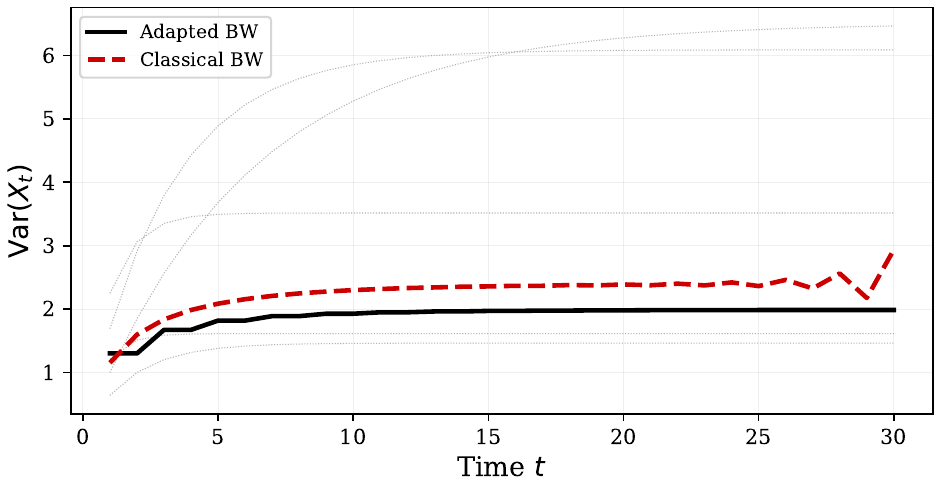}
\caption{Marginal variance $\mathrm{Var}(X_t)$ of the two barycenters.}
\label{fig:variance}
\end{subfigure}
\begin{subfigure}{0.48\textwidth}
\centering
\includegraphics[width=\linewidth]{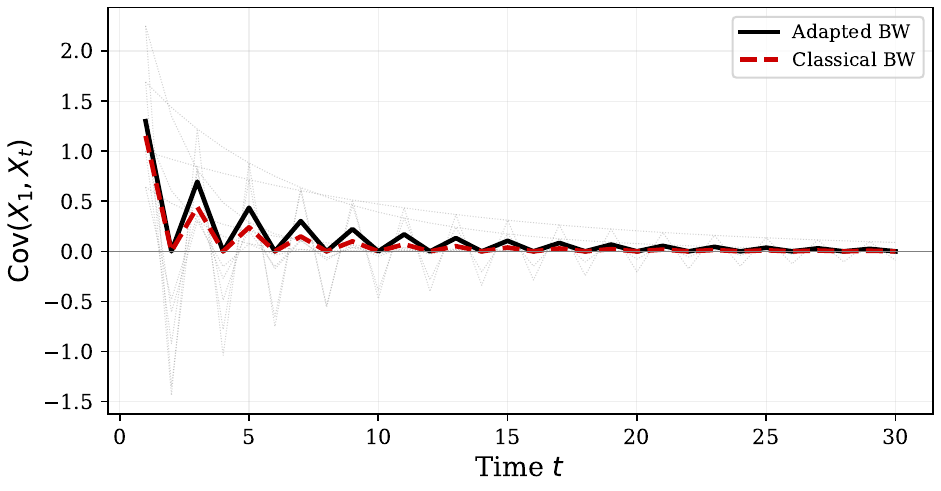}
\caption{Initial-time covariance $\Cov(X_1, X_t)$ of the two barycenters.}
\label{fig:covdecay}
\end{subfigure}
\caption{Second-order statistics of the two barycenters. Grey dotted lines: input processes.}
\end{figure}
\subsection*{Covariance matrices}
Figures \ref{fig:cov_diff} displays the difference between the covariance matrices $\widebar\Sigma_\mathrm{ABW} = \widebar L_\mathrm{ABW}\widebar L_\mathrm{ABW}^\intercal$ and $\widebar\Sigma_{\BW}$ as a heatmap. Both matrices have zero cross-parity entries and visible same-parity bands, producing the predicted checkerboard structure. The classical barycenter has larger diagonal entries through most of the horizon, consistent with Figure \ref{fig:variance}, and it also differs from the adapted barycenter on the surviving even-lag bands. The difference is dominated by the main diagonal but is not confined to it.
\begin{figure}[htbp]
\centering
\includegraphics[width=0.48\textwidth]{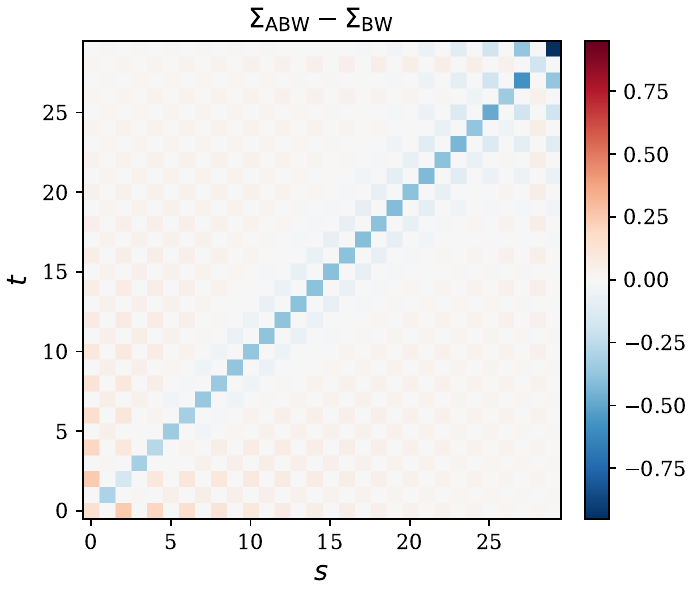}
\caption{Difference of the covariance matrices $\overline\Sigma_\mathrm{ABW} - \overline\Sigma_{\BW}$.}
\label{fig:cov_diff}
\end{figure}

\subsection*{Cholesky factors}
Figure \ref{fig:chol_diff} displays the difference between Cholesky factors $\widebar L_\mathrm{ABW}$ and $\widebar L_{\BW}$ (the lower-triangular Cholesky factor of $\widebar\Sigma_{\BW}$) as a heatmap. Recall that the entry $L_{t,s}$ encodes the influence of the innovation at time $s$ on the process at time $t$. The factors again show exact odd-lag zeros and nonzero even-lag entries. Their largest differences occur on the diagonal and in the early innovation columns, with additional discrepancies along the even subdiagonals.
\begin{figure}[htbp]
\centering
\includegraphics[width=0.48\textwidth]{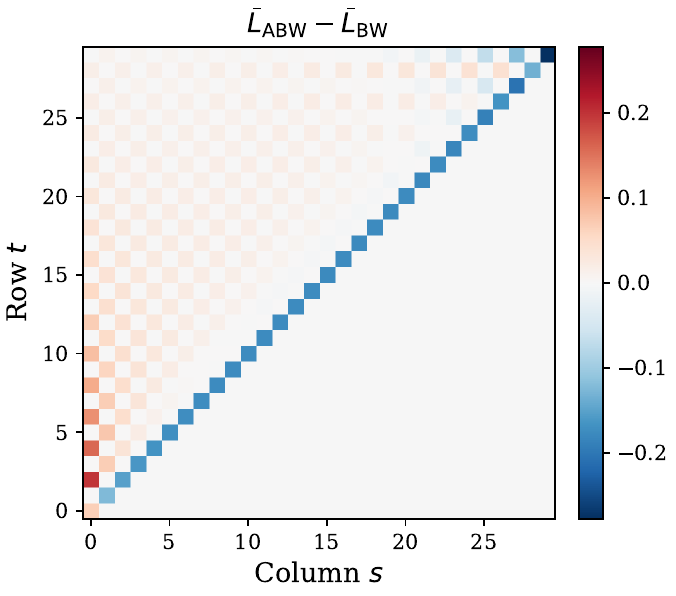}
\caption{Difference of the Cholesky factors $\overline{L}_\mathrm{ABW} - \overline{L}_{\BW}$.}
\label{fig:chol_diff}
\end{figure}

\subsection*{Discussion}
The experiment illustrates two distinct effects. At the factor level, opposite autoregressive signs cancel odd-lag contributions but preserve even-lag contributions. At the covariance level, conjugation symmetry forces the same cross-parity zeros for the classical barycenter, while leaving the magnitudes of the same-parity entries to be determined by Bures--Wasserstein geometry. The two barycenters therefore differ in both variance and surviving temporal dependence. This is consistent with the $K=2$, $T=2$ example of Section~\ref{sec:OU}, where no nonzero even off-diagonal lag is available and $\widebar\Sigma_\mathrm{ABW}=I_2\neq\widebar\Sigma_{\BW}\approx\diag(0.934,1.199)$.

\section{Conclusion}
We established existence of an adapted Wasserstein barycenter for filtered Gaussian inputs and showed that one may select a representative with Gaussian underlying law and enlarged innovations. The unrestricted problem decomposes exactly into classical Bures--Wasserstein barycenter problems for the successive innovation covariance contributions. An ordinary $d$-dimensional filtered Gaussian representative exists precisely when the local optimizers satisfy the common-noise rank criterion; without this property, the restricted filtered Gaussian problem has a strictly larger value in general. We also obtained sufficient uniqueness and regularity conditions, semidefinite optimality certificates, and an explicit solution of the martingale-constrained problem through martingale projection and increment barycenters.

\medskip
The semidefinite formulations provide globally verifiable solutions of the local unrestricted problems. The alternating Procrustes schemes are monotone and yield fixed-point candidates, but their nonconvex restricted formulation does not by itself guarantee convergence to a global minimizer. The autoregressive examples further show that neither ordinary filtered-Gaussian closure nor $AR(1)$ structure is automatic, and that sign symmetry eliminates odd-lag dependence rather than all temporal dependence. Extending these structural results beyond Gaussian inputs and discrete time remains an important direction for future work.

\section*{Acknowledgements}
F.\,M. acknowledges support from the ERC Advanced Grant NEITALG, grant agreement No.~101198055. J.W. acknowledges support from the Novo Nordisk Foundation and the Villum Foundation. T.-K.~L.~W.~acknowledges support from the NSERC Discovery Grant RGPIN-2025-06021.

\begin{center}
  \FundingLogos
  
  \vspace{0.5em}
  \begin{tcolorbox}\centering\small
   
    Funded by the European Union. Views and opinions expressed are however those of the author(s) only and do not necessarily reflect those of the European Union or the European Research Council Executive Agency. Neither the European Union nor the granting authority can be held responsible for them. This project has received funding from the European Research Council (ERC) under the European Union’s Horizon Europe research and innovation programme (grant agreement No. 101198055, project acronym NEITALG).
    
  \end{tcolorbox}
\end{center}

\bibliographystyle{abbrv}
\bibliography{references,transport}

\end{document}